\documentclass[a4paper]{article}
\usepackage[T1]{fontenc}
\usepackage[english]{babel}
\usepackage[numbers]{natbib}
\usepackage[margin=1in]{geometry}
\usepackage{amsmath}
\usepackage{yhmath}
\usepackage{amsthm}
\usepackage{amssymb}
\usepackage{dsfont}
\usepackage{enumitem}
\usepackage[noblocks]{authblk}
\usepackage[hidelinks]{hyperref} 
\usepackage{mathtools}

\usepackage{chngcntr}
\usepackage{apptools}
\AtAppendix{\counterwithin{lemma}{section}}

\usepackage{setspace}
\theoremstyle{plain}
\newtheorem{theorem}{Theorem}[section]
\newtheorem{lemma}{Lemma}[section]
\newtheorem{assumption}{Assumption}
\newtheorem{prop}{Proposition}[section]
\newtheorem{corol}[prop]{Corollary}%
\theoremstyle{definition}
\newtheorem{definition}{Definition}[section]

\theoremstyle{remark}
\newtheorem{remark}{Remark}[section]

\author{ 
Jean-François Chassagneux\textsuperscript{1}, Stéphane Crépey\textsuperscript{2}, Grégoire Loeper\textsuperscript{3}, Justin Ruelland\textsuperscript{2,3}
}

\date{ 
\textsuperscript{1}ENSAE-CREST, Institut Polytechnique de Paris, Palaiseau, France \\
\textsuperscript{2}Laboratoire de Probabilités, Statistique et Modélisation (UMR 8001), Paris, France \\
\textsuperscript{3}BNP Paribas Global Markets, Quantitative Research, Paris, France
}
\title{Differentiability and Convexity of the Solution Maps for Fully Non-linear Parabolic HJB Equations\thanks{This research has benefited from the support of \textit{Chair Futures of Quantitative Finance}, a partnership between LPSM at Université Paris Cité, CERMICS at École nationale des ponts et chaussées, and BNP Paribas Global Markets.}}

\DeclareMathOperator{\Hessian}{Hess}

\newcommand*\interior[1]{\mathrm{int}{(#1)}}

\newcommand{\wtilde}{\widetilde}

\DeclareMathOperator*{\dom}{dom\,}

\DeclareMathOperator{\spn}{span}
\DeclareMathOperator{\BUC}{BUC}

\newcommand{\demi}{\frac{1}{2}}

\newcommand{\R}{\mathbb R}
\newcommand{\N}{\mathbb N}
\newcommand{\E}{\mathbb E}

\newcommand{\cF}{\mathcal F}
\newcommand{\cL}{\mathcal L}
\newcommand{\cP}{\mathcal P}

\newcommand{\dt}{\partial_t}

\newcommand{\Dx}{\nabla_x}
\newcommand{\dx}{\partial_x}
\newcommand{\Dxx}{\nabla^2_x}
\newcommand{\dxx}{\partial^2_{xx}}

\newcommand{\veps}{\varepsilon}

\newcommand{\abs}[1]{\lvert #1 \rvert}

\newcommand{\esperance}[1]{\mathbb{E}\left[ #1 \right]}
\newcommand{\Norm}[1]{\left\| #1 \right\|}

\begin{document}

\maketitle
\begin{abstract} We study the dependence on the terminal condition of solutions to a fully nonlinear parabolic Hamilton--Jacobi--Bellman (HJB) equation, when the terminal condition varies along a finite-di\-men\-sion\-al family. This problem is motivated by a dual formulation of semimartingale optimal transport. The equation is posed as a terminal-value Cauchy problem on an unbounded spatial domain. Under structural ellipticity and regularity assumptions on the nonlinear operator, we recall local parabolic Hölder estimates and stability results for bounded uniformly continuous terminal data. We then establish continuous directional differentiability of the solution map with respect to the finite-di\-men\-sion\-al parameter, and identify its derivative through the corresponding linearized Cauchy problem. Under an additional convexity assumption on the Hamiltonian and a non-degeneracy condition on the terminal data, we prove local strong convexity of the point-evaluated solution map. The analysis relies on Evans--Krylov theory for fully nonlinear parabolic equations and combines PDE arguments with probabilistic representations. Finally, we extend the differentiability and convexity results to cascades of Cauchy problems in which successive equations are coupled through terminal conditions at several prescribed times. This structure is motivated, for instance, by optimal transport problems with marginal constraints imposed at several maturities.
\end{abstract}

\section{Introduction}

We study the following terminal-value Cauchy problem. Let $T>0$, $M\in\N_*$, $d\in\N_*$, and let $G:\R^d\to\R^M$ be a vector of bounded uniformly continuous functions. For each $\lambda\in\R^M$, we consider the fully nonlinear parabolic Hamilton--Jacobi--Bellman (HJB) equation
\begin{equation}\label{eq:hjb-simple}
    \left\{ \begin{array}{l} 
        \dt \phi^\lambda + H(t,x,\Dx \phi^\lambda, \demi \Dxx \phi^\lambda) = 0 \text{ on } (0,T)\times\R^d, \\ 
        \phi^\lambda(T,\cdot) = \sum_{m=1}^M \lambda_m G_m(\cdot).
    \end{array} \right. 
\end{equation} 

This problem arises, for example, in the dual formulation of calibration problems in semimartingale optimal transport (SMOT); see, for instance, \cite{SurveyOT,guo2022lsv-calibration}. The goal of the SMOT calibration problem is to find an $\R^d$-valued diffusion process $(X_t)_{t\in[0,T]}$, with $X_0=x_0$, driven by a Brownian motion $(W_t)_{t\in[0,T]}$, satisfying \begin{equation*}
    dX_t = b(t,X_t)dt + \sigma(t,X_t) dW_t.
\end{equation*} The drift $b$ and the covariance matrix $a=\sigma\sigma^T$ are calibrated so as to match a finite number of moment constraints of the form $\E[G_m(X_T)]=\pi_m$, $m\leq M$. In financial calibration problems, these constraints typically encode market prices of options. To this end, one considers the constrained minimization problem
\begin{equation*}
   \inf_{a,b} \E \left[ \int_0^T C(t,X_t,b(t,X_t),a(t,X_t)) dt \right],
\end{equation*}
 where $C$ is a convex running cost and the infimum is taken over admissible Markov controls for which the corresponding diffusion satisfies
\[
    \E[G_m(X_T)] = \pi_m, \qquad m=1,\dots,M.
\] The Hamiltonian $H$ appearing in~\eqref{eq:hjb-simple} is then obtained as the Fenchel--Legendre transform of $C$ with respect to the control variables. The dual formulation of this problem is given by \begin{equation*}
    \sup_{\lambda \in \R^M} \left\{ \sum_{m=1}^M \lambda_m \pi_m - \phi^\lambda(0,x_0) \right\},
\end{equation*} where $\phi^\lambda$ is the solution of~\eqref{eq:hjb-simple}. In this context, the differentiability and convexity properties of the map $\lambda \mapsto \phi^\lambda(0,x_0)$ are crucial for studying the well-posedness of the dual problem. \\

The main objective of the paper is therefore to study differentiability and convexity properties of the solution map
\begin{equation*}
\R^M \ni \lambda \mapsto \phi^\lambda(0,x_0) \in \R,
\end{equation*}
for $x_0\in \R^d$. Under suitable regularity and ellipticity assumptions on the nonlinear operator $H$, we prove continuous differentiability of this map and identify its gradient through the corresponding linearized Cauchy problem. Under additional convexity and non-degeneracy assumptions, we prove its local strong convexity. Throughout the paper, solutions are understood in the classical sense in the parabolic interior: they solve the equation pointwise, belong locally to the parabolic \(C^2\) class specified below, and are controlled by interior Hölder estimates. Under our structural assumptions on $H$ and the regularity assumptions on the terminal data, the relevant viscosity solutions are regular and coincide with the classical solutions considered below. We therefore formulate the analysis directly in the classical parabolic framework.\\ 

We also extend these results to a cascade of solutions, a structure motivated by multi-marginal constraints in SMOT calibration. We consider a finite number $P\in\N_*$ of terminal times $(\tau_p)_{p\leq P}$ such that $0=\tau_0<\tau_1<\dots<\tau_P =T$.  We introduce the notation $i(p,m)=(p-1)M+m$. Let $G\in\mathrm{BUC}(\R^d;\R^{P M})$ be a vector of bounded uniformly continuous functions, and $\lambda \in \R^{P M}$. Formally, the associated Cauchy problem can be written as \begin{equation} \label{eq:hjb-mult}
    \left\{ \begin{array}{l} \displaystyle
        \dt \phi^\lambda + H(t,x,\Dx \phi ^\lambda, \demi \Dxx \phi^\lambda) + \sum_{p=1}^{P}\;\sum_{m=1}^{M} \lambda_{i(p,m)} G_{i(p,m)} \delta_{\tau_p} = 0 \text{ on } (0,T)\times\R^d,\\ 
         \phi^\lambda(T,\cdot) \equiv 0.
    \end{array} \right.
\end{equation}

\noindent \textbf{Outline of the paper.} The remainder of this paper is organised as follows. In Section~\ref{sec:setting-notation}, we introduce the notation and the precise assumptions on the nonlinear operator used throughout the paper. In Section~\ref{sec:estimates}, we establish well-posedness and a priori estimates for~\eqref{eq:hjb-simple}. In Sections~\ref{sec:differentiability} and \ref{sec:convexity}, we study the differentiability and convexity of the solution map associated with \eqref{eq:hjb-simple}. In Section~\ref{sec:gluing}, we extend these results to the solution map of the cascade equation~\eqref{eq:hjb-mult}.

\section{Notation and Assumptions}\label{sec:setting-notation}

\subsection{Notation}
Let \(E\subseteq \R^d\) be a non-empty set. We denote by \(C_b(E)\) the space of bounded continuous functions on \(E\), where continuity is understood with respect to the topology induced by \(\R^d\). We denote by \(\BUC(E)\) the subspace of bounded uniformly continuous functions. These spaces are endowed with the corresponding uniform norm on \(E\), $\Norm{f}_{\infty;E}$. The notation $K\Subset E$ means that $K$ is compact and contained in $E$; if $E$ is open, this is understood as compact containment. When $E=\R^d$, we simply write $\Norm{f}_\infty$. For $t\in(0,T]$ and for a domain $\Omega \subseteq \R^d$, we write $Q^t_\Omega = (0,t)\times\Omega$, in particular when $\Omega=\R^d$ we simply write $Q^t$. When needed, for any normed vector space \(F\), we use the notation \(X(E;F)\) for the \(F\)-valued version of a functional space \(X(E)\) introduced below. For instance, \(C_b(E;F)\) and \(\BUC(E;F)\) denote bounded continuous and bounded uniformly continuous functions from \(E\) to \(F\).  We also denote by $\overline{Q}_\Omega^t = [0,t]\times\overline{\Omega}$ the closure. For \(\omega=(t,x)\) and \(\omega'=(s,y)\) in a parabolic cylinder, we use the parabolic distance
\begin{equation*}
d_\cP(\omega,\omega')
:=
|t-s|^{1/2}\vee |x-y|.
\end{equation*}
For a function \(f\) defined on a parabolic domain \(Q\), we define the parabolic Hölder seminorm
\begin{equation*}
[f]_{\alpha;Q}
:=
\sup_{\substack{\omega,\omega'\in Q\\ \omega\neq \omega'}}
\frac{|f(\omega)-f(\omega')|}{d_\cP(\omega,\omega')^\alpha},
\qquad \alpha\in(0,1],
\end{equation*}
and the corresponding Hölder norm
\begin{equation*}
\Norm{f}_{\alpha;Q}
:=
\Norm{f}_{\infty;Q}+[f]_{\alpha;Q}.
\end{equation*}
The case \(\alpha=1\) corresponds to Lipschitz continuity with respect to the parabolic distance. We use the standard parabolic convention that one derivative in time has the same order as two derivatives in space. Thus, for a multi-index \(\beta\in\mathbb{N}^d\) and \(r\in\mathbb{N}\), the parabolic order of \(\partial_t^rD_x^\beta\) is \(2r+|\beta|\). For an open parabolic domain \(Q\), we define
\begin{equation*}
C^2(Q)
:=
\left\{
u:Q\to\R
:
\partial_t^rD_x^\beta u \text{ exists and is continuous on } Q
\text{ for all } 2r+|\beta|\le 2
\right\}.
\end{equation*}
We endow it with the norm
\begin{equation*}
\Norm{u}_{2;Q}
:=
\sum_{2r+|\beta|\le 2}
\Norm{\partial_t^rD_x^\beta u}_{\infty;Q}.
\end{equation*}
In particular, when \(Q\) is unbounded, the condition \(\Norm{u}_{2;Q}<\infty\) imposes boundedness of \(u\), \(D_xu\), \(D_x^2u\), and \(\partial_tu\) on \(Q\).

If \(K\subset Q\) is compact, the derivatives in \(\Norm{u}_{2;K}\) are always understood as derivatives taken in the ambient open set \(Q\), then restricted to \(K\). If \(Q=Q_\Omega^T\) and norms are evaluated on \(\overline Q\), derivatives are understood in \(Q\) and continuously extended to \(\overline Q\). Accordingly, we write \(u\in C^2(\overline Q)\) if \(u\in C^2(Q)\), all derivatives \(\partial_t^rD_x^\beta u\), \(2r+|\beta|\le2\), admit continuous extensions to \(\overline Q\), and
\begin{equation*}
\Norm{u}_{2;\overline Q}
:=
\sum_{2r+|\beta|\le 2}
\Norm{\partial_t^rD_x^\beta u}_{\infty;\overline Q}
<\infty.
\end{equation*} We also use the local version
\begin{equation*}
C^2_{\mathrm{loc}}(Q)
:=
\left\{
u\in C^2(Q)
:
\Norm{u}_{2;K}<\infty
\text{ for every compact }K\subset Q
\right\}.
\end{equation*}
Thus \(C^2_{\mathrm{loc}}(Q)\) does not require any global bound on \(Q\).

For \(\alpha\in(0,1]\), we define the parabolic Hölder space \(C^{2+\alpha}(Q)\) as the set of functions \(u\in C^2(Q)\) such that
\begin{equation*}
\Norm{u}_{2+\alpha;Q}
:=
\Norm{u}_{2;Q}
+
\sum_{2r+|\beta|=2}
[\partial_t^rD_x^\beta u]_{\alpha;Q}
<\infty.
\end{equation*}

We define \(C^{2+\alpha}(\overline Q)\), for \(\alpha\in(0,1]\), in the analogous way, with derivatives taken in \(Q\) and continuously extended to \(\overline Q\), and with the Hölder seminorm computed on \(\overline Q\). Finally,
\begin{equation*}
C^{2+\alpha}_{\mathrm{loc}}(Q)
:=
\left\{
u\in C^{2+\alpha}(Q)
:
\Norm{u}_{2+\alpha;K}<\infty
\text{ for every compact }K\subset Q
\right\},
\qquad \alpha\in(0,1].
\end{equation*}

When \(Q\) is a bounded parabolic domain, we denote by \(\mathcal{P}Q\) its parabolic boundary. For $Q=(0,T)\times\Omega$, $\cP Q = \{T\} \times \overline{\Omega} \cup (0,T)\times \partial \Omega$ (for $\Omega = \R^d$, $\partial\Omega = \emptyset$).  For \(\omega\in Q\), we write
\begin{equation*}
d_Q(\omega)
:=
\inf_{\zeta\in \mathcal{P}Q} d_\cP(\omega,\zeta)
\end{equation*}
for the parabolic distance from \(\omega\) to the parabolic boundary. Following the notation of \cite{lieberman}, for \(f\in C^\alpha(Q)\), \(\alpha\in(0,1]\), we introduce the weighted Hölder norm
\begin{equation*}
\Norm{f}_{\alpha;Q}^{(2)}
:=
\sup_{\omega\in Q}
d_Q(\omega)^2 |f(\omega)|
+
\sup_{\substack{\omega,\omega'\in Q\\ \omega\neq \omega'}}
\left(d_Q(\omega)\wedge d_Q(\omega')\right)^{2+\alpha}
\frac{|f(\omega)-f(\omega')|}{d_\cP(\omega,\omega')^\alpha}.
\end{equation*}

\begin{remark}\label{rk:holder-weighted}
    If $f\in C^\alpha(\overline{Q})$, $\alpha \in (0,1]$, and $Q_1 \subset Q$ such that $\mathrm{diam}(Q_1)<1$, then \begin{equation*}
        \Norm{f}_{\alpha;Q_1}^{(2)}=O(\mathrm{diam}(Q_1)^2 ).
    \end{equation*}
\end{remark}

We conclude by recalling a standard compactness property of Hölder spaces that will be used throughout the paper. 

\begin{lemma}\label{lem:holder-compact-embedding}
Let \(d\in\mathbb N_*\), \(n\in\mathbb N\), and let \(0<\beta<\alpha\leq 1\).
Then, for every compact set \(K\subset \mathbb R^d\), the embedding
\[
C^{n+\alpha}(K) \hookrightarrow C^{n+\beta}(K)
\]
is compact. Moreover, for every \(R>0\), the closure of
\[
B_R^{n+\alpha}
:=
\left\{
u\in C^{n+\alpha}(\mathbb R^d)
:
\|u\|_{C^{n+\alpha}(\mathbb R^d)}\le R
\right\}
\]
in \(C^{n+\beta}(\mathbb R^d)\) is $B^{n+\alpha}_R$ i.e. \begin{equation*}
    \overline{B_R^{n+\alpha}}^{\,C^{n+\beta}(\mathbb R^d)}
\subseteq B_R^{n+\alpha}.
\end{equation*}
\end{lemma}

\subsection{Assumptions on the non-linear operator}

Our results are based on the Evans--Krylov theory. We therefore use the structural conditions introduced by Krylov in \cite{Kry82} for boundedly non-homogeneous parabolic equations. We state here the assumptions made on the non-linear operator $H$ ensuring well-posedness of the Cauchy problem~\eqref{eq:hjb-simple}.  

\begin{definition}[Class of operators $\cF(\mathfrak{v},Q^T)$]\label{def:F_nu_omega}
    Let $\mathfrak{v} \in (0,1]$ and let $Q^T$ be a parabolic domain. We denote by $\cF(\mathfrak{v},Q^T)$ the class of functions $\mathcal{G}:Q^T\times \R^*_+ \times \R \times \R^d \times \R^{d\times d} \ni (t,x,\beta,r,p,M) \mapsto \mathcal{G}(t,x,\beta,r,p,M) \in \R$ such that: \begin{enumerate}[label=(\roman*)]
        \item $\mathcal{G}$ is positively homogeneous of degree one with respect to $(\beta,r,p,M)$, i.e. for every $s>0$,
\[
    \mathcal{G}(t,x,s\beta,sr,sp,sM)
    =
    s\,\mathcal{G}(t,x,\beta,r,p,M),
\]
for all $(t,x,\beta,r,p,M)\in Q^T\times\R_+^*\times\R\times\R^d\times\R^{d\times d}$,
        \item $\mathcal{G} \in C(Q^T\times \R^*_+ \times \R \times \R^d \times \R^{d\times d}) $  and $\forall t \in [0,T],$ $(x,\beta,r,p,M)\mapsto \mathcal{G}(t,x,\beta,r,p,M)$ is twice continuously differentiable on $\R^d\times \R_+^*\times \R \times \R^d \times\R^{d\times d}$,
        \item $\mathcal{G}$ is $\mathfrak{v}$-elliptic i.e. $\forall z \in \R^d$, $\mathfrak{v} ||z||^2 \leq z^T D_M \mathcal{G} z $,
        \item For every $(t,x,\beta,r,p,M)\in
        Q^T\times\R_+^*\times\R\times\R^d\times\R^{d\times d}$ and every
        $(\widetilde r,\widetilde p,\widetilde M)\in
        \R\times\R^d\times\R^{d\times d}$,
        \[
        D^2_{(r,p,M)}\mathcal{G}(\wtilde{r},\wtilde{p},\wtilde{M})
        \geq
        -\mathfrak{v}^{-1}\beta^{-1}
        \left(\|\widetilde r\|^2+\|\widetilde p\|^2\right),
        \]
        \item For every $(t,x,\beta,r,p,M)\in
Q^T\times\R_+^*\times\R\times\R^d\times\R^{d\times d}$ and all
$i,j,k\in\{1,\dots,d\}$, the norms of $D_{M_{ij}} \mathcal{G}, D_{p_i}\mathcal{G}, D_{\beta}\mathcal{G}, D^2_{(M_{ij},x_k)}\mathcal{G},D^2_{(r,x_k)}\mathcal{G}, D^2_{(\beta,x_k)}\mathcal{G}$ are smaller than $\mathfrak{v}^{-1}$, and the norm of $D^2_{(x_i,x_j)}\mathcal{G}$ is smaller than $\mathfrak{v}^{-1} w$ with $w=\sqrt{\beta^2 + r^2+\Norm{p}^2+\Norm{M}^2}$,
        \item $\mathcal{G}$ is differentiable in $t$ and $|\partial_t \mathcal{G}| \leq \mathfrak{v}^{-1} w$.
    \end{enumerate}
\end{definition}

\begin{remark}\label{rk:cvx-assmpt-hn}
    The condition (iv) implies that $\mathcal{G}$ is convex in $M$. 
\end{remark}

\begin{definition}[Non-linear operator class]\label{def:Bellman-operator}
    The operator $H:Q^T\times \R\times \R^d \times \R^{d\times d}   \ni (t,x,r,p,M) \mapsto H(t,x,r,p,M) \in \R $ is said to be boundedly non-homogeneous on $Q^T$ if $\exists \mathfrak{v} \in (0,1],\, (H^n)_{n\in \N} \in \cF(\mathfrak{v}, Q^T),$\begin{align}
        H(t,x,r,p,M) = \sup_{n} H^n(t,x,1,r,p,M).
    \end{align} for a $\mathfrak{v} >0$, $\forall n \in \N,\, H^n \in \cF(\mathfrak{v}, Q^T)$. We denote this class of operators by $\mathcal{H}(\mathfrak{v},Q^T)$.
\end{definition}

\begin{remark}
    In \cite{Kry82} the results apply to PDEs of the form \begin{equation*}
    \dt u = \wtilde{H}(t,x,u,\Dx u, \Dxx u), \quad \text{on } Q^T_\Omega,
\end{equation*} with prescribed data on $t=0$. Our setup is reversed in time and is of the form~\eqref{eq:hjb-simple}, with prescribed data at $t=T$. We therefore use the change of variables $H(t,x,p,M)=-\wtilde{H}(T-t,x,-p,-M)$. This is why we consider mainly a convexity assumption rather than a concavity assumption in $\cF(\mathfrak{v},Q^T)$ and use a "$\sup$" rather than a "$\inf$" in $\mathcal{H}(\mathfrak{v},Q^T)$, in contrast to the framework in Krylov's article~\cite{Kry82}.
\end{remark}

In addition to the structural assumptions of Krylov’s class $\mathcal{H}(\mathfrak{v},Q^T)$, we assume that $H$ is twice differentiable with respect to $(p,M)\in\mathbb R^d\times\mathbb S^d$, and that its second derivative $D^2_{(p,M)}H$ is globally Lipschitz with respect to $(t,x,p,M)$, where $Q^T$ is endowed with the parabolic metric. Namely, there exists $L>0$ such that, for all $(t,x,p,M),(s,y,p',M') \in Q^T\times\R^d\times\mathcal{S}^d$\begin{equation}\label{eq:assumption-h-1}
    \left| D^2_{(p,M)} H(t,x,p,M) -D^2_{(p,M)} H(s,y,p',M') \right| \leq L \,(d_\cP((t,x),(s,y))+|p-p'|+|M-M'|).
\end{equation} We assume that $D^2_{(p,M)} H$ is globally bounded, there exists $L_1>0$ \begin{equation}\label{eq:assumption-h-2}
    \sup_{(t,x,p,M) }
\left|D_{(p,M)}^2H(t,x,p,M)\right|<L_1.
\end{equation} We also assume that $(t,x) \mapsto D_{(p,M)} H(t,x,0_d,0_{d\times d})$ is globally Lipschitz for the parabolic metric, namely there exists $L_2>0$ such that for all $(t,x),(s,y) \in Q^T$, \begin{equation}\label{eq:assumption-h-3}
    \left| D_{(p,M)}H(t,x,0_d,0_{d\times d})-D_{(p,M)}H(s,y,0_d,0_{d\times d})  \right| \leq L_2 d_{\cP}((t,x),(s,y)).
\end{equation}

\begin{remark}\label{rk:lipschitz-dH}
    Consequently, if $u,v \in C^{\alpha}(Q^T)$ with $\alpha \in (0,1]$, then $D_{(p,M)} H(t,x,u(t,x),v(t,x)) \in C^{\alpha}(Q^T;\R^d \times \R^{d\times d})$ and $D^2_{(p,M)} H(t,x,u(t,x),v(t,x)) \in C^{\alpha}(Q^T;\R^d \times \R^{d\times d})$.
\end{remark}

\begin{assumption}\label{assmpt:operatorH}
    We suppose that there exists $\mathfrak{v}\in(0,1]$, such that $H \in \mathcal{H}(\mathfrak{v},Q^T)$ and $H$ satisfies~\eqref{eq:assumption-h-1}–\eqref{eq:assumption-h-3} for the constant $\mathfrak{v}^{-1}$. Moreover, we suppose that $\Norm{H(t,\cdot,0_d,0_{d\times d })}_\infty < \mathfrak{v}^{-1}, $ for all $t\in[0,T]$.
\end{assumption}

In the following, we denote by
\[
    \Hessian{H}(t,x,p,M) := D^2_{(p,M)}H(t,x,p,M)
\]
the Hessian of $H$ with respect to its last two variables, viewed as a
symmetric bilinear form on $\mathbb R^d\times\mathcal S^d$.

The structural assumptions above imply convexity of $H$ in $M$.
For the local strong convexity result, we additionally impose joint
convexity in $(p,M)$ and positive definiteness of the Hessian near
the origin.

\begin{assumption}[Convexity and local non-degeneracy of the Hamiltonian]
\label{assmpt:operatorH-cvx}
For every $(t,x)\in Q^T$, the map
$(p,M)\mapsto H(t,x,p,M)$ is convex on
$\R^d\times\mathcal S^d$. Equivalently,
\begin{equation*}
    \Hessian{H}(t,x,p,M)\succeq0
    \qquad
    \text{for every }(t,x,p,M)
    \in Q^T\times\R^d\times\mathcal S^d.
\end{equation*}
Moreover, there exists $\Delta>0$ such that
\begin{equation*}
    \Hessian{H}(t,x,p,M)\succ0
    \qquad
    \text{for every }(t,x)\in Q^T
    \text{ and } |(p,M)|<\Delta.
\end{equation*}
\end{assumption}

Joint convexity ensures that the quadratic source term in the
second linearized equation is nonnegative. The local positive
definiteness condition is used, together with the non-degeneracy
of the terminal data, to obtain a strictly positive contribution
to the second variation.

\noindent To simplify notation, when no ambiguity arises we write $H(p,M):= H(t,x,p,M)$.

\noindent \textbf{Example.} We give an example of a non-linear operator that satisfies Assumptions~\ref{assmpt:operatorH} and \ref{assmpt:operatorH-cvx}. This example is motivated by the nature of the equation~\eqref{eq:hjb-simple} (and \eqref{eq:hjb-mult}) in the optimal transport calibration problem. In this setting, the non-linear operator $H$ is the Fenchel--Legendre transform of the cost function of the problem. Consider the following ``cost'' function in spatial dimension $d=1$, \begin{equation*}
    C(t,x,b,a)=
    \begin{cases}
        \displaystyle
        \frac12-\sqrt{\frac14-(b-1)^2-(a-1)^2},
        & (b,a)\in\overline{\Gamma},\\
        +\infty, & \text{otherwise}.
    \end{cases}
\end{equation*}
where $\Gamma := \interior{\dom{C}} = \mathcal{B}((1,1), \demi)$. Then $C$ is a proper closed convex function. Defining $H=C^*$ as the convex conjugate of $C$ (in its last two variables), we have \begin{equation*}
    H:Q^T \times \R\times \R \to \R, \quad  (t,x,p,M) \mapsto p + M  + \demi \left( \sqrt{1+p^2+M^2}-1 \right).
\end{equation*} Note that by definition, we have $H(p,M)=\sup_{(b,a)}H^{(b,a)}(p,M)$ (we omit the space-time dependence) where \begin{equation*}
    H^{(b,a)}(p,M) = b\cdot p + a:M - C(b,a)
\end{equation*} so that $ H^{(b,a)}$ is affine in $(p,M)$. By Lemma~\ref{lem:countable-infimum} below, the supremum can be restricted to a countable set. Thus, we define $H^n(t,x,\beta,p,M)= b_n\cdot p + a_n:M - \beta C(b_n,a_n)$. We fix $\mathfrak{v} \in (0,1]$ by setting
\begin{equation*}
    \mathfrak{v}^{-1} = \max\left\{ \overline{\beta}, \underline{\beta}^{-1}, \left|\overline{\alpha}\right|, \left|\underline{\alpha} \right|, {\sup |c|},{\sup |\dt c|},{\sup |\Dx c|},{\sup |\Dxx c|} \right\}.
\end{equation*} Hence, $H \in \mathcal{H}(\mathfrak{v},Q^T)$. 
The remaining assumptions are then immediate from the explicit formula for \(H\). Since \(H\) is independent of \((t,x)\), \eqref{eq:assumption-h-3} holds with constant \(0\). Its Hessian with respect to \((p,M)\) is
\[
D^2_{(p,M)}H(p,M)
=
\frac12\left(
\frac{I_2}{\sqrt{1+p^2+M^2}}
-
\frac{(p,M)^\top (p,M)}{(1+p^2+M^2)^{3/2}}
\right),
\]
which is globally bounded and Lipschitz on \(\mathbb R^2\). Hence \eqref{eq:assumption-h-1} and \eqref{eq:assumption-h-2} hold, up to increasing \(\mathfrak v^{-1}\), while \(H(0,0)=0\) gives
\[
    \|H(t,\cdot,0,0)\|_\infty=0<\mathfrak v^{-1}.
\]
Moreover, for every $z=(p,M)\in\R^2$ and every $\xi\in\R^2$,
\begin{equation*}
    \xi^\top \Hessian{H}(p,M)\xi
    \geq
    \frac{|\xi|^2}{2(1+|z|^2)^{3/2}}.
\end{equation*}
Thus $\Hessian{H}(p,M)$ is positive definite for every
$(p,M)\in\R^2$, and Assumption~\ref{assmpt:operatorH-cvx}
is satisfied.

\begin{lemma}\label{lem:countable-infimum}
    Let $S$ be a separable metric space, Z a topological space. We consider $\forall s \in S$, the family of functions $f^s : Z \to \R$. If $\forall z \in Z,$ the map $s\mapsto f^s(z)$ is continuous on $S$, then $\exists S_0$ countable subset of $S$, \begin{align}
        \inf_S f^s = \inf_{S_0}f^s \text{ on } Z.
    \end{align}
\end{lemma}

\begin{proof}%
    By separability of $S$, $\exists S_0 \subset S$ that is countable and dense in $S$. We denote $F= \inf_S f^s$, let $z \in Z$, $\forall \veps > 0$, $\exists s \in S$, $f^s(z) < F(z) + \frac{\veps}{2}$. By density $\exists (s_n)_{n \in \N} \in S_0$, $s_n \xrightarrow[n\infty]{}s$, by continuity we have $\exists N \in \N_*, \, \forall n > N, f^{s_n}(z) < f^s(z) + \frac{\veps}{2}$ so $\inf_{S_0} f^s(z) < F(z) + \veps$.
\end{proof}

For the sake of clarity, we work throughout the paper in the one-dimensional setting, $d=1$. All the results presented below remain valid in arbitrary dimension $d\geq1$.

\section{A priori estimates and stability of solutions}\label{sec:estimates}

In this section, we present various results on fully nonlinear second-order parabolic equations in Hölder spaces: a priori estimates, stability of solutions, and linearization. These results are classical and rely primarily on the Evans--Krylov theorem. Our arguments are based on the theory developed in \cite{Kry82}. In addition, we use tools such as interior Schauder estimates for linear parabolic equations, for which we refer to \cite{Kry96_holder,lieberman}. All proofs are postponed to Appendix~\ref{sec:appendix-estimates}, since they follow from well-established arguments. \\

Throughout the section, we assume that $H$ satisfies Assumption~\ref{assmpt:operatorH}, with a fixed parameter   $\mathfrak{v}\in(0,1]$ for the operator. For $\varphi \in C_b(\R)$, we consider the following nonlinear Cauchy problem, \begin{equation}
    \left\{\begin{array}{ll}\label{eq:nonlinear-cauchy-problem}
         \dt \phi + H(t,x,\dx\phi(t,x),\demi\dxx\phi(t,x)) = 0 \text{ in } Q^T,  \\
         \phi(T,\cdot) = \varphi (\cdot).
    \end{array}\right. 
\end{equation} We begin by recalling a comparison principle for this class of non-linear equations on an unbounded domain. 

\begin{prop}\label{prop:comp-principle-nonlinear}
    Consider $u,v \in C^{2}_{loc}(Q^T)\cap C_b(\overline{Q}^T)$. Assume that $u$ and $v$ are respectively a subsolution and a supersolution to \begin{equation}\label{eq:comp-prcp-nn}
    \dt w + H(\dx w,\demi \dxx w)=0 \text{ on } Q^T \, .\end{equation} If $u(T,\cdot)\leq v(T,\cdot)$, then $ u \leq v \text{ on } Q^T$. Moreover, if $u$ and $v$ are solutions of~\eqref{eq:comp-prcp-nn}, then \begin{equation*}\Norm{u-v}_{\infty;\overline{Q}^T}\leq \Norm{u(T,\cdot)-v(T,\cdot)}_{\infty}.\end{equation*}
\end{prop}

As a consequence of the previous comparison principle, we obtain the following maximum principle under a suitable integrability assumption on the non-linear operator.

\begin{corol}\label{coro:max-principle-nonlinear}
    Let $\phi \in C^{2}_{loc}(Q^T)\cap C_b(\overline{Q}^T)$ be a solution to~\eqref{eq:nonlinear-cauchy-problem}. Assume that \begin{align}\label{ineq:hyp-fstar}
        \int_0^T \Norm{H(t,\cdot,0,0)}_{\infty}\, dt < \infty.
    \end{align} Then, \begin{align}\label{ineq:maxprinc-nonlin}
        \Norm{\phi}_\infty \leq \Norm{\varphi}_\infty + \int_0^T \Norm{H(t,\cdot,0,0)}_{\infty}\, dt.
    \end{align}
\end{corol}

\begin{remark}
    The above bound is consistent with the probabilistic representation 
    of the solution as the value function of a stochastic control problem. 
\end{remark}

\begin{remark}
    The integrability condition \eqref{ineq:hyp-fstar} is automatically satisfied 
    under Assumption~\ref{assmpt:operatorH}.
\end{remark}

Under the condition~\eqref{ineq:hyp-fstar}, the classical solution to \eqref{eq:nonlinear-cauchy-problem} constructed in Theorem~5.5 of \cite{Kry82} is bounded. Consequently, the constant appearing in the interior estimate of Theorem~5.6 in \cite{Kry82} depends only on the uniform bound of the terminal condition.

\begin{prop}\label{prop:exi-uni-edp-nonlin}
     Let $R>0$. Consider the Cauchy problem~\eqref{eq:nonlinear-cauchy-problem} with $\varphi \in C_b(\R)$ satisfying $\Norm{\varphi}_\infty \leq R$. Then this problem admits a unique solution $\phi \in C^{2}_{loc}(Q^T)\cap C(\overline{Q}^T).$\\
    If \eqref{ineq:hyp-fstar} holds, then the solution is bounded on $Q^T$. Moreover, for every $\tau > 0$, there exists $\alpha \in (0,1)$ and $ N\geq 0$, depending only on $\mathfrak{v}, \tau$ and $R$, such that \begin{align}\label{eq:holder_bound_nonlinear}
        \Norm{\phi}_{2+\alpha;\overline{Q}^{T-\tau}} \leq N,
    \end{align} where $Q^{T-\tau} = (0,T-\tau)\times\R$.
\end{prop}

\begin{remark}\label{rk:regularity-coeff-domain}
Under Assumption~\ref{assmpt:operatorH} and 
\eqref{eq:holder_bound_nonlinear}, it follows by Remark~\ref{rk:lipschitz-dH} that for any $\tau \in (0,T]$ there exists $\alpha \in (0,1)$ such that $D_{(p,M)} H(\dx \phi,\demi \dxx\phi)\in  C^{\alpha}(Q^{T-\tau};\R^2)$.
\end{remark}

We now derive a Schauder-type stability estimate in the interior domain, controlled by the uniform norm of the terminal condition. 

  \begin{prop}\label{prop:stability-nonlin-lip-coeff-bis}
        Let $(\phi_i)_{i\in I} \subset C^{2}_{loc}(Q^T) \cap C_b(\overline{Q}^T)$ be a family of solutions of the Cauchy problem~\eqref{eq:nonlinear-cauchy-problem} with respective terminal conditions $\varphi_i \in C_b(\R)$. Assume that there exists $R>0$ such that $\Norm{\varphi_i}_\infty < R$ for all $i\in I$. \\ Then for every $\tau \in(0,T)$, there exists $\alpha \in (0,1)$ and $ C>0$, such that for all $i,j \in I$ \begin{align} \label{eq:propstab-1-bis}
    &\Norm{\phi_{i} - \phi_j}_{2+\alpha; \overline{Q}^{T-\tau}} \leq C \Norm{\varphi_{i}-\varphi_{j}}_\infty ,\\
     &\Norm{\dx \phi_i}_{2+\alpha;\overline{Q}^{T-\tau} } \leq C . \label{eq:propstab-2-bis}
\end{align} 
    \end{prop}
    \begin{remark}
The stability result~\eqref{eq:propstab-1-bis} shows in particular that the solution operator of the nonlinear Cauchy problem \eqref{eq:nonlinear-cauchy-problem} is locally Lipschitz continuous with respect to the terminal condition, when viewed as a map from $C_b(\R)$ to $C^{2+\alpha}(Q^{T-\tau})$.
\end{remark}

The estimate~\eqref{eq:propstab-2-bis} implies that, for any classical solution $\phi$ of~\eqref{eq:nonlinear-cauchy-problem} with terminal condition $\varphi \in C_b(\R)$, the spatial derivatives $(\dx\phi,\dxx\phi)$ are locally Lipschitz on $Q^T$. Using the regularity and boundedness assumptions on $H$, we obtain the following result, which will be useful when studying the linear parabolic equations or SDEs associated with these coefficients.

\begin{corol}\label{coro:lip-local-coeff}
Let $\phi\in C^{2}_{loc}(Q^T)\cap C_b(\overline{Q}^T)$ be the solution to 
problem~\eqref{eq:nonlinear-cauchy-problem} with terminal condition 
$\varphi \in C_b(\R)$. Define the coefficients $(b,a)$ by
\begin{equation*}
b(t,x) = D_p H(t,x,\dx \phi(t,x),\tfrac12 \dxx \phi(t,x)),
\qquad
a(t,x) = D_M H(t,x,\dx \phi(t,x),\tfrac12 \dxx \phi(t,x)).
\end{equation*}
Then $a$ and $b$ are bounded and continuous on $Q^T$, and $a$ is uniformly elliptic. Moreover, for any $\tau \in (0,T)$, $a$ and $b$ are Lipschitz on $\overline{Q}^{T-\tau}$.
\end{corol}

\begin{remark}
    Since $a$ is uniformly elliptic, $0<\mathfrak{v}\leq a \leq \mathfrak{v}^{-1}$, hence for every $\tau\in(0,T]$, $x\mapsto \sqrt{a(t,x)}$ is Lipschitz, for every $t\in[0,T-\tau]$.
\end{remark} \noindent We fix a filtered probability space
$(\Omega,\mathcal{F},(\mathcal{F}_t)_{0\leq t \leq T},\mathbb{P})$
supporting a Brownian motion $(W_t)_{0\leq t\leq T}$. We consider the following SDE \begin{equation}\label{eq:eds-nonlin}
        dX_t = b(t,X_t) dt + a(t,X_t)^\demi dW_t.
\end{equation} By Lipschitz continuity and boundedness of the coefficients we obtain the following result concerning strong existence and uniqueness of solutions on $[0,T]$. \begin{lemma}\label{lem:strong-solution}
    For every $x\in\R$, SDE~\eqref{eq:eds-nonlin} with initial
    condition $X_0=x$ admits a strong solution on $[0,T]$,
    and pathwise uniqueness holds.
\end{lemma} \noindent The coefficients of the SDE thus depend on the solution $\phi$ associated with the terminal condition $\varphi$. Accordingly, we denote by $\left(X_\cdot(\varphi)\right)$ the solution to \eqref{eq:eds-nonlin} corresponding to $(\phi,\varphi)$.

\noindent Due to the relatively mild regularity assumption imposed on the terminal condition, norms of the solution's derivatives can be controlled only in the interior of the domain and may blow up at the terminal boundary. In order to properly study the linearization of the equation~\eqref{eq:nonlinear-cauchy-problem}, we introduce a regularized non-linear Cauchy problem obtained by smoothing the terminal condition. In the following, we assume that the terminal condition $\varphi\in C_b(\R)$ can be approximated by a sequence of mollifications $(\varphi_\eta)_{\eta\in \N} \in C^{2+1}(\R)$ such that \begin{equation*}
    \Norm{\varphi-\varphi_\eta}_\infty \xrightarrow[\eta\to \infty]{}0.
\end{equation*} Such an approximation is guaranteed, for instance, if $\varphi \in \BUC(\R).$ For $\eta \in \N$, let $\phi_\eta \in C^{2}_{loc}(Q^T)\cap C_b(\overline{Q}^T)$ be the solution to problem~(\ref{eq:nonlinear-cauchy-problem}) with terminal condition $\varphi_\eta$. Theorem~5.3 in \cite{Kry82} ensures the existence of the solution and its regularity up to the boundary: there exists $\alpha_\eta \in (0,1)$ such that $ \phi_\eta \in C^{2+\alpha_\eta}(\overline{Q}^T)$. When no ambiguity arises, we denote the solution of the SDE~\eqref{eq:eds-nonlin} associated with $\varphi_\eta$ by $\left(X_t(\eta)\right)_{t\in[0,T]}$. 

\noindent We consider the following linear Cauchy problem, which arises from the linearization of equation~\eqref{eq:nonlinear-cauchy-problem}.\begin{align} \label{eq:linearization-cauchy-pb}
    \left\{ \begin{array}{l}
         \dt \psi + \dx \psi \, \partial_p H(t,x,\dx\phi,\demi\dxx\phi)+\demi \dxx \psi \, \partial_M H(t,x,\dx\phi,\demi \dxx\phi) = 0 \text{ on } Q^T,  \\
          \psi(T,\cdot) = g(\cdot).
    \end{array} \right. 
\end{align} This problem depends on the terminal data $g$ and, through its coefficients, on the solution $\phi$ of \eqref{eq:nonlinear-cauchy-problem} with terminal data $\varphi$. Accordingly, we denote by $(\varphi,g)$ the pair of parameters associated with problem~\eqref{eq:linearization-cauchy-pb}. Let $g\in \BUC(\R)$, hence there exists a sequence $(g_\eta)_\eta \in C^{2+1}(\R),$ such that \begin{equation*}
    \Norm{g_\eta -g}_\infty \xrightarrow[\eta \to \infty]{} 0.
\end{equation*} The mollified family $(g_\eta)_\eta$ shares the same modulus of continuity as $g$. By Theorem~9.2.3 in \cite{Kry96_holder}, the linear problem~\eqref{eq:linearization-cauchy-pb}, for parameters $(\varphi_\eta,g_\eta)$ admits a unique solution $\psi_\eta \in C^{2+\alpha_\eta}(\overline{Q}^T).$ We state the following result to study the existence and uniqueness of a classical solution $\psi \in C^{2}_{loc}(Q^T)\cap C_b(\overline{Q}^T)$ to~\eqref{eq:linearization-cauchy-pb} for parameters~$(\varphi,g)$. This result is mainly based on a compactness argument provided by compact embeddings of Hölder spaces, see Lemma~\ref{lem:holder-compact-embedding}.

\begin{prop}\label{prop:abstrait-existence-convergence}
Consider for each $\eta\in\N$, $\alpha_\eta\in (0,1)$, $(a_\eta,b_\eta,f_\eta)\in C^{\alpha_\eta}(Q^T)$ and $g_\eta \in C^{2+\alpha_\eta}(\R)$. Suppose that the family $(g_\eta)_{\eta}$ admits a common modulus of continuity $\omega(\cdot)$, and that there exists $\mathfrak{v} >0$ such that \begin{equation*}
    \mathfrak{v} I_d\le a_\eta(t,x)\le \mathfrak{v}^{-1}I_d, \;|b_\eta(t,x)|\leq \mathfrak{v}^{-1},\; |f_\eta(t,x)|\leq \mathfrak{v}^{-1}
\qquad \forall(t,x)\in Q^T,\ \eta\in\mathbb N.
\end{equation*} Assume also that there exists $(a,b,f)\in C_b(Q^T)$ and $g\in \BUC(\R)$, and for all $\tau\in(0,T)$, there exists $\theta_\tau\in(0,1)$, such that \begin{equation*}
    \Norm{g_\eta-g}_\infty \to 0,\quad \Norm{a_\eta-a}_{\theta_\tau;Q^{T-\tau}} \to 0,\quad \Norm{b_\eta-b}_{\theta_\tau;Q^{T-\tau}}\to 0,\quad \Norm{f_\eta-f}_{\theta_\tau;Q^{T-\tau}}\to 0 \quad \text{as } \eta\to\infty.
\end{equation*} Then, there exists a unique solution $u\in C^2_{loc}(Q^T)\cap C_b(\overline{Q}^T)$ of \begin{equation}\label{eq:pde-lemm-abstrait-holder}
    \dt u + b\dx u + \demi a\dxx u + f = 0 \quad \text{on } Q^T, \qquad u(T,\cdot)=g(\cdot). 
\end{equation} Moreover, considering $(u_\eta)_{\eta\in\N}$ the sequence of solutions of \begin{equation}\label{eq:pde-lemm-abstrait-holder-eta}
    \dt u + b_\eta\dx u + \demi a_\eta\dxx u + f_\eta = 0 \quad \text{on } Q^T, \qquad u(T,\cdot)=g_\eta(\cdot),
\end{equation} such that each $u_\eta \in C^{2+\alpha_\eta}(\overline{Q}^T)$, we have that for every $(t,x)\in \overline{Q}^T$, $u_\eta(t,x)\to u(t,x)$ as $\eta\to\infty$ and for every $\tau \in (0,T)$, there exists $C_\tau>0$ such that \begin{equation}\label{eq:res-limit-lem}
    \Norm{u}_{2+\theta_\tau;\overline{Q}^{T-\tau}}<C_\tau, \text{ and } \quad \forall K \Subset [0,T-\tau]\times\R, \forall \gamma\in (0,\theta_\tau),\;\Norm{u_\eta-u}_{2+\gamma;K}\to 0 \text{ as } \eta\to\infty. 
\end{equation}
\end{prop}

Applying this result to the problem~\eqref{eq:linearization-cauchy-pb} we obtain classical well-posedness for coefficients associated with data $(\varphi,g)\in \BUC(\R)\times\BUC(\R) $. Moreover, denoting by $\psi_\eta$ the solution associated with mollified data $(\varphi_\eta,g_\eta)$, we obtain a stability result for solutions.

\begin{corol}\label{coro:existence-convergence-linearization}
Let $g \in \BUC(\R)$.  
Then the problem \eqref{eq:linearization-cauchy-pb} associated with $(\varphi,g)$ admits a unique solution $\psi \in C^{2}_{loc}(Q^T)\cap C_b(\overline{Q}^T)$. For every $\tau \in (0,T]$, there exists $\alpha(\tau)\in (0,1)$ and $C_\tau>0$, \begin{equation*}
    \Norm{\psi}_{2+\alpha(\tau);\overline{Q}^{T-\tau}} \leq C_\tau,
\end{equation*} and for every compact $K \Subset (0,T-\tau)\times \R$, and $\gamma \in (0,\alpha(\tau))$ we have
\begin{equation*}
        \Norm{\psi_\eta - \psi}_{2+\gamma;K} \to0, \text{ as } \eta \to \infty.
    \end{equation*} Moreover, \begin{equation*}
        \Norm{\psi_\eta - \psi}_{\infty;\overline{Q}^T} \to 0, \text{ as } \eta \to \infty.
\end{equation*}
\end{corol}

\section{Differentiability of the solution map}\label{sec:differentiability}

In this section, we study the differentiability of the mapping 
$\Phi : f \mapsto \phi^f$, where $\phi^f$ denotes the solution to Problem~\eqref{eq:nonlinear-cauchy-problem} with terminal condition $\phi^f(T,\cdot)=f(\cdot)$.
The precise functional spaces for $f$ and for the corresponding solutions will be specified below. We first establish a regularity result for this map in the case of smooth terminal data. Using a density argument, we then extend this regularity to non-smooth data restricted to a suitable subspace. We relate solutions of the linearized equation to the derivative of the map $\Phi$. Throughout this section, we assume that $H$ satisfies Assumption~\ref{assmpt:operatorH}. \\

\noindent We consider the following Banach spaces, for $\alpha \in (0,1)$: \begin{align*}
    & E_\alpha := C^{2+\alpha}(\overline{Q}^T),\\ %
    & F_\alpha := C^\alpha(Q^T)\times C^{2+\alpha}(\R).
\end{align*}
In the sequel, differentiability of maps defined between Banach spaces is understood in the sense of the Fréchet derivative. We introduce notation for the Cauchy problems associated with the differentials of the solution map, corresponding to successive linearizations of equation~\eqref{eq:nonlinear-cauchy-problem}:
\begin{align}
   &\left\{
    \begin{array}{ll} \label{eq:linearisation-1}
          \dt \psi^{f,g} + \nabla H[\phi^f][\psi^{f,g}] = 0, \quad \forall(t,x) \in Q^T,  \\
          \psi^{f,g}_T \equiv g,
    \end{array}
    \right.\\
    &\left\{
    \begin{array}{ll} \label{eq:linearisation-2}
         \dt\gamma^{f,g,h} + \nabla H[\phi^f][\gamma^{f,g,h}] + \Hessian{H}[\phi^f][\psi^{f,g},\psi^{f,h}]
          = 0, \quad \forall(t,x) \in Q^T,\\
         \gamma_T \equiv 0,
    \end{array}
    \right.
\end{align} 
\vspace{-2em} 
where
\begin{equation*} \begin{split}
    \nabla H[a][b]:=&D_p H(\dx a, \demi\dxx a)\dx b + \demi D_M H(\dx a, \demi\dxx a)\dxx b,  \\
    \Hessian{H}[a][b,c] := &D^2_{pp} H(\dx a, \demi\dxx a) \dx b \, \dx c + \frac{1}{4} D^2_{MM} H(\dx a, \demi\dxx a) \dxx b \, \dxx c  \\
    &+ \demi D^2_{p,M} H(\dx a, \demi\dxx a) \left( \dx b \, \dxx c + \dxx b \, \dx c\right).
    \end{split}
\end{equation*}

\begin{remark}\label{rk:linearisation-existence-source}
    Let $a, b, c \in E_\alpha$. By assumptions, see Remark~\ref{rk:lipschitz-dH}, the coefficients of the linear Cauchy problems above belong to $C^\alpha(Q^T)$. Let $(f_1,g_1) \in F_\alpha$. Applying Theorem~9.2.3 in \cite{Kry96_holder}, there exist unique $u,v \in E_\alpha$ such that
    \begin{align*}
    \begin{split}
        &\dt u + \nabla H[a][u] = f_1 
        \quad \text{on } Q^T, 
        \qquad u_T = g_1 \quad \text{on } \R,\\
        &\dt v + \nabla H[a][v] + \Hessian{H}[a][b,c] = f_1 
        \quad \text{on } Q^T, 
        \qquad v_T = g_1 \quad \text{on } \R .
    \end{split}
    \end{align*}
    These equations are linear parabolic equations with coefficients in $C^\alpha(Q^T)$ and source terms $f_1$ and $f_1 - \Hessian{H}[a][b,c]$, respectively.
\end{remark}

\begin{prop}\label{prop:smooth-coef-linearization}
Assume $H$ satisfies Assumption~\ref{assmpt:operatorH}. The map $$\Phi : C^{2+1}(\R) \to C^2(\overline{Q}^T),\quad  f \mapsto \phi^f  $$ is twice continuously differentiable, such that $d\Phi_f(g)=\psi^{f,g}, $ and $ d^2\Phi_f(g,h)=\gamma^{f,g,h}$, for $g,h \in C^{2+1}(\R)$.
\end{prop}

\begin{proof}%
Let $f \in C^{2+1}(\R)$. By Theorem~5.3 in \cite{Kry82}, there exist a solution $\phi^f$ and a constant $\alpha \in (0,1)$ such that $\phi^f \in E_\alpha$. We introduce the map
\begin{equation*}
     \Theta : E_\alpha \times E_\alpha \ni \begin{pmatrix}
         \phi \\
         \psi
     \end{pmatrix} \mapsto  \begin{pmatrix}
            \dt \phi + H(\dx \phi, \demi \Dxx \phi) \\
            \phi_T \\
            \dt \psi + \nabla H[\phi][\psi]   \\
            \psi_T
        \end{pmatrix} \in F_\alpha \times F_\alpha .
 \end{equation*}
By regularity of $H$, we have $\Theta \in C^1(E_\alpha\times E_\alpha; F_\alpha\times F_\alpha)$, with differential given by the linearization of the above system: for $(\phi,\psi)\in E_\alpha\times E_\alpha$,
\begin{equation*}
     d\Theta_{(\phi,\psi)} : E_\alpha \times E_\alpha \ni \begin{pmatrix}
        u\\
        v
     \end{pmatrix} \mapsto \begin{pmatrix}
                \dt u + \nabla H[\phi][u] \\
                u_T \\
                \dt v + \nabla H[\phi][v] + \Hessian{H}[\phi][\psi,u] \\
                v_T
            \end{pmatrix} \in F_\alpha \times F_\alpha .
 \end{equation*} Let $(f_1,g_1,f_2,g_2) \in F_\alpha \times F_\alpha$. By Remark~\ref{rk:linearisation-existence-source}, there exist unique functions $u,v \in E_\alpha$ such that
\[
d\Theta_{(\phi,\psi)}(u,v) = (f_1,g_1,f_2,g_2).
\]
Hence, $d\Theta_{(\phi,\psi)}$ is an isomorphism from $E_\alpha\times E_\alpha$ onto $F_\alpha\times F_\alpha$. Since these are Banach spaces, the Banach--Schauder theorem (see, for instance, Corollary~II.6 in \cite{brezis}) implies that $d\Theta_{(\phi,\psi)}$ is a homeomorphism. Now let $g,h \in C^{2+1}(\R)$. In particular, $g,h \in C^{2+\alpha}(\R)$. Since $\phi^f \in E_\alpha$, Remark~\ref{rk:linearisation-existence-source} yields $\psi^{f,g}, \psi^{f,h} \in E_\alpha$. Moreover, $$\Theta(\phi^f,\psi^{f,g}) = (0,f,0,g).$$ By the local inverse function theorem (see Theorem~4.F in \cite{Zeidler_1986}), there exist neighborhoods $V_1 \subset E_\alpha \times E_\alpha$ of $(\phi^f,\psi^{f,g})$ and $V_2 \subset F_\alpha\times F_\alpha$ of $(0,f,0,g)$ such that $\Theta$ is a bijection from $V_1$ onto $V_2$, and its inverse satisfies $\Theta^{-1} \in C^1(V_2,V_1)$. Let $(a,b):=d(\Theta^{-1})_{(0,f,0,g)}(0,h,0,0)$. Since
\begin{equation*}
    d\Theta_{(\phi^f,\psi^{f,g})}\!\left(d(\Theta^{-1})_{(0,f,0,g)}(\cdot)\right)=\mathrm{Id}(\cdot),
\end{equation*} we obtain \begin{equation*}
    d\Theta_{(\phi^f,\psi^{f,g})}(a,b)=(0,h,0,0).
\end{equation*} By uniqueness of solutions stated in Remark~\ref{rk:linearisation-existence-source}, we conclude that $(a,b)=(\psi^{f,h},\gamma^{f,g,h})$.

Let $\pi_1$ and $\pi_2$ denote the projections onto the first and second components of 
$E_\alpha \times E_\alpha$, i.e. $\pi_i : (u_1,u_2) \mapsto u_i$. Observe that, locally (on $V_2$), $\Phi(\cdot)$ coincides with $\pi_1 \circ \Theta^{-1}(0,\cdot,0,g)$. A priori, this inverse could depend on the fixed last two variables $(0,g)$. However, since $\Phi$ is well defined by uniqueness of the solution to the Cauchy problem~\eqref{eq:nonlinear-cauchy-problem} (see Theorem~5.3 in \cite{Kry82}), the choice of these last two variables is in fact arbitrary.
Therefore, \begin{equation*}
    d\Phi_f(\cdot ) = \pi_1 \circ d(\Theta^{-1})_{(0,f,0,g)}(0,\cdot,0,0).
\end{equation*} Similarly, we observe that $d\Phi_\cdot(g) = \pi_2 \circ \Theta^{-1}(0,\cdot,0,g)$. However, the second component of $\Theta^{-1}$ depends on all its arguments. Hence,
\begin{equation*}
    d^2\Phi_f(g,\cdot) = \pi_2 \circ d(\Theta^{-1})_{(0,f,0,g)}(0,\cdot,0,0).
\end{equation*} This shows the claimed expressions for the derivatives, since $(d\Phi_f(h), d^2\Phi_f(g,h)) = (a,b)$ and we already proved that $(a,b) = (\psi^{f,h}, \gamma^{f,g,h})$.

Although differentiability is established locally at each point in a Hölder space $C^{2+\alpha}$, with the exponent depending on the point, the final statement is formulated from a strengthened domain $C^{2+1}(\R) \subset C^{2+\alpha}(\R) $ to a relaxed codomain $C^2(\overline{Q}^T) \supset E_\alpha$, in the sense of continuous embeddings. This allows for a global definition of the Fréchet derivative on these spaces.
\end{proof}

\begin{remark}\label{rk:global-inverse} We can show that the function $\Theta$ is injective by uniqueness of solutions to the related system of Cauchy problems. Let $(\phi,\psi)$ and $(\wtilde{\phi},\wtilde{\psi})$ be such that $\Theta(\phi,\psi)=\Theta(\wtilde{\phi},\wtilde{\psi})$. We first show that $\phi =\wtilde{\phi}$. Set $u=\phi - \wtilde{\phi}$. Then $u_T \equiv 0$, and $u$ satisfies the same equation as in the proof of Proposition~\ref{prop:stability-nonlin-lip-coeff-bis}. By Corollary~8.1.5 in \cite{Kry96_holder}, it follows that $u\equiv 0$, hence $\phi = \wtilde{\phi}$. Since $\phi=\wtilde{\phi}$, the functions $\psi$ and $\wtilde{\psi}$ solve the same linear Cauchy problem. Applying again Corollary~8.1.5, we obtain $\psi=\wtilde{\psi}$.
Hence, the local inverse function theorem used above can be extended globally, since $\Theta$ is injective. 
\end{remark}

Let $M\in\N_*$, consider a vector $G\in \BUC(\R;\R^M)$. For $\lambda\in\mathbb{R}^M$ and $i,j\in\{1,\dots,M\}$, we set $\psi_i^\lambda := \psi^{\lambda\cdot G,\, G_i}$, and $\gamma_{i,j}^\lambda := \gamma^{\lambda\cdot G,\, G_i,\, G_j}$.

The previous proposition establishes that the solution map associated with equation~\eqref{eq:nonlinear-cauchy-problem} is twice differentiable between suitable Banach spaces of sufficiently regular functions. Using compactness and stability arguments, we extend the first-order differentiability of the solution map when considered on the subspace $\spn{(G)}$ of $\BUC$ functions. In the proof, when referring to the modulus of continuity of $G$, we set $\omega(\cdot):=\max_{i\le M}\omega_i(\cdot)$. Let $(G^\eta)_{\eta\in\mathbb{N}}$ be a sequence of mollifications of $G$, obtained by mollifying each component, such that for all $m\in\{1,\dots,M\}$ and $\eta\in\mathbb{N}$ we have $G_m^\eta\in C^{2+1}(\mathbb{R})$. For $\eta\in\mathbb{N}$ and $\lambda\in\mathbb{R}^M$, we denote by $\phi_\eta^\lambda$ the solution to Problem~\eqref{eq:nonlinear-cauchy-problem} with terminal condition $\lambda\cdot G^\eta$. Similarly, for each $m$, we denote by $\psi_m^{\lambda,\eta}$ the function $\psi^{f,g}$ with $f=\lambda\cdot G^\eta$ and $g=G_m^\eta$.

\begin{prop}\label{prop:put-linearization}
    Assume $H$ satisfies Assumption~\ref{assmpt:operatorH}. The map \begin{equation*}
        \Phi_G : \R^M \to C_b(\overline{Q}^T), \quad  \lambda \mapsto \phi^\lambda 
    \end{equation*} is continuously differentiable, such that $d(\Phi_G)_\lambda(\mu) = \nabla\Phi_G(\lambda)\cdot \mu = \psi^{\lambda}\cdot\mu$ for $\mu \in \R^M$. 
\end{prop}

\begin{proof}
    We establish the continuous differentiability of $\Phi_G$ using Proposition~\ref{prop:smooth-coef-linearization} and a compactness argument.  For $\eta \in \N$, we define the smooth-payoff solution map \begin{equation*}
    \Phi_G^\eta : \R^M \ni \lambda \mapsto \phi_\eta^\lambda \in C_b(\overline{Q}^T). 
\end{equation*} By Proposition~\ref{prop:smooth-coef-linearization}, we have that for all $\eta \in \N$,  $\Phi_G^\eta \in C^1(\R^M, C_b(\overline{Q}^T))$. The remainder of the proof is to show that for any point in $\R^M$ there is a compact neighbourhood $K$, such that the sequence $(\Phi_G^\eta)_{\eta\in\N}$ is relatively compact in $C^1\left(K, C_b(\overline{Q}^T)\right)$ equipped with the following norm \begin{equation}\label{eq:norm-c1-proof}
\Norm{\Phi_G}_{C^1(K,C_b(\overline{Q}^T))}
=
\sup_{\lambda \in K}
\left\{
\Norm{\Phi_G(\lambda)}_{\infty;\overline{Q}^T}
+
\Norm{\nabla \Phi_G(\lambda)}_{\mathcal{L}(\R^M,C_b(\overline{Q}^T))}
\right\}. 
\end{equation} Using previous stability results, we then conclude, by uniqueness of the limit, that the map $\Phi_G$ is differentiable. Consider a bounded open set $U\subset \R^M$, and denote by $K$ its closure. In order to apply the Arzelà--Ascoli theorem, we show that the sequence $(\Phi_G^\eta)_{\eta\in\N}$ is pointwise relatively compact and equicontinuous on $K$. Let $\lambda \in K$.\\
\textit{Relative compactness}: 
By Proposition~\ref{prop:exi-uni-edp-nonlin}, $\Phi_G(\lambda) \in C_b(\overline{Q}^T)$ and by Proposition~\ref{prop:comp-principle-nonlinear} we have, \begin{equation*}
    \Norm{\Phi_G^\eta(\lambda)-\Phi_G(\lambda)}_{\infty;\overline{Q}^T} \to 0, \quad \text{as } \eta \to \infty.
\end{equation*} Similarly, by Corollary~\ref{coro:existence-convergence-linearization}, $\psi^\lambda \in \cL(\R^M,C_b(\overline{Q}^T))$ and \begin{equation*}
     \Norm{\nabla \Phi_G^\eta(\lambda)- \psi^\lambda}_{\mathcal{L}(\R^M,C_b(\overline{Q}^T))} \to 0, \quad \text{as } \eta \to \infty.
\end{equation*} 
\textit{Equicontinuity}: We show that for every $\varepsilon>0$ there exists $\delta>0$ such that for all 
$\mu \in \mathcal{B}(\lambda,\delta)$ and all $\eta \in \N$, \begin{equation*}
\Norm{\nabla \Phi_G^\eta(\lambda)-\nabla \Phi_G^\eta(\mu)}_{\mathcal{L}(\R^M,C_b(\overline{Q}^T))}
< \varepsilon .
\end{equation*} For $(t,x)\in\overline{Q}^T$ and $\nu\in\R^M$, we denote by 
$(X_s^{t,x}(\nu,\eta))_{s\in[t,T]}$ the solution to the SDE~\eqref{eq:eds-nonlin}
with coefficients $(a^\eta_\nu,b^\eta_\nu)$ associated with the terminal
condition $\nu\cdot G_\eta$, starting from $X_t^{t,x}(\nu,\eta)=x$. Applying the Feynman--Kac formula to $\psi^{\lambda,\eta}(t,x)$ and
$\psi^{\mu,\eta}(t,x)$, we obtain for every $m\in\{1,\dots,M\}$,
\begin{equation*}
\left| \psi^{\lambda,\eta}_m(t,x)-\psi^{\mu,\eta}_m(t,x) \right|
\le
\E\!\left[
\left| G^\eta_m(X_T^{t,x}(\lambda,\eta))
      -G^\eta_m(X_T^{t,x}(\mu,\eta)) \right|
\right].
\end{equation*} First, note that we can fix a $\wtilde{\tau}\in[0,T)$, such that for all $(t,x)\in [\wtilde{\tau},T]\times\R$, $\lambda,\mu \in \R^M$, $\eta\in \N$, and $m\in\{1,\dots,M\},$ \begin{equation*}
    \left| \psi^{\lambda,\eta}_m(t,x)-\psi^{\mu,\eta}_m(t,x) \right|<\varepsilon .
\end{equation*} Indeed, for any $\delta_1 >0$, \begin{equation*}
    \E\!\left[
\left| G^\eta_m(X_T^{t,x}(\lambda,\eta))
      -G^\eta_m(X_T^{t,x}(\mu,\eta)) \right| \right] \leq \omega(\delta_1) + 2\Norm{G}_\infty \mathbb{P}\left(\left|X_T^{t,x}(\lambda,\eta)-X_T^{t,x}(\mu,\eta) \right| > \delta_1 \right).
\end{equation*} Applying the Markov inequality, by uniform boundedness of coefficients, there exists $L(\mathfrak{v})>0$ such that  \begin{equation*}
    \E\!\left[
\left| G^\eta_m(X_T^{t,x}(\lambda,\eta))
      -G^\eta_m(X_T^{t,x}(\mu,\eta)) \right| \right] \leq \omega(\delta_1) + \frac{\Norm{G}_\infty }{\mathfrak{v}\delta_1} L(\mathfrak{v}) \left( \sqrt{T-t} + T-t \right).
\end{equation*} Consider $(t,x)\in [0,\wtilde{\tau}]\times \R$. Using the same argument we obtain that for any $\eta\in \N$, and $\delta_2>0$ \begin{equation*}
     \esperance{\left| G^\eta_m( X^{t,x}_T(\lambda,\eta)) - G^\eta_m( X^{t,x}_T(\mu,\eta)) \right| } \leq \omega(\delta_2) + \frac{2 \Norm{G}_\infty }{\delta_2} \esperance{\left| X_{T}^{t,x}(\lambda,\eta)-X_{T}^{t,x}(\mu,\eta) \right|}.
\end{equation*} For $\tau\in (0,T-\wtilde{\tau}]$,\begin{equation}\label{ineq:ineq-2}\begin{split}
        \esperance{\left| X^{t,x}_{T}(\lambda,\eta)) - X^{t,x}_{T}(\mu,\eta)) \right|} \leq& \, \esperance{\left| X_{T-\tau}^{t,x} (\lambda,\eta)-X_{T}^{t,x} (\lambda,\eta)\right|+\left| X_{T-\tau}^{t,x} (\mu,\eta)-X_{T}^{t,x} (\mu,\eta)\right|}\\
        +& \, \esperance{\left| X^{t,x}_{T-\tau}(\lambda,\eta)) - X^{t,x}_{T-\tau}(\mu,\eta)) \right|}\end{split}
    \end{equation} By uniform boundedness of the drift and diffusion coefficients, there exists $C(\mathfrak{v})>0,$ such that for all $\nu \in \R^M, \, \eta \in \N, $ and $(t,x)\in [0,T-\wtilde{\tau}]\times \R$,  \begin{align*}
        \esperance{\left| X_{T-\tau}^{t,x} (\nu,\eta)-X_{T}^{t,x} (\nu,\eta)\right|} \leq C(\mathfrak{v})\left(\sqrt{\tau}+\tau\right).
    \end{align*} Hence, by choosing $\tau$ sufficiently small, we can ensure that the first expectation in the right-hand side of~\eqref{ineq:ineq-2} is uniformly smaller than $\frac{\veps}{4}.$\\  Let $s\in [t,T-\tau],$\begin{align*}
        \esperance{\left| X^{t,x}_{s }(\lambda,\eta) -  X^{t,x}_{s }(\mu,\eta) \right|^2} \leq\, &2\esperance{\left| \int_{t}^{s} b^\eta_\lambda(r,X_r^{t,x}(\lambda,\eta)) - b^\eta_\mu(r,X_r^{t,x}(\mu,\eta)) dr \right|^2} \\
        +\, &2\esperance{\left| \int_{t}^{s} a^\eta_\lambda(r,X_r^{t,x}(\lambda,\eta)) - a^\eta_\mu(r,X_r^{t,x}(\mu,\eta)) dW_r \right|^2}.
    \end{align*}
    Thus, we obtain the following inequalities. The second inequality follows from the uniform (in $(t,x,\lambda,\eta$)) $L$-Lipschitz continuity of $b^\eta_\lambda$ on $[0,T-\tau]\times\R$ provided by Corollary~\ref{coro:lip-local-coeff} for the first term, while the second term is controlled using the stability result of Proposition~\ref{prop:stability-nonlin-lip-coeff-bis}, with $C_1>0$ denoting the associated constant.
    \begin{equation*}\begin{split}
        &\esperance{\left| \int_{t}^{s} b^\eta_\lambda(r,X_r^{t,x}(\lambda,\eta)) - b^\eta_\mu(r,X_r^{t,x}(\mu,\eta)) dr \right|^2} \\
        &\quad\leq 2T\, \esperance{\int_{t}^{s} \left| b^\eta_\lambda(r,X_{r}^{t,x}(\lambda,\eta)) - b^\eta_\lambda(r,X_{r}^{t,x}(\mu,\eta))\right|^2 dr } \\
        &\qquad+ 2T \, \esperance{\int_{t}^{s} \left| b^\eta_\lambda(r,X_{r}^{t,x}(\mu,\eta)) - b^\eta_\mu(r,X_{r}^{t,x}(\mu,\eta))\right|^2 dr } \\
        &\quad\leq 2T \left( L^2 \esperance{\int_t^s \left| X_{r}^{t,x}(\lambda,\eta) - X_{r}^{t,x}(\mu,\eta)\right|^2 dr } + C_1 |\lambda-\mu|^2 \right). 
        \end{split}
    \end{equation*}
    Applying Itô's isometry and proceeding as above yields the following inequality for the diffusion coefficients:\begin{equation*}\begin{split}
        &\esperance{\left| \int_t^{s} a_\lambda^\eta(r,X_r^{t,x}(\lambda,\eta))-a_\mu^\eta(r,X_r^{t,x}(\mu,\eta)) dW_r \right|^2} \\
        &\quad\leq 2 L^2\esperance{\int_t^s \left| X_{r}^{t,x}(\lambda,\eta) - X_{r}^{t,x}(\mu,\eta)\right|^2 dr } + C_1 \abs{\lambda-\mu}^2.
    \end{split}
    \end{equation*}
    Hence, by the uniform boundedness of the diffusion coefficient, there exists $C_2>0$ such that $\forall \nu \in \R^M, \forall \eta \in \N, \forall (t,x)\in [0,T-\wtilde{\tau}]\times \R$, $\forall s\in(t,T-\tau)$ \begin{align*}
         \esperance{\left| X^{t,x}_{s }(\lambda,\eta) -  X^{t,x}_{s }(\mu,\eta) \right|^2} \leq C_2 \left(  \esperance{\int_t^s \left| X_r^{t,x}(\lambda,\eta) - X_r^{t,x}(\mu,\eta)\right|^2 dr } +  |\lambda-\mu|^2 \right).
    \end{align*} An application of Grönwall's inequality on $[t,T-\tau]$ yields the equicontinuity of the sequence.\\
    By the Arzelà--Ascoli theorem, we can extract a subsequence of $(\Phi_G^\eta)$ converging to some
    \begin{equation*}
        \Phi_G^\star \in C^1(\R^M,C_b(\overline{Q}^T))
    \end{equation*}
    with respect to the uniform norm on each compact set $K$. In particular, the convergence holds pointwise on the whole domain. Moreover, by the stability result of Proposition~\ref{prop:stability-nonlin-lip-coeff-bis}, for every $ \lambda \in \R^M,$ $\Phi_G^\eta(\lambda)$ converges pointwise to $\Phi_G(\lambda)$. By uniqueness of the limit, we conclude that $\Phi_G^\star= \Phi_G$. Using Corollary~\ref{coro:existence-convergence-linearization} we characterize the gradient by the same argument.
\end{proof}

\begin{remark}
    In particular, let $\Phi_0$ denote the \textit{space-evaluated} solution map \begin{equation*}
        \Phi_0 : \R^M \ni \lambda \mapsto \phi^\lambda(0,x_0)\in \R. \\
    \end{equation*} Then its gradient is given by \begin{equation*}
        \nabla \Phi_0[\lambda] = \begin{pmatrix}
            \psi^\lambda_1(0,x_0) \\
            \vdots \\
            \psi^\lambda_M(0,x_0)
        \end{pmatrix}.
    \end{equation*}
\end{remark}

\section{Convexity property of the solution map}\label{sec:convexity}

We first establish convexity of the solution map with respect to
the terminal condition by a comparison argument. We then study
local strong convexity of the point-evaluated map
$\Phi_0:\R^M \ni \lambda\mapsto\phi^\lambda(0,x_0) \in \R$ under additional
non-degeneracy assumptions.

\begin{prop}[Convexity of the solution map]
\label{prop:convexity-solution-map}
Assume that $H$ satisfies Assumption~\ref{assmpt:operatorH}
and that $(p,M)\mapsto H(t,x,p,M)$ is convex for every
$(t,x)\in Q^T$.
For $f\in\BUC(\R)$, let $\phi^f$ denote the bounded classical
solution of~\eqref{eq:nonlinear-cauchy-problem} with terminal
condition $f$.

The solution map
\begin{equation*}
    \Phi:\BUC(\R)\longrightarrow C_b(\overline{Q}^T),
    \qquad f\longmapsto\phi^f,
\end{equation*}
is convex with respect to the pointwise order: for every
$f,g\in\BUC(\R)$ and $\theta\in[0,1]$,
\begin{equation}\label{eq:convexity-solution-map}
    \phi^{\theta f+(1-\theta)g}(t,x)
    \leq
    \theta\phi^f(t,x)+(1-\theta)\phi^g(t,x),
    \qquad (t,x)\in\overline{Q}^T.
\end{equation}
In particular, for every $G\in\BUC(\R;\R^M)$ and every
$(t,x)\in\overline{Q}^T$, the map
\begin{equation*}
    \R^M\ni\lambda\longmapsto\phi^{\lambda\cdot G}(t,x)
\end{equation*}
is convex.
\end{prop}

\begin{proof}
Fix $f,g\in\BUC(\R)$ and $\theta\in[0,1]$, and set
$w:=\theta\phi^f+(1-\theta)\phi^g$.
The solutions belong to
$C^2_{\mathrm{loc}}(Q^T)\cap C_b(\overline{Q}^T)$, so the
same holds for $w$. By convexity of $H$ in its last two variables,
\begin{align*}
    &\partial_t w
      +H\left(t,x,\partial_x w,\tfrac12\partial_{xx}w\right)\\
    &\quad\leq
    \theta\left[
        \partial_t\phi^f
        +H\left(t,x,\partial_x\phi^f,
                     \tfrac12\partial_{xx}\phi^f\right)
    \right]\\
    &\qquad+
    (1-\theta)\left[
        \partial_t\phi^g
        +H\left(t,x,\partial_x\phi^g,
                     \tfrac12\partial_{xx}\phi^g\right)
    \right]
    =0
    \qquad\text{on }Q^T.
\end{align*}
Thus $w$ is a supersolution with terminal condition
$w(T,\cdot)=\theta f+(1-\theta)g$.
Proposition~\ref{prop:comp-principle-nonlinear} gives
\eqref{eq:convexity-solution-map} in $Q^T$, and continuity
extends the inequality to $\overline{Q}^T$.
The last assertion follows by taking $f=\lambda\cdot G$
and $g=\mu\cdot G$, for $\lambda,\mu\in\R^M$.
\end{proof}

Recall that a function $f:\R^n\to\R$ is called locally strongly convex on $\R^n$, if for every compact convex set $K \Subset \R^n$, $f$ is strongly convex on $K$. To obtain local strong convexity, we impose the following
non-degeneracy condition on the terminal data.

\begin{assumption}\label{assmpt:G}
    The vector $G=(G_1,\dots,G_M)$ belongs to $\BUC(\R;\R^M)$ and its components are linearly independent modulo constants, i.e. if $\sum_{m=1}^M \mu_m G_m(\cdot)$ is constant on $\mathbb R$, then $\mu=0$.
\end{assumption} 

\subsection{Local strong convexity under strong Hamiltonian convexity}

\begin{theorem}\label{thm:strict-cvx-flow}
    If $H$ satisfies Assumptions~\ref{assmpt:operatorH},\ref{assmpt:operatorH-cvx} and $G$ satisfies Assumption~\ref{assmpt:G}, then, for $x_0 \in \R^d$, the map $\Phi_0$ is locally strongly convex.
\end{theorem}

\noindent To prove this result, we first establish a technical lemma relying on the weak unique continuation property for linear parabolic equations.

\begin{lemma}\label{lem:wucp} Assume $H$ satisfies Assumption~\ref{assmpt:operatorH} and $G$ satisfies Assumption~\ref{assmpt:G}.
Let $\nu \in \R^M$ and $\mu \in \mathbb{S}^{M-1}$. 
There exists $\tau \in (0,T)$, depending continuously on $(\nu,\mu)$, such that 
for any $\omega_0 \in [T-\tau,T)\times\R$ and any $\delta_0>0$, there exist 
$\omega_1 \in \mathcal{B}(\omega_0,\delta_0)\cap Q^T$ and $\delta_1>0$ such that
\begin{equation*}
        \langle \mu, \dx\psi^\nu \rangle \neq 0 \text{ on } \mathcal{B}(\omega_1,\delta_1).
    \end{equation*} 
\end{lemma}

\begin{proof}
    Fix $\nu \in \R^M, \mu \in \mathbb{S}^{M-1}$, and set $g(x)=\mu \cdot G(x)$. By Assumption~\ref{assmpt:G}, there exist $x_1,x_2 \in \R$, such that $g(x_1)>g(x_2)$. Recall that $\Psi:=\langle \mu, \psi^\nu\rangle\in C^{2}_{loc}(Q^T)\cap C_b(\overline{Q}^T)$ solves the linear Cauchy problem~\eqref{eq:linearisation-1} associated with $(\phi^\nu,g)$. By continuity, there exists $\delta>0$, \begin{equation}\label{ineq:preuve-wucp}
        \sup_{(t,x)\in\mathcal{B}((T,x_2),\delta)} \Psi(t,x) < \inf_{(t,x)\in \mathcal{B}((T,x_1),\delta)} \Psi(t,x).
    \end{equation} By the stability result of Proposition~\ref{prop:put-linearization}, $\delta$ can be chosen continuously with respect to $(\nu,\mu)$. We set $\tau := \delta$. Suppose, for the sake of contradiction, the existence of $\omega_0 \in [T-\tau,T)\times \R,$ and $\delta_0 \in (0,T-t_0)$ such that $\dx\Psi = 0$ on $\mathcal{B}(\omega_0,\delta_0)$. Since $\Psi \in C^{2}_{loc}(Q^T)$, this also implies that $D_x^2\Psi = 0$ on  $\mathcal{B}(\omega_0,\delta_0)$. Using the linear PDE satisfied by $\Psi$, we then obtain $\partial_t \Psi = 0$ on $ \mathcal{B}(\omega_0,\delta_0)$. Hence, $\Psi$ is constant on the ball. Define the re-centered function $\wtilde{\Psi}:=\Psi-\Psi(\omega_0)$. Then $\wtilde{\Psi}$ also satisfies \begin{equation*}
        \dt \wtilde{\Psi}+b(t,x)\dx\wtilde{\Psi}+a(t,x)\demi \dxx \wtilde{\Psi}=0 \text{ on } Q^T,
    \end{equation*} where the coefficients belong to $C^1(\R)$ for each $t\in(0,T)$ by Corollary~\ref{coro:lip-local-coeff}, and $\wtilde{\Psi}$ vanishes on $\mathcal{B}(\omega_0,\delta_0)$. Moreover, $\wtilde{\Psi}$ verifies \eqref{ineq:preuve-wucp}. By the weak unique continuation property for second order parabolic operators of Theorem~1.1 in \cite{saut-1987-ucp}, since $\wtilde{\Psi}$ vanishes on the ball, it also vanishes on its horizontal component $B:=\{ (t,x)\in Q^T, \exists \wtilde{x}\in \R, (t,\wtilde{x})\in \mathcal{B}(\omega_0,\delta_0) \}.$ However, $$B\cap \mathcal{B}((T,x_1),\delta)\neq \emptyset \qquad\text{and}\qquad  B\cap \mathcal{B}((T,x_2),\delta)\neq \emptyset,$$ which contradicts \eqref{ineq:preuve-wucp}.
\end{proof}

\begin{proof}[Proof of Theorem \ref{thm:strict-cvx-flow}]
Consider a convex set $K \Subset \R^M$ and define \begin{equation*}\mathcal{K}:=\Big\{(x,y)\in K \times \R^M:y \in K\ominus \{x\} \Big\},\end{equation*} where $\ominus$ denotes the Minkowski difference.
    In view of Lemma \ref{lem:strict-convexity-onedimen}, it is sufficient to show that there exists $ \underline{\Lambda} >0$, such that for all $(\lambda,\mu) \in \mathcal{K},$\begin{equation*}
        \Phi^{\lambda,\mu}_0 : [0,1] \ni h \mapsto \phi^{\lambda + h \mu}(0,x_0) \in \R 
    \end{equation*} is strongly convex of modulus $\underline{\Lambda} ||\mu||^2$. We decompose the proof into three steps. First, using a limiting argument, we establish a lower bound for the difference quotient \begin{equation}\label{eq:proof-quotient}
        \frac{\partial_h \Phi^{\lambda,\mu}_0(h_2)-\partial_h \Phi^{\lambda,\mu}_0(h_1)}{h_2-h_1} .
    \end{equation} Next, we show that this lower bound is strictly positive for each $(\lambda,\mu)\in \mathcal{K}$. Finally, using compactness and spectral arguments, we conclude that there exists a strictly positive constant  $\underline{\Lambda}$, independent of $(\lambda,\mu)$, such that the lower bound is $\underline{\Lambda} ||\mu||^2$.\\
\textit{Step 1:} Let $(\lambda,\mu) \in \mathcal{K},$ for $\eta \in \N,$ we consider \begin{equation*}
    \Phi^{\lambda,\mu}_{0,\eta} : [0,1]  \ni h \mapsto \phi^{\lambda+h\mu}_\eta(0,x_0) \in \R.
\end{equation*}Let $1\geq h_2 > h_1 \geq 0,$ \begin{equation*}
        \partial_h \Phi^{\lambda,\mu}_{0,\eta}(h_2)-\partial_h \Phi^{\lambda,\mu}_{0,\eta}(h_1) = \int_{h_1}^{h_2} \partial^2_{hh} \Phi^{\lambda,\mu}_{0,\eta}(z)dz = \int_{h_1}^{h_2} \mu^\top \Gamma^\eta(\lambda + z\mu)\mu dz,
        \end{equation*} with $\Gamma^\eta \in \mathcal{S}_+^M$ the Hessian of $\Phi_{0,\eta}$ obtained by Proposition~\ref{prop:smooth-coef-linearization}, with $\Gamma_{i,j}^\eta(\nu) = \gamma_{i,j}^\nu(0,x_0)$ solution to \eqref{eq:linearisation-2}. For $\eta \in \N,\nu \in \R^M,$ we denote by $\left(X_t(\nu,\eta)\right)_{t\in[0,T]}$ the solution to the SDE~\eqref{eq:eds-nonlin} with coefficients $(a^\nu_\eta,b^\nu_\eta)$ associated with $\nu\cdot G_\eta$, starting from $X_0(\nu,\eta)=x_0$. For $\nu,\mu \in \R^M,$ the following Feynman--Kac formula holds, \begin{equation}\label{eq:hessian-matrix}
        \mu^T \Gamma^\eta (\nu)\mu = \esperance{\int_0^T \begin{pmatrix}
            \langle \mu, \dx \psi^{\nu,\eta} \rangle \\
            \langle \mu, \demi \dxx \psi^{\nu,\eta} \rangle 
        \end{pmatrix}^\top \Hessian{H}(\dx\phi^{\nu}_\eta,\demi\dxx\phi^\nu_\eta) \begin{pmatrix}
            \langle \mu, \dx \psi^{\nu,\eta} \rangle \\
            \langle \mu, \demi \dxx \psi^{\nu,\eta} \rangle 
        \end{pmatrix} (t,X_t(\nu,\eta))dt}.
    \end{equation} 
    Let $E\Subset \R$ be connected. We define $\forall (t,x) \in [0,T),$ \begin{align*}
        f_\eta(t,x;\nu,\mu,E) &:= \begin{pmatrix}
            \langle \mu, \dx \psi^{\nu}_\eta \rangle \\
            \langle \mu, \demi \dxx \psi^{\nu}_\eta \rangle 
        \end{pmatrix}^\top \Hessian{H}(\dx\phi^{\nu}_\eta,\demi\dxx\phi^\nu_\eta) \begin{pmatrix}
            \langle \mu, \dx \psi^{\nu}_\eta \rangle \\
            \langle \mu, \demi \dxx \psi^{\nu}_\eta \rangle 
        \end{pmatrix}(t,x)\, \chi_{E}(x), \\
        f(t,x;\nu,\mu,E) &:= \begin{pmatrix}
            \langle \mu, \dx \psi^{\nu} \rangle \\
            \langle \mu, \demi \dxx \psi^{\nu} \rangle 
        \end{pmatrix}^\top \Hessian{H}(\dx\phi^{\nu},\demi\dxx\phi^\nu) \begin{pmatrix}
            \langle \mu, \dx \psi^{\nu} \rangle \\
            \langle \mu, \demi \dxx \psi^{\nu} \rangle 
        \end{pmatrix}(t,x)\, \chi_{E}(x),
    \end{align*} where $\chi_E$ is a smooth cutoff function, Hausdorff-Lipschitz stable with respect to its support $E$. One may construct a family of smooth cutoff functions $(\chi_E)_E \subset C_c^\infty(\R)$ such that the map $E\mapsto \chi_E$ is Lipschitz continuous on the space of compact sets of $\R$ endowed with the Hausdorff distance to the normed space $C_b(\R)$. This can be achieved by mollifying Lipschitz cutoffs of the form $x\mapsto \eta(d(x,E))$, where $\eta:\R_+ \to [0,1]$ is a Lipschitz, non-increasing function satisfying $\eta(0)=1$ and $\eta(t)=0$ for $t\geq 1$. This ensures that the map $(x,E)\mapsto \chi_E(x)$ is jointly continuous. The functions $f_\eta$ and $f$ are continuous with respect to $(t,x,\nu,\mu,E)$, and non-negative. Moreover, by Corollary~\ref{coro:existence-convergence-linearization} for $\tau \in (0,T]$, there exists $C(\tau)>0$ such that, for any $\eta \in \N$, $\Norm{f_\eta}_{\infty;\overline{Q}^{T-\tau}}$, $\Norm{f}_{\infty;\overline{Q}^{T-\tau}} < C$. We obtain, $\forall \tau \in (0,T],$\begin{equation}\label{ineq:ineq-4}
        \partial_h \Phi^{\lambda,\mu}_{0,\eta}(h_2)-\partial_h \Phi^{\lambda,\mu}_{0,\eta}(h_1) \geq \int_{h_1}^{h_2} \esperance{\int_0^{T-\tau} f_\eta (t,X_t(\lambda + z\mu,\eta);\lambda + z\mu,\mu,E)dt}dz.
    \end{equation} For $\nu \in \R^M,$ we denote by $\left(X_t(\nu)\right)_{t\in[0,T]}$ the solution to the SDE (\ref{eq:eds-nonlin}) with coefficients $(a^\nu,b^\nu)$ associated with $\nu\cdot G$. We now prove that $\forall \nu,\mu \in \R^M, \forall \tau \in (0,T], \forall E \Subset\R$ connected,\begin{align}\label{eq:conv-subhess}
        \E \Bigg[ \int_0^{T-\tau} f_\eta(t,X_t(\nu,\eta);\nu,\mu,E)dt \Bigg] \xrightarrow[\eta\to \infty]{} \E \Bigg[ \int_0^{T-\tau} f(t,X_t(\nu);\nu,\mu,E)dt \Bigg].
    \end{align} For the sake of clarity, the dependence on $(\nu,\mu,E)$ is omitted in the following equations. We have \begin{equation*}\begin{split}
        \esperance{\int_0^{T-\tau} \left|f_\eta(t,X_t(\nu,\eta))-f(t,X_t(\nu))\right|dt} &\leq\,  \esperance{\int_0^{T-\tau} \left|f_\eta(t,X_t(\nu,\eta))-f(t,X_t(\nu,\eta))\right|dt} \\
         &+\, \esperance{\int_0^{T-\tau} \left|f(t,X_t(\nu,\eta))-f(t,X_t(\nu))\right|dt}.
    \end{split}
    \end{equation*} We denote by $N_1\in \N,$ $\forall \eta \geq N_1,$ the first term is smaller than $\frac{\veps}{2}$ by the stability results of Proposition~\ref{prop:stability-nonlin-lip-coeff-bis} and Corollary~\ref{coro:existence-convergence-linearization}. Denoting $Q^{T-\tau} = (0,T-\tau)\times \R,$ by continuity, $f\in \BUC(\overline{Q}^{T-\tau})$  with a modulus of continuity $\omega^\tau(\cdot).$ Let $\delta>0$, we remark that by Markov's inequality\begin{align*}
        \esperance{\int_0^{T-\tau} \left|f(t,X_t(\nu,\eta))-f(t,X_t(\nu))\right|dt} \leq  \Norm{f}_{\infty;\overline{Q}^{T-\tau}} \left( \frac{1}{\delta}\int_0^{T-\tau} \esperance{\left| X_{t}(\nu,\eta)-X_{t}(\nu) \right|}dt + T \omega^\tau(\delta) \right).
    \end{align*}
    For $t\in [0,T-\tau],$\begin{align*}
        \esperance{\left| X_{t}(\nu,\eta)-X_{t}(\nu) \right|^2}
        &\leq 4T\,\esperance{\int_0^{t} \left| b_\eta^\nu(s,X_{s}(\nu,\eta))-b^\nu(s,X_{s}(\nu,\eta))\right|^2ds } \\
        &+ 4T\, \esperance{\int_0^{t} \left| b^\nu(s,X_{s}(\nu,\eta))-b^\nu(s,X_{s}(\nu))\right|^2ds }\\
        &+4\, \esperance{\int_0^{t} \left| a_\eta^\nu(s,X_{s}(\nu,\eta))-a^\nu(s,X_{s}(\nu,\eta))\right|^2ds } \\
        &+4\, \esperance{\int_0^{t} \left| a^\nu(s,X_{s}(\nu,\eta))-a^\nu(s,X_{s}(\nu))\right|^2ds }.
    \end{align*} By the stability result of Proposition~\ref{prop:stability-nonlin-lip-coeff-bis}, $\exists C_1(\tau)>0,$ $\forall \eta \in \N,$\begin{equation*}
        \esperance{\int_0^{t} \left| b_\eta^\nu(s,X_s(\nu,\eta))-b^\nu(s,X_s(\nu,\eta))\right|^2ds } \leq C_1 \Norm{G^\eta-G}_\infty.
    \end{equation*} By Corollary~\ref{coro:lip-local-coeff}, $\exists C_2(\tau)>0,$ $\forall\eta\in \N,$\begin{equation*}
        \esperance{ \int_0^{t} \left| b^\nu(s,X_s(\nu,\eta))-b^\nu(s,X_s(\nu))\right|^2ds } \leq C_2 \esperance{ \int_0^{t} \left|X_s(\nu,\eta)-X_s(\nu)\right|^2ds}.
    \end{equation*} Applying the same arguments to the diffusion terms, by Grönwall's inequality and the convergence of $\Norm{G_\eta-G}_\infty$, we can choose $N_2\in \N $ such that, $ \forall \eta \geq N_2,$ \begin{align*}
        \esperance{\int_0^{T-\tau} \left|f(t,X_t(\nu,\eta))-f(t,X_t(\nu))\right|dt} < \frac{\veps}{2}.
    \end{align*} Hence $\forall \eta \geq N_1 \vee N_2,$\begin{align*}
         \E \Bigg[ \int_0^{T-\tau} \left | f_\eta(t,X_t(\nu,\eta);\nu,\mu,E)-f(t,X_t(\nu);\nu,\mu,E)\right|dt \Bigg] < \veps.
    \end{align*} By Proposition~\ref{prop:put-linearization}, we have $\partial_h \Phi_0^{\lambda,\mu}(h)=\psi^{\lambda+h\mu}(0,x_0)\cdot \mu$. Recalling that $\partial_h \Phi_{0,\eta}^{\lambda,\mu}(h)=\psi^{\lambda+h\mu,\eta}(0,x_0)\cdot \mu$, Corollary~\ref{coro:existence-convergence-linearization} implies that for all $h_2>h_1\in [0,1]$, \begin{equation*}
        \partial_h \Phi_{0,\eta}^{\lambda,\mu} (h_2)-\partial_h \Phi_{0,\eta}^{\lambda,\mu} (h_1) \xrightarrow[\eta\to \infty]{}  \partial_h \Phi_{0}^{\lambda,\mu} (h_2)-\partial_h \Phi_{0}^{\lambda,\mu} (h_1).
    \end{equation*} Applying Fatou's lemma together with \eqref{eq:conv-subhess} to inequality \eqref{ineq:ineq-4} yields that, for all $ \lambda,\mu \in \R^M, $ all $ h_2>h_1\in[0,1], $ all $  \tau \in (0,T],$ and all connected sets $ E\Subset \R$,\begin{equation*}
        \partial_h \Phi_{0}^{\lambda,\mu} (h_2)-\partial_h \Phi_{0}^{\lambda,\mu} (h_1) \geq \int_{h_1}^{h_2} \esperance{\int_0^{T-\tau} f(t,X_t(\lambda + z\mu);\lambda + z\mu,\mu,E) dt}dz.
    \end{equation*}
\textit{Step 2:} We prove that there exist continuous mappings $\mathbb{S}^{M-1}\times \R^M\ni(e,\nu)\mapsto \tau(e,\nu) \in (0,T),$ $\R^M \ni \nu \mapsto E(\nu)\Subset \R$, where $E(\nu)$ belongs to the space of nonempty convex subsets of $\R$ (endowed with the Hausdorff distance), such that for all $ \nu \in \R^M$, and $\mu \in \R^M_*$, denoting $\wtilde{\mu}:=\frac{\mu}{\Norm{\mu}}$, \begin{equation}\label{eq:step2-eeq}
\esperance{\int_0^{T-\tau(\wtilde{\mu},\nu)} f(t,X_t(\nu);\nu,\mu,E(\nu))\,dt}  >0 . \end{equation} By definition of $f$, this amounts to showing \begin{equation}\label{eq:cvx-proof-f}
   \int_0^{T-\tau(\wtilde{\mu},\nu)} \int_{\R} \underbrace{\begin{pmatrix}
            \langle \mu, \dx \psi^{\nu} \rangle \\
            \langle \mu, \demi \dxx \psi^{\nu} \rangle 
        \end{pmatrix}^\top}_{\text{(a)}} \underbrace{\Hessian{H}(\dx\phi^{\nu},\demi\dxx\phi^\nu)}_{\text{(b)}} \begin{pmatrix}
            \langle \mu, \dx \psi^{\nu} \rangle \\
            \langle \mu, \demi \dxx \psi^{\nu} \rangle 
        \end{pmatrix}(t,x)\, \underbrace{\chi_{E(\nu)}(x) P_{X(\nu)}(dt\times dx)}_{\text{(c)}}>0. 
\end{equation}Note that the integrand is nonnegative by convexity of $H$.\begin{enumerate}
    \item[(a)] By Lemma \ref{lem:wucp}, there exists a continuous mapping $(\wtilde{\mu},\nu)\mapsto\tau(\wtilde{\mu},\nu) \in (0,T),$ such that, for every nonempty open set $ O \subset \R,$ there exists $\omega_1 \in [T-\tau(\wtilde{\mu},\nu),T)\times O,\delta_1 >0,$ \begin{align*}
     \mu^\top \dx \psi^\nu  \neq 0 \text{ on } \mathcal{B}(\omega_1,\delta_1).
\end{align*}
    \item[(b)] By Assumption~\ref{assmpt:operatorH-cvx}, there exists $\Delta>0$ such that $\Hessian{H}$ is positive definite on $Q^T \times \mathcal{B}((0,0),\Delta)$. By Proposition~\ref{prop:exi-uni-edp-nonlin}, for every $t\in [T-\tau(\widetilde{\mu},\nu),T)$, the map $x\mapsto \phi^\nu(t,x)$ belongs to $C^2(\R)$. Hence, by Lemma~\ref{lem:grad-zero}, for all such $t$, \begin{equation}\label{eq:infimum_dxphi}
\inf_{x\in\R} |\dx \phi^\nu(t,x)| = 0.
\end{equation} Choose $t_0\in[T-\tau(\widetilde{\mu},\nu),T)$, define $Q_0(\nu) = (t_0-\frac{T-t_0}{2},t_0+\frac{T-t_0}{2})\times \R$ and set $v^\nu = \dx \phi^\nu$. By Proposition~\ref{prop:exi-uni-edp-nonlin}, there exists $\alpha_0\in(0,1)$ such that $\phi^\nu \in C^{2+\alpha_0}(\overline{Q}_0(\nu))$. Moreover by Proposition~\ref{prop:stability-nonlin-lip-coeff-bis} $v^\nu \in C^{2+\alpha_0}_{loc}(Q_0(\nu))$ and satisfies on $Q_0(\nu)$ a linear uniformly parabolic equation of the form \begin{equation*}
\partial_t v^\nu
+ b^\nu\, \partial_{x} v^\nu
+ \frac12 a^\nu\, \dxx v^\nu
+ h^\nu =0,
\end{equation*}where $(a,b)$ are coefficients defined in Corollary~\ref{coro:lip-local-coeff} associated with $\phi^\nu$ and \begin{equation*}
h^\nu(t,x):=D_x H\Bigl(t,x,\dx\phi^\nu(t,x),\frac12 \dxx\phi^\nu(t,x)\Bigr).
\end{equation*}By Corollary~\ref{coro:lip-local-coeff}, the coefficients $(a^\nu,b^\nu)$ belong to $C^{\alpha_0}(Q_0(\nu))$, and $a^\nu$ is uniformly elliptic. Therefore, by the interior Schauder estimate of Theorem~4.9 in \cite{lieberman}, there exists a constant \begin{equation*}
C \coloneq C\Bigl(\mathfrak{v},\|a^\nu\|_{\alpha_0;Q_0(\nu)},\|b^\nu\|_{\alpha_0;Q_0(\nu)}\Bigr)>0
\end{equation*} such that, for every parabolic cylinder $Q^I_\Omega(\nu) = I(\nu) \times \Omega(\nu) \subset Q_0(\nu)$, with $\Omega \subset \R$ a bounded connected open set, one has \begin{equation}\label{ineq:schauder-proof-interior}
    \Norm{\dx v^\nu}_{\infty;Q^I_\Omega(\nu)} \leq C\left( \Norm{v^\nu}_{\infty;Q^I_\Omega(\nu)} + \Norm{h^\nu}_{\alpha_0;Q^I_\Omega(\nu)}^{(2)} \right).
\end{equation} By the first statement~\eqref{eq:infimum_dxphi} and continuity of $v^\nu$ on $Q^T$, we can fix $x_0\in \R$ such that there exists an open interval $I(\nu)$ and a bounded connected open set $\Omega(\nu)\subset \R$ such that $(t_0,x_0)\in Q^I_\Omega(\nu)\coloneq I(\nu)\times \Omega(\nu) \subset Q_0(\nu)$ and on $Q^I_\Omega(\nu)$\begin{equation}\label{eq:bound-proof-stepb}
    \left| v^\nu \right| < \frac{\Delta}{4C}. 
\end{equation}
Since $h^\nu \in C^{\alpha_0}(Q_0(\nu))$, by Remark~\ref{rk:holder-weighted} the weighted Hölder norm of $h^\nu$ in inequality~\eqref{ineq:schauder-proof-interior} can be made arbitrarily small by choosing $Q^I_\Omega(\nu)$ with sufficiently small diameter. Hence, shrinking $Q^I_\Omega(\nu)$ around $(t_0,x_0)$ if necessary, we may assume that \begin{equation*}
    \Norm{h^\nu}_{\alpha_0;Q^I_\Omega(\nu)}^{(2)} < \frac{\Delta}{4C}.
\end{equation*} Combining this with the bound~\eqref{eq:bound-proof-stepb}  in inequality~\eqref{ineq:schauder-proof-interior}, we obtain \begin{equation*}
    \Norm{\dx v^\nu}_{\infty;Q^I_\Omega(\nu)} < \frac{\Delta}{2}.
\end{equation*} It follows that
\begin{equation*}
\Bigl(\dx\phi^\nu,\frac12 \dxx\phi^\nu\Bigr)(t,x)\in \mathcal{B}((0,0),\Delta),
\qquad \forall (t,x)\in Q^I_\Omega(\nu).
\end{equation*} In particular, on $Q^I_\Omega(\nu)$, the Hessian of $H$ remains positive definite. Moreover, by the stability result of Proposition~\ref{prop:stability-nonlin-lip-coeff-bis}, the maps $\nu \mapsto \dx\phi^\nu$ and $\nu \mapsto \dxx\phi^\nu$ depend continuously on $\nu$ in the relevant Hölder topologies on compact subsets. Hence, once a continuous choice of centers and radii is fixed, the associated family of intervals $I(\nu)$ and open connected bounded sets $\Omega(\nu)$ depends continuously on $\nu$ with respect to the Hausdorff distance. We set $E(\nu)=\overline{\Omega}(\nu)$.  
\item[(c)]
Consider the set \begin{equation*}
    \varPhi := \left\{ \varphi \in C_{x_0}([0,T];\R) : \varphi(t) \in E(\nu), \, \forall t \in [T-\tau(\widetilde{\mu},\nu),T)  \right\}
\end{equation*} with $C_{x_0}([0,T];\R) := \left\{ \varphi \in C([0,T];\R): \varphi(0)=x_0 \right\}$. By the Stroock--Varadhan~support~theorem, Theorem~3.1 in \cite{stroock-varadhan-support}, the law
\(\mathbb P_X\) of \(X\) has full support on
\(C_{x_0}([0,T];\R)\). Hence,
\begin{equation*}
\varPhi\subseteq\operatorname{Supp}(\mathbb P_X).
\end{equation*}
Since $E(\nu)$ has non-empty interior, we obtain that \(\varPhi\) has non-empty interior in
\(C_{x_0}([0,T];\R)\), endowed with the uniform topology. It follows
from the definition of the support that
\begin{equation*}
\mathbb P_{X(\nu)}(\varPhi)>0.
\end{equation*}
\end{enumerate}
Combining (a)–(c) yields \eqref{eq:step2-eeq}, hence the difference quotient~\eqref{eq:proof-quotient} is strictly positive.\\
\textit{Step 3:} We have that for all $e\in \mathbb S^{M-1}$ and $\nu\in\mathbb R^M$, 
$\tau(e,\nu)\in (0,T)$ and $E(\nu)\Subset\mathbb R$ is a nonempty convex set. 
Moreover, these maps are continuous by construction. Define
\begin{equation*}
\underline{\tau}
:=\min_{(e,\nu)\in\mathbb S^{M-1}\times K} \tau(e,\nu),
\qquad\text{and}\qquad \overline{E}
:=\mathrm{conv}\Bigl(\bigcup_{\nu\in K} E(\nu)\Bigr).
\end{equation*} By compactness and continuity we have $\underline{\tau}>0$. 
Fix $\tau_0\in(0,\underline{\tau})$. 
Applying Lemma~\ref{lem:compact_convex_set}, the set $\overline{E}$ is compact, convex and non-empty. 
Hence, for all $\nu\in K$ and all $\mu\in\mathbb R^M_*$, \begin{align*}
    \esperance{\int_0^{T-\tau_0} f(t,X_t(\nu);\nu,\mu,\overline{E}) dt} >0.
\end{align*} By construction, the mapping
\[
q(\cdot;\nu):\mu \longmapsto 
\mathbb E\!\left[\int_0^{T-\tau_0} 
f(t,X_t(\nu);\nu,\mu,\overline{E})\,dt\right]
\]
is a positive semidefinite quadratic form in $\mu$. By compactness of $\mathbb{S}^{M-1}$, we have \begin{equation*}
    \Lambda(\nu):=\min_{\mu \in \mathbb{S}^{M-1}} q(\mu;\nu) >0.
\end{equation*} This $\Lambda(\nu)$ is the smallest eigenvalue of the
self-adjoint operator $B(\nu)$ associated with $q(\cdot;\nu)$ through the Riesz
representation theorem. By the Rayleigh characterization of the smallest eigenvalue and the continuity of $\nu\mapsto B(\nu)$, the map $\nu\mapsto\Lambda(\nu)$
is continuous. Moreover, for every $\nu\in K$ we have $\Lambda(\nu)>0$. By compactness of $K$ and continuity of $\Lambda$, it follows that $$\underline{\Lambda}:=\min_{\nu\in K}\Lambda(\nu)>0.$$ Consequently, for all $(\lambda,\mu)\in\mathcal K$ and all $0\le h_1<h_2\le 1$,\begin{equation*}
    \partial_h \Phi^{\lambda,\mu}_{0}(h_2)-\partial_h \Phi^{\lambda,\mu}_{0}(h_1) > (h_2-h_1) \underline{\Lambda} \Norm{\mu}^2.
\end{equation*}
This establishes the local strong convexity of the operator by Lemma~\ref{lem:strict-convexity-onedimen}.
\end{proof}

\subsection{A log-martingale example beyond strong Hamiltonian convexity}
Assumption~\ref{assmpt:operatorH-cvx} is a convenient sufficient condition, but
it is too restrictive for some operators that arise naturally in calibration
problems. This is already visible under the log-martingale constraint, where
the drift coefficient is constrained to be $-\demi$ times the diffusion
coefficient. At the level of the cost function, this hard constraint restricts
the effective domain to a one-dimensional affine subspace. Consequently, the
corresponding Hamiltonian depends only on the combination $M-p/2$ and cannot be
strictly convex in the full $(p,M)$-variables. Hence
Assumption~\ref{assmpt:operatorH-cvx} fails. Nevertheless, the example below
shows that the strict positivity required in \eqref{eq:equation-convexity-nn}
can still be recovered by exploiting the specific structure of the
log-martingale constraint and of the terminal data. We take terminal data of the form
$G_i(x)=(K_i-\exp(x))^+$, corresponding to log-put payoffs, and assume that
the strikes are distinct, namely $K_i\neq K_j$ for $i\neq j$. These terminal
data satisfy Assumption~\ref{assmpt:G}. 
Consider the following cost function, where
$0<\underline a<a_0<\overline a<\infty$:
\begin{equation*}
    C(t,x,b,a)
    =
    \left\{
    \begin{array}{ll}
        (a-a_0)^2
        & \text{if } b=-\tfrac a2 \text{ and }
        a\in[\underline a,\overline a], \\[0.3em]
        +\infty
        & \text{otherwise.}
    \end{array}
    \right.
\end{equation*} The corresponding Hamiltonian is
\begin{equation*}
    H(p,M)
    =
    \sup
    \left\{
        a\left(M-\frac p2\right)-(a-a_0)^2 : a\in[\underline a,\overline a]
    \right\}.
\end{equation*} Strictly speaking, the hard-constrained Hamiltonian may fail to be $C^2$ at the
saturation thresholds; one should therefore replace it by a smooth convex
regularization preserving the dependence on $M-p/2$. The mechanism described
below is unchanged for such a regularized Hamiltonian. The following computation is understood on the open region where the maximizer
lies in the interior of $[\underline a,\overline a]$. We denote this region by
\begin{equation*}
    E
    :=
    \left\{
        (p,M):
        \underline{a}
        <
        a_0+\frac12\left(M-\frac p2\right)
        <
        \overline{a}
    \right\}.
\end{equation*} In the interior active region the maximizer is
\begin{equation*}
    a^*(p,M)
    =
    a_0+\frac12\left(M-\frac p2\right),
\end{equation*}
and therefore
\begin{equation*}
    \Hessian H(p,M)
    =
    \begin{pmatrix}
        \tfrac18 & -\tfrac14 \\
        -\tfrac14 & \tfrac12
    \end{pmatrix}
    \qquad \text{on } E .
\end{equation*}
The spectrum of this matrix is $\{0,\tfrac58\}$, with kernel generated by
$\begin{pmatrix}
    2 \\ 1
\end{pmatrix}$. Hence the Hamiltonian is not strictly convex in the full
$(p,M)$-variables, and Assumption~\ref{assmpt:operatorH-cvx} fails. Nevertheless, we claim that the quantity appearing in the convexity argument is
not identically zero, namely
\begin{equation}\label{eq:equation-convexity-nn}
    \begin{pmatrix}
        \langle \mu,\partial_x\psi^\nu\rangle \\
        \langle \mu,\demi\partial_{xx}\psi^\nu\rangle
    \end{pmatrix}^{\top}
    \Hessian H
    \left(
        \partial_x\phi^\nu,
        \demi\partial_{xx}\phi^\nu
    \right)
    \begin{pmatrix}
        \langle \mu,\partial_x\psi^\nu\rangle \\
        \langle \mu,\demi\partial_{xx}\psi^\nu\rangle
    \end{pmatrix}
    \not\equiv 0
    \quad \text{on } Q^T .
\end{equation}
Assume, by contradiction, that this quantity is identically equal to zero on
$Q^T$, and set
\begin{equation*}
    E_Q
    :=
    \left\{
        (t,x)\in Q^T:
        \left(
            \partial_x\phi^\nu,
            \demi\partial_{xx}\phi^\nu
        \right)(t,x)
        \in E
    \right\}.
\end{equation*}
On $E_Q$, the vector
\begin{equation*}
    \begin{pmatrix}
        \langle \mu,\partial_x\psi^\nu\rangle \\
        \langle \mu,\demi\partial_{xx}\psi^\nu\rangle
    \end{pmatrix}
\end{equation*}
must belong to the kernel of $\Hessian H$. Hence there exists a scalar function
$\lambda$ such that
\begin{equation}\label{eq:eigen-vector}
    \begin{pmatrix}
        \langle \mu,\partial_x\psi^\nu\rangle \\
        \langle \mu,\demi\partial_{xx}\psi^\nu\rangle
    \end{pmatrix}
    =
    \lambda
    \begin{pmatrix}
        2 \\ 1
    \end{pmatrix}
    \qquad \text{on } E_Q .
\end{equation}
Equivalently, if $u:=\langle \mu,\psi^\nu\rangle$, then
\begin{equation*}
    \partial_{xx}u=\partial_xu
    \qquad \text{on } E_Q .
\end{equation*} Moreover, $u$ solves the linearized equation
\begin{equation*}
    \left\{
    \begin{array}{ll}
        \partial_t u
        +\partial_p H
        \left(
            \partial_x\phi^\nu,
            \demi\partial_{xx}\phi^\nu
        \right)
        \partial_x u
        +\demi\partial_M H
        \left(
            \partial_x\phi^\nu,
            \demi\partial_{xx}\phi^\nu
        \right)
        \partial_{xx}u
        =0,
        & \text{on } Q^T, \\[0.3em]
        u(T,\cdot)=\langle \mu,G\rangle,
        &
    \end{array}
    \right.
\end{equation*}
and the log-martingale structure gives
\begin{equation*}
    -2\partial_p H
    \left(
        \partial_x\phi^\nu,
        \demi\partial_{xx}\phi^\nu
    \right)
    =
    \partial_M H
    \left(
        \partial_x\phi^\nu,
        \demi\partial_{xx}\phi^\nu
    \right)
    \qquad \text{on } Q^T .
\end{equation*}
Combining this identity with $\partial_{xx}u=\partial_xu$ in the linearized
equation yields
\begin{equation*}
    \partial_t u=0
    \qquad \text{on } E_Q .
\end{equation*}

        By the local argument used at the beginning of Step (b) in the proof of Theorem~\ref{thm:strict-cvx-flow},
        there exists $\tau>0$ such that, for every $t_0\in[T-\tau,T)$, the active
        region $E_Q$ intersects the time slice $\{t_0\}\times\mathbb R$. Hence, for each
        such $t_0$, we may choose $x_0\in\mathbb R$ such that
        $(t_0,x_0)\in E_Q$. Since $E_Q$ is open in $Q^T$, there exist $\eta>0$ and an
        open interval $I\subset\mathbb R$, with $x_0\in I$, such that
        \begin{equation*}
            B:=(t_0-\eta,t_0+\eta)\times I \subset E_Q .
        \end{equation*}
        Here $\eta$ is chosen small enough so that $B\subset Q^T$.
        On $E_Q$, the contradiction assumption implies
        \begin{equation*}
            \partial_{xx}u=\partial_x u .
        \end{equation*}
        In particular, on the interval $I$ at time $t_0$,
        \begin{equation*}
            \partial_{xx}u(t_0,\cdot)=\partial_xu(t_0,\cdot).
        \end{equation*}
        Therefore there exist $\alpha,\beta\in\mathbb R$ such that
        \begin{equation*}
            u(t_0,x)=\alpha e^x+\beta,
            \qquad x\in I .
        \end{equation*} Since $\partial_t u=0$ on $B$, we have $u(t,x)=\alpha e^x+\beta$ on $B$. Set $w(t,x):=u(t,x)-\alpha\exp(x)-\beta$. We have that $w$ is a solution to $\dt w + \partial_p H(\dx \phi^\nu, \demi \dxx \phi^\nu) \dx w + \demi \partial_M H(\dx \phi^\nu, \demi \dxx \phi^\nu) \dxx w=0, $ on $Q^T$. Moreover, $w=0$ on the nonempty cylinder $B$. By the weak unique continuation property for second-order parabolic operators,
Theorem~1.1 in~\cite{saut-1987-ucp}, it follows that $w$ vanishes on the
horizontal component containing $B$. In particular,
\begin{equation*}
    w(t,x)=0,
    \qquad (t,x)\in(t_0-\eta,t_0+\eta)\times\mathbb R .
\end{equation*} Since the strikes are distinct and $\mu\neq0$, we can choose three points
        $x_1,x_2,x_3$ away from the strikes such that
        $(u(T,x_i))_{i=1}^3$ does not belong to the two-dimensional affine subspace
        \[
            \left\{
                (\alpha e^{x_i}+\beta)_{i=1}^3:\ \alpha,\beta\in\R
            \right\}
            \subset\R^3.
        \]
        Since this subspace is closed, there exists $\varepsilon>0$ such that
        \begin{equation*}
            \max_{i=1,2,3}
            \left|\alpha e^{x_i}+\beta-u(T,x_i)\right|
            \geq \varepsilon,
            \qquad \forall \alpha,\beta\in\R.
        \end{equation*}
        On the other hand, the preceding argument yields, for times $t<T$ arbitrarily
        close to $T$, constants $\alpha_t,\beta_t\in\R$ such that
        \begin{equation*}
            u(t,x_i)=\alpha_te^{x_i}+\beta_t,
            \qquad i=1,2,3.
        \end{equation*}
        By continuity of $u$ on $\overline Q^T$,
        \begin{equation*}
            \max_{i=1,2,3}
            |u(t,x_i)-u(T,x_i)|
            \longrightarrow0
            \quad\text{as }t\uparrow T.
        \end{equation*}
        For $t$ close enough to $T$, this contradicts the previous lower bound with
        $\alpha=\alpha_t$ and $\beta=\beta_t$.

\begin{assumption}[Local non-degeneracy of the Hamiltonian quadratic form]
\label{assmpt:local-nondeg-hamiltonian}
Assume that $H$ is convex with respect to $(p,M)$. For every
$\nu\in\mathbb R^M$ and every $\mu\in\mathbb S^{M-1}$, there exists a nonempty
parabolic cylinder $Q_{\nu,\mu}\Subset Q^T$ such that
\begin{equation*}
    \begin{pmatrix}
        \langle \mu,\partial_x\psi^\nu\rangle \\
        \langle \mu,\demi\partial_{xx}\psi^\nu\rangle
    \end{pmatrix}^{\top}
    \Hessian H
    \left(
        \partial_x\phi^\nu,
        \demi\partial_{xx}\phi^\nu
    \right)
    \begin{pmatrix}
        \langle \mu,\partial_x\psi^\nu\rangle \\
        \langle \mu,\demi\partial_{xx}\psi^\nu\rangle
    \end{pmatrix}
    >0
    \qquad \text{on } Q_{\nu,\mu}.
\end{equation*}
\end{assumption}

\begin{remark}
Assumption~\ref{assmpt:operatorH-cvx} is therefore not the intrinsic condition
needed in the proof of Theorem~\ref{thm:strict-cvx-flow}. What is used is the
local non-degeneracy of the Hamiltonian quadratic form along the linearized
directions. More precisely, Assumption~\ref{assmpt:operatorH-cvx} may be
replaced by Assumption~\ref{assmpt:local-nondeg-hamiltonian}. Indeed, under
this weaker assumption, the integrand in the Feynman--Kac representation of
the second variation is nonnegative and is strictly positive on a nonempty
space-time cylinder. Since the diffusion associated with the linearized
equation is uniformly elliptic, this cylinder is reached with positive
occupation probability. Hence, for every $\nu\in\mathbb R^M$ and every
$\mu\in\mathbb S^{M-1}$,
\begin{equation*}
    \mathbb E
    \left[
        \int_0^T
        \begin{pmatrix}
            \langle \mu,\partial_x\psi^\nu\rangle \\
            \langle \mu,\demi\partial_{xx}\psi^\nu\rangle
        \end{pmatrix}^{\top}
        \Hessian H
        \left(
            \partial_x\phi^\nu,
            \demi\partial_{xx}\phi^\nu
        \right)
        \begin{pmatrix}
            \langle \mu,\partial_x\psi^\nu\rangle \\
            \langle \mu,\demi\partial_{xx}\psi^\nu\rangle
        \end{pmatrix}
        (t,X_t^\nu)
        \,dt
    \right]
    >0.
\end{equation*}
The compactness argument used in the proof of
Theorem~\ref{thm:strict-cvx-flow} then yields a uniform positive lower bound
on every compact subset of $\mathbb R^M$. Consequently, the conclusion of
Theorem~\ref{thm:strict-cvx-flow} remains valid under
Assumption~\ref{assmpt:local-nondeg-hamiltonian}.
\end{remark}

\section{Extension to the cascade of solutions}\label{sec:gluing}
In this section, we extend the differentiability and convexity
results of Propositions~\ref{prop:put-linearization}
and~\ref{prop:convexity-solution-map} and
Theorem~\ref{thm:strict-cvx-flow} to the cascade
problem~\eqref{eq:hjb-mult}. In principle, the number of terminal functions may depend on the maturity, say $M=M(p)$ at time $\tau_p$. To keep the notation light, we assume throughout this section that the same number $M$ of functions is attached to each maturity; the extension to maturity-dependent families is purely notational. The Dirac masses in the PDE~\eqref{eq:hjb-mult} encode jumps of the solution at times $\tau_p$, associated with $(G_{i(p,m)})_{m\leq M}$. More precisely, \begin{equation*}
   \lim_{t\uparrow \tau_p} \phi^\lambda(t ,\cdot) = \phi^\lambda(\tau_p ,\cdot) + \sum_{m=1}^M \lambda_{i(p,m)} G_{i(p,m)}(\cdot).
\end{equation*} We glue together classical solutions across the intermediate time slices as follows: recalling that $i(p,0)= (p-1)M$, and considering $\lambda \in \R^{P M}$, for $p\in\{1,\dots,P\}$, let $\lambda^{>i(p,0)}:=(\lambda_k)_{k>(p-1)M}\in \R^{(P-p+1)M}$ and $\phi^{\lambda^{>i(p,0)}}_p$ (when no ambiguity arises we may simply write $\phi^{\lambda}_p$) denote the solution to the following non-linear parabolic equation:\begin{align}\label{eq:hjb-time-cascade}\left\{ \begin{array}{ll}\displaystyle
     \dt \phi^\lambda_p + H(t,x,\Dx \phi^\lambda_p, \demi \Dxx \phi^\lambda_p) = 0 \text{ on } (\tau_{p-1},\tau_p)\times \R^d,  \\ \displaystyle
     \phi^\lambda_p(\tau_p,\cdot) = \phi^\lambda_{p+1}(\tau_{p},\cdot) + \sum_{m=1}^M \lambda_{i(p,m)} G_{i(p,m)}(\cdot),
\end{array}\right.
\end{align} with $\phi^\lambda_{P+1} \equiv 0$ and $\tau_0=0$. Since the solution $\phi^\lambda$ of the problem~\eqref{eq:hjb-mult} coincides with  $\phi^\lambda_p$ on each interval
$[\tau_{p-1},\tau_p)$, it suffices to study the family $(\phi^\lambda_p)_{p=1,\dots,P}$.

The extension is proved by a backward induction on the time index $p$. The base case corresponds to the simple problem studied in the previous sections, applied on the last time interval $[\tau_{P-1},\tau_P)$. For this section, we adapt the notation: for $p\in\{1,\dots,P\}$, $Q^p=(\tau_{p-1},\tau_p)\times\R$. Fixing $G\in \BUC(\R;\R^{P M})$, we introduce the solution map (omitting the dependence on $G$ from the notation)\begin{equation}\label{eq:solution-map-gluing}
    \Phi^p : \R^{(P-p+1)M} \to C_b(\overline{Q}^p), \quad \lambda \mapsto \phi^\lambda_p,
\end{equation} where $\phi^\lambda_p$ is the solution of~\eqref{eq:hjb-time-cascade} associated with data $(G_{k})_{k>i(p,0)} \in \BUC(\R;\R^{(P-p+1)M})$.

First, observe that the results obtained in Section~\ref{sec:estimates} rely only on the terminal condition being in $\BUC(\R)$. 
Assume that $H$ satisfies Assumption~\ref{assmpt:operatorH}. The boundedness
and uniform continuity of the terminal traces follow by backward induction on
$p$. If $x\mapsto\phi^\lambda_{p+1}(\tau_p,x)$ is bounded and uniformly
continuous, then the terminal datum on $[\tau_{p-1},\tau_p)$ belongs to
$\BUC(\R)$, so Proposition~\ref{prop:exi-uni-edp-nonlin} applies and yields a
bounded Lipschitz trace $x\mapsto\phi^\lambda_p(\tau_{p-1},x)$. Starting from
the last interval, where the future contribution is zero, this proves the claim.
Consider a compact $K\Subset \R^{P M}$. Let $\Delta_\tau = \bigwedge_{p=1}^P (\tau_p-\tau_{p-1})$, for $\tau\in (0,\Delta_\tau]$, by Proposition~\ref{prop:exi-uni-edp-nonlin}, there exists $\alpha(\tau,K) \in (0,1)$ and a constant $C(\tau,K)>0$, such that for all $p\in\{1,\dots,P\}$, $\lambda \in K$,\begin{equation*}
    \Norm{\phi_p^{\lambda}}_{2+\alpha(\tau,K);\overline{Q}^{p;\tau}} \leq C(\tau,K),
\end{equation*} with $Q^{p;\tau} = (\tau_{p-1},\tau_p-\tau)\times \R$. 

\begin{prop}\label{prop:put-linearization-gluing}
Assume $H$ satisfies Assumption~\ref{assmpt:operatorH}. Then, for every $p\in\{1,\dots,P\}$, the solution map $\Phi^p:\R^{(P-p+1)M}\to C_b(\overline{Q}^p)$ is continuously differentiable. Moreover, for
$\lambda\in\mathbb{R}^{(P-p+1)M}$ and $k\in\{1,\dots,(P-p+1)M\}$,
\begin{equation}\label{eq:grad-Phij}
\nabla \Phi^p(\lambda)_k
=
\psi^{\Phi^p(\lambda),\, g_k^p}
\quad \text{on } Q^p,
\end{equation}
where, denoting $\lambda^{>M} := (\lambda_k)_{k>M}\in\mathbb{R}^{(P-p)M}$, \begin{align*}
            g_k^p(\cdot) = \left\{ 
            \begin{array}{ll}
                 G_{i(p,k)}(\cdot), &\text{ if } k \in \{1,\dots,M\}, \\
                 \left(\nabla \Phi^{p+1}(\lambda^{>M})(\tau_{p},\cdot)\right)_{k-M}, &  \text{ otherwise. }
            \end{array}
            \right.
    \end{align*}
\end{prop}

\begin{proof}
We proceed by backward induction on $p$ to show that, for all
$p\in\{1,\dots,P\}$, the map $\Phi^p$ is continuously differentiable,
its gradient is given by~\eqref{eq:grad-Phij} on $Q^p$, and
$g_k^p\in \BUC(\R)$ for all $k$.\\
For $p=P$, the result holds by the assumption $G\in \BUC(\R)^M$ and Proposition~\ref{prop:put-linearization}. Assume that the statement holds for $p+1$. Let us show it for $p$. We first prove that $g_k^p\in \BUC(\R)$. If $k\in\{1,\dots,M\}$, then $g_k^p = G_{i(p,k)}\in \BUC(\R)$. Let $k>M$. By Corollary~\ref{coro:existence-convergence-linearization},
there exist $\alpha_\eta\in(0,1]$ and
$\psi_k^\eta\in C^{2+\alpha_\eta}(\overline{Q}^{p+1})$
such that
\[
\|\psi_k^\eta -
\psi^{\Phi^{p+1}(\lambda^{>M}),g^{p+1}_k}\|_{\infty;\overline{Q}^{p+1}}
\longrightarrow 0
\quad \text{as } \eta\to\infty.
\]
Hence $\bigl(\nabla\Phi^{p+1}(\lambda^{>M})(\tau_p,\cdot)\bigr)_{k-M}
\in \BUC(\R)$, which proves that $g_k^p\in \BUC(\R)$.
It remains to prove that $\Phi^p$ is continuously differentiable and
that its gradient is given by $\psi^{\Phi^p(\lambda),g^p}$ on $Q^p$.
We follow the same strategy as in the proof of
Proposition~\ref{prop:put-linearization}: we mollify the terminal data of 
\eqref{eq:hjb-time-cascade}, apply
Proposition~\ref{prop:smooth-coef-linearization}, and pass to the limit
using compactness arguments and the Arzelà--Ascoli theorem.

Let $(\Phi^\eta(\lambda))_{\eta\in\mathbb{N}}$ denote the sequence
associated with the mollified terminal data. The original terminal datum belongs to $\BUC(\R)$, ensuring that the mollified terminal data converge uniformly to that of~\eqref{eq:hjb-time-cascade}. By Proposition~\ref{prop:smooth-coef-linearization}, we have
\[
\Phi^\eta \in C^1(\mathbb{R}^{(P-p+1)M};C_b(\overline{Q}^p)), \quad \forall \eta \in \N.
\] Indeed, in order to apply Proposition~\ref{prop:smooth-coef-linearization} with a chain rule argument, it only remains to verify that \begin{equation*}
    T_\eta^p:\mu \mapsto \int \langle g^p(z), \mu \rangle \, \veps_\eta(\cdot-z)dz
\end{equation*} is a continuous linear map from $\R^{(P-p+1)M}$ to $C^{2+1}(\R)$, where $(\veps_\eta)_\eta$ is a sequence of mollifiers at least in $C^3(\R)$. 
By the induction hypothesis, $g^p$ is bounded. Therefore, for each fixed
$\eta$, differentiation under the integral is justified by dominated convergence:
for every $j\leq 3$,
\[
    \partial_x^j \int_{\mathbb R}
    \langle g^p(z),\mu\rangle\,\varepsilon_\eta(x-z)\,dz
    =
    \int_{\mathbb R}
    \langle g^p(z),\mu\rangle\,\partial_x^j\varepsilon_\eta(x-z)\,dz.
\]
Since $\partial_x^j\varepsilon_\eta\in L^1(\mathbb R)$, this gives for $C_\eta := \sum_{j=0}^3 \|\partial_x^j\varepsilon_\eta\|_{L^1(\mathbb R)}$,
\[
    \|T_\eta^p(\mu)\|_{C^{2+1}(\mathbb R)}
    \leq C_\eta \|g^p\|_\infty |\mu|.
\]
Thus $T_\eta^p$ is a continuous linear map from
$\mathbb R^{(P-p+1)M}$ to $C^{2+1}(\mathbb R)$, and is therefore Fréchet differentiable.
We show that for every compact
$K\Subset\mathbb{R}^{(P-p+1)M}$ the sequence
$(\Phi^\eta)_{\eta\in\mathbb{N}}$ is uniformly bounded and
equicontinuous on $K$ with respect to the
$C^1(K,C_b(\overline{Q}^p))$ norm, as similarly defined in~\eqref{eq:norm-c1-proof}. Pointwise relative compactness follows from the same arguments as in the proof of Proposition~\ref{prop:put-linearization}, since the
terminal data of $\Phi^\eta(\lambda)$ and $\nabla\Phi^\eta(\lambda)$ are mollifications
of uniformly continuous and bounded functions on $\R$ and Proposition~\ref{prop:comp-principle-nonlinear} and Corollary~\ref{coro:existence-convergence-linearization} apply. To prove equicontinuity, we need to show that for every $\veps>0$ there exists $\delta >0$ such that for all $\mu \in \mathcal{B}(\lambda,\delta)$ and all $\eta\in\N$, \begin{equation*}
        \Norm{\nabla \Phi^\eta(\lambda) - \nabla\Phi^\eta(\mu)}_{\cL(\R^{(P-p+1)M},C_b(\overline{Q}^p))} < \veps.
    \end{equation*} Fix $(t,x)\in\overline{Q}^p$.
For $k\le M$ the argument is identical to the one used in
Proposition~\ref{prop:put-linearization}, since we have $G^\eta_{i(p,k)}$ for terminal data.
For $\nu\in\R^{(P-p+1)M}$, we denote by $(X_s^{t,x}(\nu,\eta))_{s\in[t,\tau_p]}$ the solution to the SDE~\eqref{eq:eds-nonlin} on $Q^p$ associated with $\Phi^\eta(\nu)$, starting from $X_t^{t,x}(\nu,\eta)=x$. Applying the Feynman--Kac formula to $\nabla\Phi^\eta(\lambda)$ and $\nabla\Phi^\eta(\mu)$, 
    and denoting $\nabla \Phi^{p+1}_{\nu,k;\eta}:=\nabla\Phi^{p+1}_\eta(\nu^{>M})_{k-M}$, we obtain that for every $k\in\{M+1,\dots,(P-p+1)M\}$, \begin{equation}\label{ineq:ineq-6}\begin{split}
        \left| \nabla \Phi^\eta_k(\lambda)(t,x)-\nabla \Phi^\eta_k(\mu)(t,x) \right|  \leq\, 
        &\esperance{\left| \nabla \Phi^{p+1}_{\lambda,k;\eta}(\tau_p,X_{\tau_p}^{t,x}(\lambda,\eta)) - \nabla \Phi^{p+1}_{\mu,k;\eta}(\tau_p,X_{\tau_p}^{t,x}(\mu,\eta)) \right|}.
        \end{split}
    \end{equation} 
    Moreover, {\small\begin{equation}\label{ineq:ineq-7}
    \begin{split}
        \esperance{\left| \nabla \Phi^{p+1}_{\lambda,k;\eta}(\tau_p,X_{\tau_p}^{t,x}(\lambda,\eta)) - \nabla \Phi^{p+1}_{\mu,k;\eta}(\tau_p,X_{\tau_p}^{t,x}(\mu,\eta)) \right|} &\leq\, \esperance{\left| \nabla \Phi^{p+1}_{\lambda,k;\eta}(\tau_p,X_{\tau_p}^{t,x}(\mu,\eta)) - \nabla \Phi^{p+1}_{\mu,k;\eta}(\tau_p,X_{\tau_p}^{t,x}(\mu,\eta)) \right|}\\
        &+\, \esperance{\left| \nabla \Phi^{p+1}_{\lambda,k;\eta}(\tau_p,X_{\tau_p}^{t,x}(\mu,\eta)) - \nabla \Phi^{p+1}_{\lambda,k;\eta}(\tau_p,X_{\tau_p}^{t,x}(\lambda,\eta)) \right|}.
        \end{split}
    \end{equation}}\noindent The first term in \eqref{ineq:ineq-7} is controlled by the continuity of $d\Phi^{p+1}(\cdot).$ This continuity is uniform with respect to $\eta$, since the functions involved
are obtained by mollification and the map $\lambda^{>M} \mapsto \nabla\Phi^{p+1}(\lambda^{>M})(\tau_p,\cdot)$ is continuous for the uniform norm by induction hypothesis. Therefore, there exists $\delta_1>0,$ such that for all $ \mu \in \mathcal{B}(\lambda,\delta_1)$ the first term is bounded above by $\frac{\veps}{2}$. Let $\omega_\lambda(\cdot)$ denote the modulus of continuity of $\nabla\Phi^{p+1}(\lambda^{>M})(\tau_p,\cdot)$. By mollification, for every $\eta \in \N$, the function $\nabla\Phi_{\lambda,k;\eta}^{p+1}(\tau_p,\cdot)$ has the same modulus of continuity and is bounded by $\Norm{\nabla \Phi^{p+1}(\lambda^{>M})(\tau_p,\cdot)}_\infty$. Applying Markov's inequality, we obtain that for any $\eta\in\N$, $\delta_2 >0$,\begin{equation*}\begin{split}
         &\esperance{\left| \nabla \Phi^{p+1}_{\lambda,k;\eta}(\tau_p,X_{\tau_p}^{t,x}(\mu,\eta)) - \nabla \Phi^{p+1}_{\lambda,k;\eta}(\tau_p,X_{\tau_p}^{t,x}(\lambda,\eta)) \right|} \\&\leq \omega_\lambda(\delta_2) + \frac{2}{\delta_2} \Norm{\nabla \Phi^{p+1}(\lambda^{>M})(\tau_p,\cdot)}_\infty \esperance{\left| X_{\tau_p}^{t,x}(\mu,\eta))-X_{\tau_p}^{t,x}(\lambda,\eta)\right|}.
         \end{split}
    \end{equation*}\noindent The equicontinuity follows from the same arguments as in the
equicontinuity part of the proof of Proposition~\ref{prop:put-linearization}.
This completes the inductive step and concludes the proof.
\end{proof}

The convexity result of Proposition~\ref{prop:convexity-solution-map}
also extends to the cascade by backward induction.

\begin{corol}[Convexity of the cascade solution maps]
\label{coro:convexity-solution-map-gluing}
Assume that $H$ satisfies Assumption~\ref{assmpt:operatorH}
and that $(p,M)\mapsto H(t,x,p,M)$ is convex for every
$(t,x)\in Q^T$. Let $G\in\BUC(\R;\R^{PM})$.

For every $p\in\{1,\ldots,P\}$, the solution map
$\Phi^p$ defined in~\eqref{eq:solution-map-gluing} is convex
with respect to the pointwise order. More precisely, for every
$\lambda,\mu\in\R^{(P-p+1)M}$ and $\theta\in[0,1]$,
\begin{equation}\label{eq:convexity-solution-map-gluing}
    \Phi^p(\theta\lambda+(1-\theta)\mu)
    \leq
    \theta\Phi^p(\lambda)+(1-\theta)\Phi^p(\mu)
    \qquad\text{on }\overline{Q}^p.
\end{equation}
In particular, for every $x_0\in\R$, the map
\begin{equation*}
    \Phi_0^p:\R^{(P-p+1)M}\longrightarrow\R,
    \qquad
    \lambda\longmapsto\phi_p^\lambda(\tau_{p-1},x_0),
\end{equation*}
is convex.
\end{corol}

\begin{proof}
We proceed by backward induction on $p$.
For $p=P$, the terminal condition depends linearly on $\lambda$,
so the assertion follows from
Proposition~\ref{prop:convexity-solution-map}, applied on
$[\tau_{P-1},\tau_P]$.

Suppose that the assertion holds for $p+1$.
For $\lambda\in\R^{(P-p+1)M}$, denote the terminal datum on
the $p$-th interval by
\begin{equation*}
    h_p^\lambda
    :=
    \phi_{p+1}^{\lambda^{>M}}(\tau_p,\cdot)
    +\sum_{m=1}^M\lambda_mG_{i(p,m)},
\end{equation*}
where $\lambda^{>M}:=(\lambda_{M+1},\ldots,
\lambda_{(P-p+1)M})$.
Fix $\lambda,\mu\in\R^{(P-p+1)M}$ and $\theta\in[0,1]$,
and set $\lambda_\theta:=\theta\lambda+(1-\theta)\mu$.
The induction hypothesis and the linearity of the current
payoff contribution give
\begin{equation*}
    h_p^{\lambda_\theta}
    \leq \theta h_p^\lambda+(1-\theta)h_p^\mu.
\end{equation*}
Let $v$ be the solution on the $p$-th interval with terminal
datum $\theta h_p^\lambda+(1-\theta)h_p^\mu$.
By Proposition~\ref{prop:comp-principle-nonlinear},
$\phi_p^{\lambda_\theta}\leq v$.
Applying Proposition~\ref{prop:convexity-solution-map} on
this interval yields
\begin{equation*}
    v\leq
    \theta\phi_p^\lambda+(1-\theta)\phi_p^\mu.
\end{equation*}
Combining these inequalities proves
\eqref{eq:convexity-solution-map-gluing}.
The last assertion follows by evaluation at
$(\tau_{p-1},x_0)$.
\end{proof}

Assumption~\ref{assmpt:G} is extended to the cascade setting as follows.
\begin{assumption}\label{assmpt:G-bis}
    The vector $G=(G_1,\dots, G_{P M})$ belongs to $\BUC(\R;\R^{P M})$ and its components are linearly independent modulo constants on each time slice, i.e. $\forall p \in \{1,\dots,P\}$, if $\sum_{m=1}^M \mu_m G_{i(p,m)}(\cdot)$ is constant on $\R$, then $\mu = 0$.
\end{assumption}

\begin{prop}\label{prop:strong-cvx-gluing}
Assume that $H$ satisfies Assumptions~\ref{assmpt:operatorH},\ref{assmpt:operatorH-cvx} and $G$ satisfies Assumption~\ref{assmpt:G-bis}. Then, for every $p\in\{1,\dots,P\}$ and $x_0\in\mathbb{R}$, the map \begin{equation*}
        \Phi^p_0:\R^{(P-p+1)M}  \longrightarrow  \R,\quad \lambda \mapsto \phi^\lambda_p(\tau_{p-1},x_0)
    \end{equation*} is locally strongly convex.
\end{prop}

\begin{lemma}[Smooth regularized cascade]\label{lem:smooth-regularized-cascade}
Assume that $H$ satisfies Assumption~\ref{assmpt:operatorH}.  Let
$(\rho_\eta)_{\eta\in\N}$ be a sequence of smooth spatial mollifiers and set
$G_k^\eta:=\rho_\eta*_x G_k$.  Fix $\eta\in\N$.  Define the regularized cascade
backward by setting $\Phi_{P+1,\eta}\equiv0$ and, for
$p=P,\dots,1$, by letting $\Phi_{p,\eta}(\lambda)=\phi_{p,\eta}^\lambda$ be
the classical solution on $Q^p$ with terminal datum
\[
    \mathfrak{g}_{p,\eta}(\lambda)
    :=
    \rho_\eta*_x
    \Phi_{p+1,\eta}(\lambda^{>M})(\tau_p,\cdot)
    +
    \sum_{m=1}^M \lambda_m G_{i(p,m)}^\eta,
\]
where the first term is omitted for $p=P$.  Then, for every
$p\in\{1,\dots,P\}$,
\[
    \Phi_{p,\eta}:\R^{(P-p+1)M}\longrightarrow C^2(\overline{Q}^p)
\]
is twice continuously differentiable.  Moreover, for
$\mu,\nu\in\R^{(P-p+1)M}$, the first variation
\[
    \psi_{p,\eta}^{\nu,\mu}
    :=
    d\Phi_{p,\eta}(\nu)[\mu]
\]
solves
\[
\left\{
\begin{array}{ll}
    \dt\psi_{p,\eta}^{\nu,\mu}
    +
    \nabla H[\phi_{p,\eta}^{\nu}]
    [\psi_{p,\eta}^{\nu,\mu}]
    =
    0,
    &\text{on }Q^p,\\[0.3em]
    \psi_{p,\eta}^{\nu,\mu}(\tau_p,\cdot)
    =
    \rho_\eta*_x
    \psi_{p+1,\eta}^{\nu^{>M},\mu^{>M}}(\tau_p,\cdot)
    +
    \sum_{m=1}^M \mu_m G_{i(p,m)}^\eta,
    &
\end{array}
\right.
\]
with the future term omitted for $p=P$.  The second variation
\[
    \gamma_{p,\eta}^{\nu,\mu}
    :=
    d^2\Phi_{p,\eta}(\nu)[\mu,\mu]
\]
solves
\[
\left\{
\begin{array}{ll}
    \dt\gamma_{p,\eta}^{\nu,\mu}
    +
    \nabla H[\phi_{p,\eta}^{\nu}]
    [\gamma_{p,\eta}^{\nu,\mu}]
    +
    \Hessian H[\phi_{p,\eta}^{\nu}]
    [\psi_{p,\eta}^{\nu,\mu},\psi_{p,\eta}^{\nu,\mu}]
    =
    0,
    &\text{on }Q^p,\\[0.3em]
    \gamma_{p,\eta}^{\nu,\mu}(\tau_p,\cdot)
    =
    \rho_\eta*_x
    \gamma_{p+1,\eta}^{\nu^{>M},\mu^{>M}}(\tau_p,\cdot),
    &
\end{array}
\right.
\]
again with zero terminal datum for $p=P$.
\end{lemma}

\begin{proof}
Throughout this proof, \( *_x \) denotes convolution with respect to the
spatial variable \(x\).
We argue by backward induction on $p$.  For $p=P$, the terminal datum is the
smooth affine map
\[
    \lambda\longmapsto
    \sum_{m=1}^M \lambda_m G_{i(P,m)}^\eta
    \quad\text{from }\R^M\text{ to }C^{2+1}(\R).
\]
Proposition~\ref{prop:smooth-coef-linearization} therefore gives the claimed
$C^2$ dependence and identifies the first and second variations.  In this case
the second derivative of the terminal datum is zero.

Assume now that the statement is known at rank $p+1$.  The only point to check
is the regularity, with respect to the parameter, of the terminal trace
inherited from the next slice.  For $\zeta\in\R^{(P-p)M}$, we regularize this
trace only in the spatial variable and consider
\[
    \rho_\eta*_x
    \Phi_{p+1,\eta}(\zeta)(\tau_p,\cdot).
\]
By the induction
hypothesis, the map
\[
    \zeta\longmapsto
    \Phi_{p+1,\eta}(\zeta)(\tau_p,\cdot)
\]
is twice continuously differentiable.  Moreover, the induction hypothesis
identifies the first and second parameter derivatives with the corresponding
first and second variations on the next slice.  Their terminal data are
bounded, and the source term in the second variation equation is bounded,
because the first variations are bounded and the derivatives of $H$ are bounded
under Assumption~\ref{assmpt:operatorH}.  The maximum principle for bounded
solutions of linear parabolic equations on unbounded domains therefore yields
that these derivatives, evaluated at $t=\tau_p$, are bounded functions of the
space variable.  Since convolution in the spatial variable by the fixed smooth
kernel $\rho_\eta$ is linear, bounded on bounded functions, and regularizes into
$C^{2+1}(\R)$, we may differentiate under the convolution.  Hence
\[
    \zeta\longmapsto
    \rho_\eta*_x
    \Phi_{p+1,\eta}(\zeta)(\tau_p,\cdot)
\]
is twice continuously differentiable as a map into $C^{2+1}(\R)$.  Composing
this map with the projection $\lambda\mapsto\lambda^{>M}$, the first and second
derivatives of $\mathfrak{g}_{p,\eta}$ are
\[
    d \mathfrak{g}_{p,\eta}(\nu)[\mu]
    =
    \rho_\eta*_x
    \psi_{p+1,\eta}^{\nu^{>M},\mu^{>M}}(\tau_p,\cdot)
    +
    \sum_{m=1}^M \mu_m G_{i(p,m)}^\eta,
\]
and
\[
    d^2\mathfrak{g}_{p,\eta}(\nu)[\mu,\mu]
    =
    \rho_\eta*_x
    \gamma_{p+1,\eta}^{\nu^{>M},\mu^{>M}}(\tau_p,\cdot).
\]
Thus $\mathfrak{g}_{p,\eta}$ is a twice continuously differentiable map from
$\R^{(P-p+1)M}$ to $C^{2+1}(\R)$.  Applying
Proposition~\ref{prop:smooth-coef-linearization} on the slice $Q^p$ and using
the Banach-space chain rule gives the $C^2$ regularity of $\Phi_{p,\eta}$ and
the announced terminal conditions for the first and second variations.  This
closes the induction.
\end{proof}

\begin{proof}[Proof of Proposition~\ref{prop:strong-cvx-gluing}]
Fix $p\in\{1,\dots,P\}$, $x_0\in\R$, and a compact convex set
$K\Subset\R^{(P-p+1)M}$.  We set
\[
    \mathcal K_p
    :=
    \{(\lambda,\mu)\in K\times\R^{(P-p+1)M}:\lambda+\mu\in K\}.
\]
For $(\lambda,\mu)\in\mathcal K_p$, define the one-dimensional restriction
\[
    \Phi_p^{\lambda,\mu}(h)
    :=
    \phi_p^{\lambda+h\mu}(\tau_{p-1},x_0),
    \qquad h\in[0,1].
\]
By Lemma~\ref{lem:strict-convexity-onedimen}, it is enough to prove that
there exists $\underline\Lambda^p_K>0$ such that, for every
$(\lambda,\mu)\in\mathcal K_p$ and every $0\le h_1<h_2\le1$,
\begin{equation}\label{eq:final-gluing-strong-cvx-target}
    \partial_h\Phi_p^{\lambda,\mu}(h_2)
    -
    \partial_h\Phi_p^{\lambda,\mu}(h_1)
    \ge
    \underline\Lambda^p_K\Norm{\mu}^2(h_2-h_1).
\end{equation}

We first pass to the smooth regularized cascade of
Lemma~\ref{lem:smooth-regularized-cascade}.  For each mollification parameter
$\eta$, the terminal data are regularized backward through the cascade: on the
slice $Q^p$ they are given by
\[
    \mathfrak g_{p,\eta}(\lambda)
    =
    \rho_\eta*_x
    \Phi_{p+1,\eta}(\lambda^{>M})(\tau_p,\cdot)
    +
    \sum_{m=1}^M \lambda_m G_{i(p,m)}^\eta,
\]
with the future term omitted when $p=P$.  Thus the regularization is applied
successively, both to the original terminal functions and to the terminal trace
inherited from the next slice.  By
Lemma~\ref{lem:smooth-regularized-cascade}, for fixed $\eta$ the map
\[
    \Phi_{p,\eta}:
    \R^{(P-p+1)M}\longrightarrow C^2(\overline Q^p)
\]
is twice continuously differentiable.  For
$\nu,\mu\in\R^{(P-p+1)M}$, we denote its first and second variations by
\[
    \psi_{p,\eta}^{\nu,\mu}
    :=
    d\Phi_{p,\eta}(\nu)[\mu],
    \qquad
    \gamma_{p,\eta}^{\nu,\mu}
    :=
    d^2\Phi_{p,\eta}(\nu)[\mu,\mu].
\]
In particular, $\gamma_{p,\eta}^{\nu,\mu}$ solves the second linearized
equation on $Q^p$ with terminal condition
\[
    \gamma_{p,\eta}^{\nu,\mu}(\tau_p,\cdot)
    =
    \rho_\eta*_x
    \gamma_{p+1,\eta}^{\nu^{>M},\mu^{>M}}(\tau_p,\cdot),
\]
again with zero future terminal datum when $p=P$.

For $(\lambda,\mu)\in\mathcal K_p$, set
\[
    \Phi_{p,\eta}^{\lambda,\mu}(h)
    :=
    \phi_{p,\eta}^{\lambda+h\mu}(\tau_{p-1},x_0),
    \qquad h\in[0,1].
\]
The chain rule and the identification of the first and second variations give
\[
    \partial_h\Phi_{p,\eta}^{\lambda,\mu}(h)
    =
    \psi_{p,\eta}^{\lambda+h\mu,\mu}(\tau_{p-1},x_0),
    \qquad
    \partial_{hh}^2\Phi_{p,\eta}^{\lambda,\mu}(h)
    =
    \gamma_{p,\eta}^{\lambda+h\mu,\mu}(\tau_{p-1},x_0).
\]
Consequently, for $0\le h_1<h_2\le1$,
\begin{equation}\label{eq:final-gluing-mollified-second-variation}
    \partial_h\Phi_{p,\eta}^{\lambda,\mu}(h_2)
    -
    \partial_h\Phi_{p,\eta}^{\lambda,\mu}(h_1)
    =
    \int_{h_1}^{h_2}
    \gamma_{p,\eta}^{\lambda+z\mu,\mu}(\tau_{p-1},x_0)\,dz .
\end{equation}

We now introduce the localized auxiliary problem.  Fix
$\nu\in K$ and $\mu\in\R^{(P-p+1)M}$.  When
$\mu^{\le M}\neq0$, we choose, as in Step~2 of the proof of
Theorem~\ref{thm:strict-cvx-flow}, a time localization
$\tau(\nu,\mu^{\le M})\in(0,\tau_p-\tau_{p-1})$ and a compact interval
$E(\nu)\Subset\R$ such that the current-slice source term is strictly positive
somewhere in the corresponding space-time region.  When $\mu^{\le M}=0$, this
choice is immaterial, and we fix any admissible one.  Let
$\xi_p^{\nu,\mu}$ be a smooth cutoff with values in $[0,1]$, supported in
\[
    (\tau_{p-1},\tau_p-\tau(\nu,\mu^{\le M}))\times E(\nu),
\]
and nontrivial on a smaller region where the current-slice source term is
strictly positive when $\mu^{\le M}\neq0$.  Let also $R\Subset\R$ be a compact
interval with nonempty interior, and let $\chi_R$ be a smooth spatial cutoff
with $0\le\chi_R\le1$.

We define $\wtilde\gamma_{p,\eta}^{\nu,\mu}$ as the solution of the auxiliary
second linearized problem obtained by restricting the source term to the compact
space-time localization, and by first mollifying the future auxiliary terminal
datum and then restricting it to the compact set selected by $\chi_R$:
\[
\left\{
\begin{array}{ll}
    \dt\wtilde\gamma_{p,\eta}^{\nu,\mu}
    +
    \nabla H[\phi_{p,\eta}^{\nu}]
    [\wtilde\gamma_{p,\eta}^{\nu,\mu}]
    +
    \xi_p^{\nu,\mu}
    \Hessian H[\phi_{p,\eta}^{\nu}]
    [\psi_{p,\eta}^{\nu,\mu},\psi_{p,\eta}^{\nu,\mu}]
    =
    0,
    &\text{on }Q^p,\\[0.3em]
    \wtilde\gamma_{p,\eta}^{\nu,\mu}(\tau_p,\cdot)
    =
    \begin{cases}
        \big(
            \rho_\eta*_x
            \wtilde\gamma_{p+1,\eta}^{\nu^{>M},\mu^{>M}}(\tau_p,\cdot)
        \big)\chi_R(\cdot),
        & p<P,\\
        0,
        & p=P,
    \end{cases}
    &
\end{array}
\right.
\]
The dependence of
$\wtilde\gamma_{p,\eta}^{\nu,\mu}$ on the auxiliary cutoffs is kept implicit.

We shall prove the following localized regularized assertion by backward
induction.  For every rank $q\in\{p,\dots,P\}$, every compact convex set
$K_q\Subset\R^{(P-q+1)M}$, and every compact interval $R\Subset\R$ with
nonempty interior, there exists
$\underline\Lambda^q_{K_q,R}>0$ such that, for every
$\nu\in K_q$ and every $\mu\in\R^{(P-q+1)M}$, the auxiliary problems defined
above at rank $q$ satisfy
\begin{equation}\label{eq:final-main-induc-gluing-conv}
    \gamma_{q,\eta}^{\nu,\mu}(\tau_{q-1},\cdot)\chi_R(\cdot)
    \ge
    \wtilde\gamma_{q,\eta}^{\nu,\mu}(\tau_{q-1},\cdot)\chi_R(\cdot)
    \xrightarrow[\eta\to\infty]{\Norm{\cdot}_\infty}
    \wtilde\gamma_q^{\nu,\mu}(\tau_{q-1},\cdot)\chi_R(\cdot)
    \ge
    \underline\Lambda^q_{K_q,R}\Norm{\mu}^2\chi_R(\cdot).
\end{equation}
Here $\wtilde\gamma_q^{\nu,\mu}$ denotes the limit of the regularized auxiliary
solutions $\wtilde\gamma_{q,\eta}^{\nu,\mu}$, obtained by
Proposition~\ref{prop:abstrait-existence-convergence}.  In particular,
$\wtilde\gamma_q^{\nu,\mu}$ solves the corresponding limiting auxiliary linear
problem.

We prove~\eqref{eq:final-main-induc-gluing-conv} by backward induction.  For
the terminal rank $q=P$, there is no future terminal contribution.  The lower
bound on $\wtilde\gamma_P^{\nu,\mu}$ is precisely the one-maturity strict
convexity estimate of Theorem~\ref{thm:strict-cvx-flow}, applied on the
translated time interval $(\tau_{P-1},\tau_P)$ and uniformly for initial points
in the compact set $\operatorname{supp}\chi_R$.  Multiplying by $\chi_R$ gives
the rightmost inequality in~\eqref{eq:final-main-induc-gluing-conv}.  The
convergence of the localized auxiliary equations follows from the stability of
the nonlinear background solutions, Proposition~\ref{prop:stability-nonlin-lip-coeff-bis},
the convergence of the first linearizations,
Corollary~\ref{coro:existence-convergence-linearization}, and
Proposition~\ref{prop:abstrait-existence-convergence}, after restricting the
conclusion to the compact set $\operatorname{supp}\chi_R$.  Finally, the
comparison
$\gamma_{P,\eta}^{\nu,\mu}\ge\wtilde\gamma_{P,\eta}^{\nu,\mu}$ follows from the
comparison principle, since both terminal data vanish and the localized source
term is bounded above by the full nonnegative source term.

Assume now that~\eqref{eq:final-main-induc-gluing-conv} has been proved at
rank $q+1$.  Let
\[
    K_q^{>M}:=\{\zeta^{>M}:\zeta\in K_q\}
    \Subset \R^{(P-q)M}.
\]
The induction hypothesis applied to $K_q^{>M}$ and to the same compact interval
$R$ gives the lower bound for the future auxiliary terminal contribution, as
well as the compact-uniform convergence of the corresponding regularized
future auxiliary traces.  Hence the terminal datum
\[
    \big(
        \rho_\eta*_x
        \wtilde\gamma_{q+1,\eta}^{\nu^{>M},\mu^{>M}}(\tau_q,\cdot)
    \big)\chi_R(\cdot)
\]
converges uniformly on $\operatorname{supp}\chi_R$ to
$\wtilde\gamma_{q+1}^{\nu^{>M},\mu^{>M}}(\tau_q,\cdot)\chi_R$.  Together with
Proposition~\ref{prop:stability-nonlin-lip-coeff-bis} for the nonlinear
background equations and Corollary~\ref{coro:existence-convergence-linearization}
for the first variations, this allows us to apply
Proposition~\ref{prop:abstrait-existence-convergence} to the localized
auxiliary equations.  We obtain the convergence in
\eqref{eq:final-main-induc-gluing-conv}; this convergence is asserted only
after multiplication by $\chi_R$, because this stability result gives
the required uniform convergence locally.

The comparison inequality at rank $q$ is again a direct consequence of the
comparison principle.  Indeed, the terminal datum of
$\wtilde\gamma_{q,\eta}^{\nu,\mu}$ is bounded above by the terminal datum of
$\gamma_{q,\eta}^{\nu,\mu}$, by the induction comparison
$\gamma_{q+1,\eta}\ge\wtilde\gamma_{q+1,\eta}$, the positivity of the
mollifier, and $0\le\chi_R\le1$.  Moreover, the source term of
$\wtilde\gamma_{q,\eta}^{\nu,\mu}$ is obtained from the source term of
$\gamma_{q,\eta}^{\nu,\mu}$ by multiplication by a cutoff with values in
$[0,1]$, while $\Hessian H$ is positive semidefinite.

It remains to prove the lower bound for the limiting auxiliary solution at
rank $q$.  The Feynman--Kac representation gives, for
$x\in\operatorname{supp}\chi_R$,
\[
\begin{split}
    \wtilde\gamma_q^{\nu,\mu}(\tau_{q-1},x)
    =
    \E\Bigg[
        &\int_{\tau_{q-1}}^{\tau_q-\tau(\nu,\mu^{\le M})}
        f_q(t,X_t^{\nu,x};\nu,\mu)\,dt \\
        &+
        \wtilde\gamma_{q+1}^{\nu^{>M},\mu^{>M}}
        \bigl(\tau_q,X_{\tau_q}^{\nu,x}\bigr)
        \chi_R\bigl(X_{\tau_q}^{\nu,x}\bigr)
    \Bigg],
\end{split}
\]
where $f_q$ denotes the localized nonnegative current-slice source term, and
the terminal contribution is omitted when $q=P$.  The induction hypothesis
controls the second term from below by a positive multiple of
$\Norm{\mu^{>M}}^2$, after using the nondegeneracy and support argument from
Step~2(c) of Theorem~\ref{thm:strict-cvx-flow}.  If $\mu^{>M}=0$ but
$\mu^{\le M}\neq0$, the one-maturity localization argument on the current slice
gives a strictly positive contribution from the source term.  Thus the
right-hand side defines a positive definite quadratic form in $\mu$.  The same
compactness argument as in Step~3 of Theorem~\ref{thm:strict-cvx-flow}, now on
$K_q\times\operatorname{supp}\chi_R$, yields a constant
$\underline\Lambda^q_{K_q,R}>0$ such that the rightmost inequality in
\eqref{eq:final-main-induc-gluing-conv} holds.  This completes the induction.

We now conclude the proof.  Repeating the backward stability argument used to
identify the first variations in
Proposition~\ref{prop:put-linearization-gluing}, namely the stability of the
nonlinear equations together with
Proposition~\ref{prop:abstrait-existence-convergence} applied to the first
linearized equations, gives
\[
    \partial_h\Phi_{p,\eta}^{\lambda,\mu}(h)
    \longrightarrow
    \partial_h\Phi_p^{\lambda,\mu}(h),
    \qquad h\in[0,1].
\]
Let $(\lambda,\mu)\in\mathcal K_p$ and $0\le h_1<h_2\le1$.  Since $K$ is
convex, $\lambda+z\mu\in K$ for every $z\in[h_1,h_2]$.  Choose a spatial cutoff
with $\chi_R(x_0)>0$.  Evaluating
\eqref{eq:final-main-induc-gluing-conv} at rank $q=p$ and at the point $x_0$
gives, for every $z\in[h_1,h_2]$,
\[
    \liminf_{\eta\to\infty}
    \gamma_{p,\eta}^{\lambda+z\mu,\mu}(\tau_{p-1},x_0)
    \ge
    \underline\Lambda^p_{K,R}\Norm{\mu}^2.
\]
Combining this pointwise lower bound with
\eqref{eq:final-gluing-mollified-second-variation}, and using Fatou's lemma
for the nonnegative second variations, yields
\[
\begin{split}
    \partial_h\Phi_p^{\lambda,\mu}(h_2)
    -
    \partial_h\Phi_p^{\lambda,\mu}(h_1)
    &=
    \lim_{\eta\to\infty}
    \int_{h_1}^{h_2}
    \gamma_{p,\eta}^{\lambda+z\mu,\mu}(\tau_{p-1},x_0)\,dz  \\
    &\ge
    \int_{h_1}^{h_2}
    \liminf_{\eta\to\infty}
    \gamma_{p,\eta}^{\lambda+z\mu,\mu}(\tau_{p-1},x_0)\,dz  \\
    &\ge
    \underline\Lambda^p_{K,R}\Norm{\mu}^2(h_2-h_1).
\end{split}
\]
Thus~\eqref{eq:final-gluing-strong-cvx-target} holds with
$\underline\Lambda^p_K:=\underline\Lambda^p_{K,R}$, and the proof is complete.
\end{proof}

\appendix
\section{Proofs of Section~\ref{sec:estimates}}\label{sec:appendix-estimates}
\subsection{Proof of Proposition~\ref{prop:comp-principle-nonlinear}}
    Let $\psi = u-v$. Then $\psi$ is a subsolution to the linearized equation associated with \eqref{eq:comp-prcp-nn}, namely \begin{equation*}
        \dt \psi + b(t,x)\dx \psi + a(t,x) \dxx\psi \geq 0 \text{ on } Q^T.
    \end{equation*} where the coefficients $(a,b)$ are defined as follows\begin{equation*}
        \begin{split}
        b(t,x) &= \int_0^1 D_p H(t,x,s\dx u +(1-s) \dx v, \demi \dxx u)ds,\\
            a(t,x)&=\demi \int_0^1 D_M H(t,x,\dx v,s\demi \dxx u + (1-s) \demi \dxx v)ds.
        \end{split}
    \end{equation*} By the boundedness assumptions on the derivatives of $H$, the coefficients $a$ and $b$ are bounded on $Q^T$, and $a$ is uniformly elliptic. Since $u(T,\cdot)\leq v(T,\cdot)$, we have $\psi(T,\cdot)\leq 0$. The maximum principle on unbounded domains, Theorem~8.1.4 in~\cite{Kry96_holder}, then implies that $\psi\leq 0$ on $Q^T$. This proves the first part of the proposition.

    The second part follows from the same argument, observing that in this case $\psi$ is a solution to the linearized equation. We conclude by applying Corollary~8.1.5 in \cite{Kry96_holder}.

\subsection{Proof of Corollary~\ref{coro:max-principle-nonlinear}}
     For $t\in [0,T]$, set \begin{equation*}
    M(t) = \Norm{\varphi}_\infty + \int_t^T \Norm{H(s,\cdot,0,0)}_\infty ds, \quad m(t)=-M(t).
\end{equation*}
    Note that $\dt M + H(t,x,\dx M, \demi \dxx M)= -\Norm{H(t,\cdot,0,0)}_\infty+ H(t,x,0,0) \leq 0$, hence $M$ is a supersolution to~\eqref{eq:comp-prcp-nn}. Similarly $m$ is a subsolution. Since $\phi$ is a solution to~\eqref{eq:comp-prcp-nn}, and $m(T)\leq \phi(T,\cdot)\leq M(T)$, applying the comparison principle of Proposition~\ref{prop:comp-principle-nonlinear} to the pairs $(m,\phi)$ and $(\phi,M)$ yields $m\leq \phi \leq M$ on $Q^T$, which gives the desired estimate.

\subsection{Proof of Proposition~\ref{prop:exi-uni-edp-nonlin}}
The existence and uniqueness of a solution 
$\phi\in C^{2}_{loc}(Q^T)\cap C(\overline{Q}^T)$
follow from Theorem~5.5 in \cite{Kry82}. 
In the proof of this theorem, the solution is obtained as the limit (in the uniform topology) of a sequence 
$(\phi_n)$ solving the Cauchy problem with mollified terminal conditions 
$p_n\in C^3(\R)$. Theorem~5.3 in \cite{Kry82} ensures that each $\phi_n$ is bounded on 
$\overline{Q}^T$. Moreover, under assumption~\eqref{ineq:hyp-fstar}, 
Corollary~\ref{coro:max-principle-nonlinear} yields the uniform estimate
\[
\Norm{\phi_n}_\infty
\le 
R+\int_0^T \Norm{H(t,\cdot,0,0)}_{\infty}\,dt
=: \overline C <\infty, \quad \forall n \in \N.
\] Passing to the limit gives $\Norm{\phi}_\infty \leq \overline{C}$. The interior regularity estimate then follows from Theorem~5.6 in \cite{Kry82}. 
This result states that for every $\tau>0$ there exist 
$\alpha\in(0,1)$ and $N\ge0$, depending only on $\mathfrak{v}$, $\tau$ and 
$\Norm{\phi}_\infty$, such that
\[
\Norm{\phi}_{2+\alpha;\overline{Q}^{T-\tau}}\le N .
\]
Since $\Norm{\phi}_\infty\le\overline C$, the constants $\alpha$ and $N$ can be 
chosen depending only on $\mathfrak{v}$, $\tau$ and $\overline C$. 
This yields the desired estimate.

\subsection{Proof of Proposition~\ref{prop:stability-nonlin-lip-coeff-bis}}
    Fix $\tau \in (0,T)$. By Proposition~\ref{prop:exi-uni-edp-nonlin}, there exist $\alpha_0\in(0,1)$ and $C_0>0$ such that
\begin{equation}\label{eq:proof-stab-1-bis}
\Norm{\phi_i}_{2+\alpha_0;\overline{Q}^\tau} \le C_0
\qquad \forall i\in I.
\end{equation} We first show the stability estimate~\eqref{eq:propstab-1-bis}. Let $\psi=\phi_i-\phi_j$. Then $\psi$ satisfies
\[
\partial_t \psi+a(t,x)\partial_{xx}\psi+b(t,x)\partial_x\psi=0
\quad \text{on } Q^T,
\]
with terminal condition $\psi(T,\cdot)=\varphi_i(\cdot)-\varphi_j(\cdot)$, where the coefficients are defined by
\begin{align*}
a(t,x)
&=\tfrac12 \int_0^1
D_M H(t,x,\partial_x\phi_j,
s\tfrac12\partial_{xx}\phi_i+(1-s)\tfrac12\partial_{xx}\phi_j)\,ds,\\
b(t,x)
&=\int_0^1
D_p H(t,x,
s\partial_x\phi_i+(1-s)\partial_x\phi_j,
\tfrac12\partial_{xx}\phi_i)\,ds .
\end{align*} By the maximum principle of Proposition~\ref{prop:comp-principle-nonlinear} we have
\begin{equation}\label{eq:maxprcpl-linear-bis}
\Norm{\psi}_{\infty}
\le
\Norm{\varphi_i-\varphi_j}_{\infty}.
\end{equation} Estimate~\eqref{eq:proof-stab-1-bis} yields uniform bounds on the
$C^{2+\alpha}$ norms of $\phi_i$ and $\phi_j$ on $\overline{Q}^{T-\tau}$. The global regularity and boundedness of $D_{(p,M)}H$, combined with
Remark~\ref{rk:regularity-coeff-domain} and the preceding Hölder bounds, imply
that the coefficients $(a,b)$ are Hölder continuous on
$\overline{Q}^{T-\tau}$, with Hölder norms bounded uniformly in $(i,j)$.
Moreover, $\psi\in C^{2+\alpha}(\overline{Q}^{T-\tau})$. Applying the interior Schauder estimate of Theorem~4.9 in \cite{lieberman}
together with \eqref{eq:maxprcpl-linear-bis}, for every cylinder $Q_r \subset Q^{T-\tau}$ of radius $r>0$, there exists a constant \begin{equation*}
   C_1 := C(d,r,C_0,\mathfrak{v}) >0
\end{equation*} hence independent of $(i,j)$, such that
\begin{equation*}
\Norm{\psi}_{2+\alpha;\overline{Q}_r}
\le
C_1\,\Norm{\varphi_i-\varphi_j}_{\infty}.
\end{equation*} The constant $C_1$ in the local Schauder estimate is uniform with respect to spatial
translations. Covering \([0,T-\tau]\times\mathbb R\) by such cylinders, we obtain
\[
\|\phi_i-\phi_j\|_{2+\alpha;\overline{Q}^{T-\tau}}
\leq
C_1\|\varphi_i-\varphi_j\|_\infty \qquad \forall i,j\in I.
\]
The second estimate~\eqref{eq:propstab-2-bis} follows from a bootstrap argument; we detail it to track the uniformity of the constants. Let $|h| \in (0,1)$, and define \begin{equation*}
    v_i^h(t,x):=\frac{\phi_i(t,x+h)-\phi_i(t,x)}{|h|}.
\end{equation*} By estimate~\eqref{eq:proof-stab-1-bis}, $v_i^h \in C^{2+\alpha_0}(\overline{Q}^{T-\tau})$ and satisfies \begin{equation*}\begin{split}
    &\partial_t v_i^h + a_i^h(t,x)\tfrac12\dxx v_i^h + b_i^h(t,x)\dx v_i^h + f_i^h(t,x)=0
 \text{ on } Q^{T-\tau},\\
 &|v_i^h(t,x)| \leq C_0 \text{ on } \overline{Q}^{T-\tau},
 \end{split}
\end{equation*} where \begin{align*}
        a^h_i(t,x) &= \int_0^1 D_M H(t,x+h,\dx \phi^{+h}_i, s\demi \dxx \phi^{+h}_i + (1-s)\demi \dxx \phi_i)ds,\\
        b^h_i(t,x)&= \int_0^1 D_p H(t,x+h,s\dx \phi^{+h}_i + (1-s)\dx \phi_i , \demi \dxx \phi_i)ds,\\
        f^h_i(t,x)&=\int_0^1 D_x H(t,x+sh,\dx \phi_i,\demi \dxx \phi_i)ds.
    \end{align*} Using the uniform $C^{2+\alpha_0}$ bounds on $\phi_i$ together with the structural assumptions on $H$, the coefficients $(a_i^h,b_i^h,f_i^h)$ are Hölder continuous on $\overline{Q}^{T-\tau}$ uniformly with respect to $i$ and $h$, and there exists $C_2>0$, \begin{equation*}
    \Norm{a_i^h}_{\alpha_0;\overline{Q}^{T-\tau}},\Norm{b_i^h}_{\alpha_0;\overline{Q}^{T-\tau}},\Norm{f_i^h}_{\alpha_0;\overline{Q}^{T-\tau}} < C_2 \qquad \forall i \in I \text{ and } h.
\end{equation*} Hence, by the interior Schauder estimate of Theorem~4.9 in \cite{lieberman}, for every cylinder $Q_r \subset Q^{T-\tau}$ of radius $r>0$, there exists a constant $$C_3:=C(d,r,C_2,\mathfrak{v})>0$$ hence independent of $(i,j)$, such that
\[
\Norm{v_i^h}_{2+\alpha_0;\overline{Q}_r}
\le
C_3\Bigl(
\Norm{v_i^h}_{\infty;\overline{Q}^{T-\tau}}
+
\Norm{f_i^h}_{\alpha_0;\overline{Q}^{T-\tau}}
\Bigr)
\le
C_3\Bigl(C_0+ C_2 \Bigr)=:C_4,
\]
uniformly over $i\in I$ and $h$. Passing to the limit as $h\to0$ and using Arzelà--Ascoli, we obtain \begin{equation*}
    \Norm{\dx \phi}_{2+\alpha_0;\overline{Q}_r} \leq C_4.
\end{equation*} By the same arguments as above, since the constant $C_4$ is uniform over translated cylinders of fixed radius, we obtain the global estimate~\eqref{eq:propstab-2-bis}.

\subsection{Proof of Lemma~\ref{lem:strong-solution}}
We fix a filtered probability space
$(\Omega,\mathcal{F},(\mathcal{F}_t)_{0\leq t \leq T},\mathbb{P})$
supporting a Brownian motion $(W_t)_{0\leq t\leq T}$, and let $x\in\R$.

Let $(t_n)_{n\geq 1}$ be a strictly increasing sequence such that $t_n\longrightarrow T$. By Corollary~\ref{coro:lip-local-coeff}, for every $n\geq 1$, the coefficients
$b$ and $a^{1/2}$ are Lipschitz continuous with respect to the spatial variable $x$
on $[0,t_n]\times\R$, with a Lipschitz constant depending on $n$. Since they are also bounded, Theorem~7.1 in \cite{Le_Gall2012} yields a unique strong solution $X^{(n)}$ to the SDE~\eqref{eq:eds-nonlin} on $[0,t_n]$.

By pathwise uniqueness, the solutions $(X^{(n)})_{n\geq 1}$ are consistent on overlapping intervals. They therefore define an adapted process $(X_t)_{0\leq t<T}$ with continuous trajectories such that, for every $t<T$,
\begin{equation}
    X_t
    =
    x
    +
    \int_0^t b(s,X_s)\,ds
    +
    \int_0^t a(s,X_s)^\demi\,dW_s.
\end{equation}

Extend the progressively measurable processes
$s\mapsto b(s,X_s)$ and $s\mapsto a(s,X_s)^\demi$ by $0$ at $s=T$, and denote these extensions by $\mathrm{b}$ and $\mathrm{a}$, respectively. Since the coefficients are bounded,
\begin{equation}
    \int_0^T |\mathrm{b}_s|\,ds<\infty
    \quad\mathbb{P}\text{-a.s.},
    \qquad
    \mathbb{E}\left[
        \int_0^T |\mathrm{a}_s|^2\,ds
    \right]<\infty.
\end{equation}
Hence the processes
\begin{equation}
    B_t:=\int_0^t \mathrm{b}_s\,ds,
    \qquad
    M_t:=\int_0^t \mathrm{a}_s\,dW_s,
    \qquad 0\leq t\leq T,
\end{equation}
are well defined, adapted, and continuous.

Define
\begin{equation}
    \widetilde X_t:=x+B_t+M_t,
    \qquad 0\leq t\leq T.
\end{equation}
For every $n\geq 1$, the processes $\widetilde X$ and $X^{(n)}$ coincide on $[0,t_n]$. Consequently, $\widetilde X=X$ indistinguishably on $[0,T)$. It follows that
\begin{equation}
    \mathrm{b}_s=b(s,\widetilde X_s),
    \qquad
    \mathrm{a}_s=a(s,\widetilde X_s)^\demi,
\end{equation}
for $ds\otimes d\mathbb{P}$-almost every $(s,\omega)\in[0,T]\times\Omega$. Therefore, for every $t\in[0,T]$,
\begin{equation}
    \widetilde X_t
    =
    x
    +
    \int_0^t b(s,\widetilde X_s)\,ds
    +
    \int_0^t a(s,\widetilde X_s)^\demi\,dW_s.
\end{equation}
Thus $\widetilde X$ is a strong solution to SDE~\eqref{eq:eds-nonlin} on $[0,T]$. Finally, pathwise uniqueness on $[0,T]$ follows from pathwise uniqueness on every $[0,t_n]$ and continuity at $T$.

\subsection{Proof of Proposition~\ref{prop:abstrait-existence-convergence}}
The proof is divided into a preliminary step and three main steps. The
preliminary step makes explicit the existence and uniqueness of a solution to the Cauchy
problem~\eqref{eq:pde-lemm-abstrait-holder-eta}, as well as the a priori
uniqueness of solutions to the Cauchy problem
\eqref{eq:pde-lemm-abstrait-holder}. In the first step, we prove that, for
every $\tau\in(0,T)$ and every $\eta\in\N$, the function $u_\eta$ belongs to
$C^{2+\theta_\tau}([0,T-\tau]\times\R)$, with a bound on this Hölder norm that
is uniform in $\eta$. In the second step, this uniform bound and compactness
yield a subsequential limit
$u^\star\in C^2_{\mathrm{loc}}(Q^T)\cap C_b(\overline{Q}^T)$ of the sequence
$(u_\eta)_\eta$. This limit solves the differential equation in
\eqref{eq:pde-lemm-abstrait-holder} on $Q^T$ and satisfies
\eqref{eq:res-limit-lem} along the subsequence. Finally, we show that
$u^\star$ extends continuously to the terminal time with terminal condition
$u^\star(T,\cdot)=g$. Hence $u^\star$ solves the Cauchy
problem~\eqref{eq:pde-lemm-abstrait-holder}, and uniqueness implies that the
convergence in~\eqref{eq:res-limit-lem} holds for the whole sequence.\\
    \textit{Preliminary:} By Theorem~9.2.3 in \cite{Kry96_holder} for each $\eta\in\N$, there exists a unique solution $u_\eta$ to the Cauchy problem~\eqref{eq:pde-lemm-abstrait-holder-eta} that belongs to $C^{2+\alpha_\eta}(\overline{Q}^T)$. We have that there exists a constant \begin{equation*}
        C_0^\eta := C(d,T,\alpha_\eta,\Norm{a_\eta}_{\alpha_\eta},\Norm{b_\eta}_{\alpha_\eta},\mathfrak{v})>0
    \end{equation*} such that \begin{equation}\label{eq:estimate-existence-holder-923}
        \Norm{u_\eta}_{2+\alpha_\eta;\overline{Q}^T} \leq C_0^\eta \left( \Norm{g_\eta}_{2+\alpha_\eta;\R} + \Norm{f_\eta}_{\alpha_\eta;Q^T}\right)
    \end{equation} and by the maximum principle of Corollary~8.1.5 in \cite{Kry96_holder},\begin{equation}\label{eq:estimate-815}
        \Norm{u_\eta}_{\infty;\overline{Q}^T} \leq T \Norm{f_\eta}_{\infty;\overline{Q}^T} + \Norm{g_\eta}_{\infty;\overline{Q}^T}.
    \end{equation} By Theorem~8.1.4 in \cite{Kry96_holder} we obtain the a priori uniqueness of a solution to the Cauchy problem~\eqref{eq:pde-lemm-abstrait-holder} belonging to $C^{2}_{loc}(Q^T)\cap C_b(\overline{Q}^T)$.\\
    \textit{Step 1:} Let $\tau\in(0,T)$ and $\eta\in\N$, suppose that $\theta_\tau<\alpha_\eta$, we mollify coefficients $(a_\eta,b_\eta,f_\eta)$ and $g_\eta$ by $(a_\eta^n,b_\eta^n,f_\eta^n)_{n\in\N} \in C^{\theta_\tau}(Q^T)$ and $(g_\eta^n)_{n\in\N} \in C^{2+\theta_\tau}(\R)$. By Theorem~9.2.3 in \cite{Kry96_holder}, there exists $u_\eta^n \in C^{2+\theta_\tau}(\overline{Q}^T)$ solution to the Cauchy problem \begin{equation*}
    \dt u + b_\eta^n \dx u + \demi a_\eta^n\dxx u + f_\eta^n = 0 \quad \text{on } Q^T, \qquad u(T,\cdot)=g_\eta^n(\cdot).
\end{equation*}  Moreover, $u^n_\eta$ satisfies~\eqref{eq:estimate-existence-holder-923} and \eqref{eq:estimate-815}. Hence, \begin{equation*}
    \Norm{u_\eta^n}_{2+\alpha_\eta;\overline{Q}^T} \leq C_0^\eta \left(\Norm{g_\eta}_{2+\alpha_\eta;\R} + \Norm{f_\eta}_{\alpha_\eta;Q^T}\right).
\end{equation*} Applying Lemma~\ref{lem:holder-compact-embedding}, there exists $u_\eta^\star \in C^{2+\alpha_\eta}(\overline{Q}^T)$ solution of the problem~\eqref{eq:pde-lemm-abstrait-holder-eta}, by uniqueness of the solution $u_\eta = u_\eta^\star$ and for every compact $K\Subset [0,T]\times\R,$ \begin{equation*}
    \Norm{u_\eta^n-u_\eta^\star}_{2+\gamma;K}\to 0, \;\text{ as } n \to \infty, \; \forall \gamma\in (0,\alpha_\eta).
\end{equation*} First, note that there exists $C_1^\eta>0$, such that \begin{equation*}
    \Norm{a_\eta}_{\theta_\tau;\overline{Q}^{T-\tau}},
    \Norm{b_\eta}_{\theta_\tau;\overline{Q}^{T-\tau}} < C_1^\eta,
\end{equation*} hence \begin{equation*}
    \Norm{a_\eta^n}_{\theta_\tau;\overline{Q}^{T-\tau}},
    \Norm{b_\eta^n}_{\theta_\tau;\overline{Q}^{T-\tau}} < C_1^\eta \quad \forall n\in\N.
\end{equation*} By the interior Schauder estimate of Theorem~4.9 in \cite{lieberman}, for any $r>0$, there exists a constant \begin{equation*}
    C_2:=C(d,T,\tau,C_1,\mathfrak{v},r)
\end{equation*} such that for every cylinder $Q_r \subset [0,T-\tau]\times\R$ of radius $r>0$, \begin{equation*}
    \Norm{u_\eta^n}_{2+\theta_\tau;Q_r} \leq C_2 \left( \Norm{g_\eta}_\infty+\Norm{f_\eta}_{\theta_\tau;[0,T-\tau]\times\R} \right).
\end{equation*} The constant $C_2$ is uniform with respect to spatial translations. Covering $[0,T-\tau]\times\R$ by such cylinders, we obtain \begin{equation}
    \Norm{u_\eta^n}_{2+\theta_\tau;[0,T-\tau]\times\R} \leq C_2 \left( \Norm{g_\eta}_\infty+\Norm{f_\eta}_{\theta_\tau;[0,T-\tau]\times\R} \right).
\end{equation} Again using Lemma~\ref{lem:holder-compact-embedding}, there exists $u^{\star\star}_\eta \in C^{2+\theta_\tau}([0,T-\tau]\times\R)$, such that on each compact $K\Subset [0,T-\tau]\times\R$ there is a subsequence $\phi(n)$ of $(u^n_\eta)_{n}$ such that $\forall \gamma \in (0,\theta_\tau)$,\begin{equation*}
    \Norm{u_\eta^{\phi(n)}-u^{\star\star}_\eta}_{2+\gamma;K}\to 0, \quad \text{as } n\to \infty.
\end{equation*} Moreover, \begin{equation*} \begin{split}
    \Norm{u^{\star\star}_\eta}_{2+\theta_\tau;[0,T-\tau]\times\R} &\leq C_2 \left( \Norm{g_\eta}_\infty + \Norm{f_\eta}_{\theta_\tau} \right)  \\
    &\leq C_2 \left( \Norm{g}_\infty + \Norm{f}_{\theta_\tau} \right).
    \end{split} 
\end{equation*} By uniqueness of limits we have $u^{\star\star}_\eta=u^\star_\eta=u_\eta$, hence \begin{equation}\label{eq:norm-u-eta-schauder}
    \Norm{u_\eta}_{2+\theta_\tau;[0,T-\tau]\times\R} \leq C_2 \left( \Norm{g}_\infty + \Norm{f}_{\theta_\tau} \right).
\end{equation}
\textit{Step 2:} Using estimate~\eqref{eq:norm-u-eta-schauder} and Lemma~\ref{lem:holder-compact-embedding}, we obtain a subsequential limit $u^\star$ that belongs to $C^{2+\theta_\tau}([0,T-\tau]\times\R)$ such that for every compact $K\Subset [0,T-\tau]\times\R$ there exists a subsequence $(u_{\eta_n})_{n\in\N}$ such that for every $\gamma\in(0,\theta_\tau)$, \begin{equation*}
    \Norm{u_{\eta_n}-u^\star}_{2+\theta_\tau;\overline{Q}^{T-\tau}} \to 0, \quad \text{ as } n \to \infty.
\end{equation*} Moreover, $u^\star$ satisfies the interior equation of~\eqref{eq:pde-lemm-abstrait-holder} on $Q^{T-\tau}$. These conclusions hold for every $\tau\in(0,T)$. By a diagonal extraction argument we obtain that $u^\star\in C^{2}_{loc}([0,T)\times\R)$, satisfies the interior equation of~\eqref{eq:pde-lemm-abstrait-holder} on $Q^T$ and up to a subsequence verifies~\eqref{eq:res-limit-lem}.\\
\textit{Step 3:} For $\eta \in \N$ and $(t,x) \in \overline{Q^T}$, we denote by $X_s^{t,x}(\eta)$ a weak solution to the SDE~\eqref{eq:eds-nonlin} associated with the parameters $(a_\eta,b_\eta)$, starting from $x$ at time $t$. The existence is ensured by Theorem~6.1.7 in \cite{stroock-varadhan-book}. Applying the Feynman--Kac formula to $u_\eta$, together with Markov's inequality, we obtain that for any $\eta \in \N$, $x,y\in \R$, and $\delta_1>0$,
\begin{equation*}
\begin{split}
\left| u_\eta(t,y)-g_\eta(x) \right|
&\leq \E\!\left[\left| g_\eta(X_T^{t,y}(\eta)) + \int_t^T f_\eta(s,X_s^{t,y}(\eta))ds -g_\eta(x) \right|\right] \\
&\leq \omega(\delta_1) + \E\!\left[\left| g_\eta(X_T^{t,y}(\eta))-g_\eta(x) \right|
\mathds{1}_{\{| X_T^{t,y}(\eta) -x |> \delta_1\}}\right] + \esperance{\left| \int_t^T  f_\eta(s,X_s^{t,y}(\eta))ds \right|} \\
&\leq \omega(\delta_1) + 2 \frac{\Norm{g}_\infty}{\delta_1}
\E\!\left[| X_T^{t,y}(\eta)-x |\right]  + \esperance{\left| \int_t^T  f_\eta(s,X_s^{t,y}(\eta))ds \right|}.
\end{split}
\end{equation*} By assumption the coefficients of the SDE are uniformly (in $\eta$) bounded by $\mathfrak{v}^{-1}$. This yields
\begin{equation}\label{ineq:Feynman--Kac-unif-cv-2}
\left| u_\eta(t,y)-g_\eta(x) \right|
\leq
\omega(\delta_1) +
2 \frac{\Norm{g}_\infty}{\delta_1 \mathfrak{v}}
\left( (T-t) + \sqrt{T-t} + \mathfrak{v} |x-y| \right) + \mathfrak{v}^{-1}(T-t).
\end{equation} Let $\varepsilon>0$ and $x\in\R$. We choose $\delta>0$ such that for all $(t,y)\in \mathcal{B}((T,x),\delta)\cap \overline{Q}^T$, the right-hand side of~\eqref{ineq:Feynman--Kac-unif-cv-2} is smaller than $\varepsilon/3$. Considering a subsequence $(u_{\eta_n})_n$ converging to $u^\star$ on a compact containing $(t,y)$, we have 
\begin{align*}
\left| u^\star(t,y)-g(x) \right |
\leq
\left| u^\star(t,y)-u_{\eta_n}(t,y) \right |
+ \left| u_{\eta_n}(t,y)-g_{\eta_n}(x) \right|
+ \left| g_{\eta_n}(x)-g(x) \right|,
\end{align*}
we choose $n$ large enough so that the first and the last terms are smaller than $\varepsilon/3$: the first one by the uniform convergence of $u_\eta$ to $u^\star$ on $[0,T-\tfrac{t}{2}]\times\R$, and the last one by the uniform convergence of $(g_\eta)$. Hence, $u^\star$ is a solution to the linear Cauchy problem~\eqref{eq:pde-lemm-abstrait-holder}. By uniqueness of the solution, we conclude that the convergence holds for the whole sequence.

\subsection{Proof of Corollary~\ref{coro:existence-convergence-linearization}}
    The first part is a direct consequence of Proposition~\ref{prop:abstrait-existence-convergence}. Here we consider for each $\eta\in \N$, $\alpha_\eta$ to be the associated Hölder exponent given by Theorem~5.3 in \cite{Kry82} to $\phi_\eta$. Hence, $\phi_\eta$ belongs to $C^{2+\alpha_\eta}(Q^T)$. Following Remark~\ref{rk:lipschitz-dH} and defining coefficients \begin{equation*}\begin{split}
        &a_\eta(t,x) = D_M H(t,x,\dx \phi_\eta(t,x),\demi \dxx \phi_\eta(t,x)),\quad b_\eta(t,x) = D_p H(t,x,\dx \phi_\eta(t,x),\demi \dxx \phi_\eta(t,x)), \quad f_\eta\equiv 0,\\
        &a(t,x) = D_M H(t,x,\dx \phi(t,x),\demi \dxx \phi(t,x)),\quad b(t,x) = D_p H(t,x,\dx \phi(t,x),\demi \dxx \phi(t,x)), \quad f\equiv 0,\\
    \end{split}
    \end{equation*} we have $(a_\eta,b_\eta,f_\eta) \in C^{\alpha_\eta}(Q^T)$ and $g_\eta \in C^{2+\alpha_\eta}(\R)$ as a mollification of $g$. Boundedness of the coefficients is also guaranteed by Assumption~\ref{assmpt:operatorH}. The stability result~\eqref{eq:propstab-1-bis} of Proposition~\ref{prop:stability-nonlin-lip-coeff-bis} gives the required Hölder convergence of the coefficients on the interior domain.\\
    It remains to show uniform convergence. This follows from a Feynman--Kac representation argument. Let $(t,x) \in [0,T]\times \R$, by the Feynman--Kac formula and Lemma~\ref{lem:strong-solution} we have for every $(t,x)\in [0,T]\times \R$, \begin{equation}\label{eq:decomposition-fk-linear}\begin{split}
        \left| \psi(t,x)-\psi_\eta(t,x) \right| \leq& \E\left[ \left| g(X_T^{t,x}) - g_\eta(X_T^{t,x}(\eta)) \right| \right]  \\
\leq& \E\left[ \left| g(X_T^{t,x}(\eta)) - g_\eta(X_T^{t,x}(\eta)) \right| \right] + \E\left[ \left| g(X_T^{t,x}(\eta)) - g(X_T^{t,x}) \right| \right].
    \end{split}\end{equation} The first term in \eqref{eq:decomposition-fk-linear} is controlled by $\Norm{g-g_\eta}_\infty\to 0$ as $\eta \to \infty$. For the second term, let $\delta_2 >0$. By Markov's inequality, we have \begin{equation*}
        \E\left[ \left| g(X_T^{t,x}(\eta)) - g(X_T^{t,x}) \right| \right] \leq \omega(\delta_2) + 2 \frac{\Norm{g}_\infty}{\delta_2} \E\left[ \left| X_T^{t,x}(\eta) - X_T^{t,x} \right| \right].
    \end{equation*} We hence conclude using Lemma~\ref{lem:eds-convg} below, by convergence results of Proposition~\ref{prop:stability-nonlin-lip-coeff-bis} and Corollary~\ref{coro:lip-local-coeff}.\unskip\nobreak\hfill\(\qedsymbol\)

\begin{lemma}\label{lem:eds-convg}
    Let $(t,x)\in \overline{Q}^T$, consider $(X^{t,x}_s)_{s\in[t,T]}$ and $(Y^{t,x}_s)_{s\in[t,T]}$ two strong solutions to the respective SDEs \begin{align*}
        \left\{ \begin{array}{l}
            dX_s^{t,x} = b_1(s,X_s^{t,x})ds + \sigma_1(s,X_s^{t,x})dW_s,\\
            X_t^{t,x} = x,
        \end{array} \right. \qquad \left\{ \begin{array}{l}
            dY_s^{t,x} = b_2(s,Y_s^{t,x})ds + \sigma_2(s,Y_s^{t,x})dW_s,\\
            Y_t^{t,x} = x.
        \end{array} \right.
    \end{align*} Assume: \begin{itemize}
    \item There exists $C_1>0$, such that $\Norm{b_1}_{\infty;\overline{Q}^T},\Norm{\sigma_1}_{\infty;\overline{Q}^T},\Norm{b_2}_{\infty;\overline{Q}^T},\Norm{\sigma_2}_{\infty;\overline{Q}^T} < C_1$,
    \item For every $\tau \in (0,T]$, $b_1$ and $\sigma_1$ are Lipschitz in the spatial variable on $[0,T-\tau]\times \R$ with a modulus $C_2(\tau)$. 
    \end{itemize} Then, for every $\tau\in(0,T]$, there exists $C_3(\tau)>0$ such that, for all $(t,x)\in [0,T]\times\R$, \begin{equation*}\begin{split}
        \E\left[ \left| X_T^{t,x} - Y_T^{t,x} \right| \right] &\leq 4 C_1 \left( \tau + \sqrt{\tau} \right)   \\ 
        &+ 2 \left(T^2 \Norm{b_1-b_2}^2_{\infty;\overline{Q}^{T-\tau}} +  T \Norm{\sigma_1-\sigma_2}^2_{\infty;\overline{Q}^{T-\tau}} \right)^\demi \exp{\left((T+T^2)2C_2(\tau)^2\right)}.
    \end{split} \end{equation*}
\end{lemma}

\begin{proof}[Proof of Lemma~\ref{lem:eds-convg}]
    Fix $\tau\in (0,T]$ and $x\in \R$. Let $t \in [T-\tau,T]$, \begin{equation*} \begin{split}
        \E\left[ \left| X_T^{t,x} - Y_T^{t,x} \right| \right] &\leq \E\left[ \left| \int_t^T \sigma_1(s,X_s^{t,x}) - \sigma_2(s,Y_s^{t,x}) dW_s \right| \right] + \E\left[ \left| \int_t^T b_1(s,X_s^{t,x}) - b_2(s,Y_s^{t,x}) ds \right| \right] \\
        &\leq \left( \E\left[ \int_t^T \left| \sigma_1(s,X_s^{t,x})-\sigma_2(s,Y_s^{t,x}) \right|^2 ds \right] \right)^\demi + \E \left[ \int_t^T \left| b_1(s,X_s^{t,x})-b_2(s,Y_s^{t,x}) \right| ds \right] \\
        &\leq \sqrt{T-t} 2 C_1 + (T-t) 2 C_1
    \end{split}
    \end{equation*} Let $t \in [0,T-\tau]$, for $s\in [t,T-\tau]$ we have \begin{equation*} \begin{split}
        \E\left[ \left| X_s^{t,x} - Y_s^{t,x}  \right|^2 \right] \leq& 4 \E \left[ \int_t^s \left| \sigma_1(r,X_r^{t,x}) -  \sigma_1(r,Y_r^{t,x}) \right|^2 dr \right] + 4 \E \left[ \int_t^s \left| \sigma_1(r,Y_r^{t,x}) -  \sigma_2(r,Y_r^{t,x}) \right|^2 dr \right] \\ 
        +& 4T \E \left[ \int_t^s \left| b_1(r,X_r^{t,x}) -  b_1(r,Y_r^{t,x}) \right|^2 dr \right] + 4T \E \left[ \int_t^s \left| b_1(r,Y_r^{t,x}) -  b_2(r,Y_r^{t,x}) \right|^2 dr \right]
    \end{split}
    \end{equation*} By assumptions we hence obtain that for any $s\in[t,T-\tau]$, \begin{equation*}\begin{split}
        \E\left[ \left| X_s^{t,x} - Y_s^{t,x}  \right|^2 \right]
        &\leq 4 (T+1) C_2(\tau)^2 \E\left[ \int_t^s \left| X_r^{t,x}-Y_r^{t,x} \right|^2 dr \right] \\
        &\quad+ 4T \left( \Norm{\sigma_1 - \sigma_2}^2_{\infty;\overline{Q}^{T-\tau}}  +  \Norm{b_1 - b_2}^2_{\infty;\overline{Q}^{T-\tau}} T \right) .
    \end{split}\end{equation*} Applying Grönwall's inequality we conclude that \begin{equation*}
         \E\left[ \left| X_{T-\tau}^{t,x} - Y_{T-\tau}^{t,x}  \right|^2 \right] \leq 4 \left(T^2 \Norm{b_1-b_2}^2_{\infty;\overline{Q}^{T-\tau}} +  T \Norm{\sigma_1-\sigma_2}^2_{\infty;\overline{Q}^{T-\tau}} \right) \exp{\left(4(T+T^2)C_2(\tau)^2\right)}.
    \end{equation*} Remarking that for $t\in[0,T-\tau]$, \begin{equation*}
        \E\left[ \left| X_T^{t,x} - Y_T^{t,x} \right| \right] \leq \E\left[ \left| X_{T-\tau}^{t,x} - Y_{T-\tau}^{t,x} \right| \right] + \E\left[ \left| X_{T}^{t,x} - X_{T-\tau}^{t,x} \right| \right]  + \E\left[ \left| Y_{T}^{t,x} - Y_{T-\tau}^{t,x} \right| \right]
    \end{equation*} The last term is bounded by $C_1(\tau+\sqrt{\tau})$ using the same argument as at the beginning.
\end{proof}

\section{Auxiliary results in convex and functional analysis}

We recall the following definition from Chapter~26 in \cite{rockafellar-1970a}.
\begin{definition}[Essentially smooth function]\label{def:ess-smooth}
    A proper convex function $f:\R^d\to\R\cup\{+\infty\}$ is said to be essentially smooth if,  letting $C:=\mathrm{int}(\mathrm{dom}(f))$, the following conditions hold: \begin{enumerate}
        \item $C\neq\emptyset$,
        \item $f$ is differentiable on $C$,
        \item $\abs{\nabla f(x)}\to+\infty$ as $x\to\partial C$.
    \end{enumerate}
\end{definition}

\begin{lemma}\label{lem:strict-convexity-onedimen}
    Consider $\Phi: \R^M \to \R$ a continuously differentiable function and $K \subset \R^M$ compact convex. For $(\lambda,\mu) \in \R^M\times \R^M$ we denote $\phi^{\lambda,\mu} : \R \ni t \mapsto \Phi(\lambda + t\mu)\in \R.$ Let $\alpha >0$. $\Phi$ is strongly convex of modulus $\alpha$ on $K$ iff $\forall\lambda\in \R^M,\forall\mu\in K\ominus \{\lambda\}  $ (Minkowski difference set), $\phi^{\lambda,\mu}$ is strongly convex of modulus $\alpha ||\mu||^2$ on $[0,1]$.
\end{lemma}

\begin{proof}%
    The "$\implies$" part is straightforward. Let $(\lambda_1,\lambda_2)\in K^2,$ we consider $\phi^{\lambda_1,\lambda_2 - \lambda_1}$. For all $k \in (0,1),$ we have \begin{align*}
        &\phi^{\lambda_1,\lambda_2 - \lambda_1}(kt+(1-k)s) \leq k \phi(t)+(1-k)\phi(s) - \frac{\alpha}{2} k(1-k)|t-s|^2 ||\lambda_2-\lambda_1||^2 \quad \forall(t,s)\in[0,1]^2, \\
        &\text{taking t=0, s=1 }\implies \Phi(k\lambda_1 +(1-k)\lambda_2) \leq k\Phi(\lambda_1)+(1-k)\Phi(\lambda_2)-k(1-k)\frac{\alpha}{2}||\lambda_2-\lambda_1||^2. \qedhere
    \end{align*} 
\end{proof} 

\begin{lemma}\label{lem:grad-zero}
    Let $\phi\in C^1(\R^n)$ with $ \Norm{\phi}_\infty < \infty$. Then \begin{equation*}
      \inf_{\R^n} \left| \nabla_x \phi \right| = 0.
    \end{equation*}
\end{lemma}

\begin{proof}
Assume by contradiction that there exists $\veps>0$ such that
$\left| \nabla_x\phi(x)\right|>\veps$ for all $x\in\R^n$. Define  $$f(x)=\frac{\Dx \phi(x)}{\left|\Dx \phi(x) \right|},$$ since $\nabla \phi$ is continuous and bounded from below by $\veps$, the vector field $f$ is continuous and bounded. Let $x_0 \in \R^n$. Consider $\gamma$ a solution to $$\Dot{\gamma}(t)=f(\gamma(t)), \qquad \gamma(0)=x_0.$$ By Peano's theorem such a solution exists locally, and since $|f|\le 1$ it cannot blow up in finite time, hence it extends to $[0,\infty)$. Along this curve we have \begin{equation*}
    \partial_t \phi(\gamma(t)) = \left| \nabla_x \phi(\gamma(t)) \right| > \veps.
\end{equation*} Thus by integration, $\phi(\gamma(t)) > \phi(x_0) + \veps t$, which contradicts the boundedness of $\phi$.
\end{proof}

\begin{lemma}\label{lem:compact_convex_set}
    Let $H$ be the space of nonempty compact sets of $\R^d$ equipped with the Hausdorff distance $d_H.$ Suppose $E:\R^M\to H$ is a continuous function. Then, for a nonempty compact subset $K\Subset\R^M$, \begin{align*}
        \text{conv}\left(\bigcup_{k \in K}E(k) \right)
    \end{align*} is a compact convex nonempty set of $\R^d.$
\end{lemma}

\begin{proof}
    We first show that \begin{align*}
        A:=\bigcup_{k \in K}E(k)
    \end{align*} is a compact set. Notice that by continuity, $E(K)$ is a compact set in $(H,d_H)$. This implies the boundedness of $A$. Let $(x_n)_{n\in\N} \in A,$ such that there exists $x^*\in \R$, with $x_n\xrightarrow[]{}x^*$ as $n\to \infty$. Hence for any $ n \in \N$, there exists $ k_n\in K$, such that $ x_n \in E(k_n):=E_n$. By compactness of $E(K)$, we can extract a converging subsequence of $(E_n)_n$ i.e. $\exists (n_m)_{m\in\N},\exists k^*\in K,$ $d_H(E_{n_m},E^*)\xrightarrow[]{}0$ as $m\to\infty$ with $E^*=E(k^*)$. Remarking that for any $m\in \N$, \begin{equation*}
        d(x^*,E^*) \leq d(x^*,x_{n_m})+d_H(E_{n_m},E^*)
    \end{equation*} And by closedness of $E^*$ we conclude that $x^*\in E^*$. Thus, $x^* \in A$. A corollary of Carathéodory's theorem (see Theorem~17.2 in \cite{rockafellar-1970a}) gives in finite dimension that the convex hull of a compact set is compact.
\end{proof}

\addcontentsline{toc}{section}{References}
\bibliographystyle{alpha}
\bibliography{biblio}

@article{Kry82,
  author    = {Krylov, N. V.},
  title     = {Boundedly nonhomogeneous elliptic and parabolic equations},
  journal   = {Mathematics of the USSR-Izvestiya},
  year      = {1983},
  volume    = {20},
  number    = {3},
  pages     = {459--492},
  doi       = {10.1070/IM1983v020n03ABEH001360},
  url       = {http://mi.mathnet.ru/eng/im1637},
}

@book{Kry96_holder,
author = {Krylov,Nikolai Vladimirovich},
address = {Providence, R.I},
booktitle = {Lectures on elliptic and parabolic equations in Hölder spaces},
isbn = {0-8218-0569-X},
language = {eng},
publisher = {American Mathematical Society},
series = {Graduate studies in mathematics},
title = {Lectures on elliptic and parabolic equations in Hölder spaces / N. V. Krylov},
year = {1996},
}

@book {lieberman,
    AUTHOR = {Lieberman, Gary M.},
     TITLE = {Second order parabolic differential equations},
 PUBLISHER = {World Scientific Publishing Co., Inc., River Edge, NJ},
      YEAR = {1996},
     PAGES = {xii+439},
      ISBN = {981-02-2883-X},
   MRCLASS = {35-02 (35Bxx 35Dxx 35Kxx)},
  MRNUMBER = {1465184},
MRREVIEWER = {Siegfried\ Carl},
       DOI = {10.1142/3302},
       URL = {https://doi.org/10.1142/3302},
}

@book{Zeidler_1986, place={New York [etc.]}, title={Nonlinear functional analysis and its Applications I. Fixed-Point Theorems}, ISBN={0-387-90914-1}, publisher={Springer}, author={Zeidler, Eberhard}, year={1986}, pages={XXIII, 909} }

@book{Le_Gall2012,
  title     = "Mouvement Brownien, Martingales Et Calcul Stochastique",
  author    = "Le Gall, Jean-Francois",
  publisher = "Springer",
  series    = "Mathematiques \& Applications (Berlin)",
  edition   =  2013,
  year      =  2012,
  address   = "Berlin, Germany",
  language  = "fr"
}

@book{rockafellar-1970a,
  address = {Princeton, N. J.},
  author = {Rockafellar, R. Tyrrell},
  callnumber = {UniM Maths 516.08 R59},
  description = {robotica-bib},
  notes = {A SRL reference.},
  publisher = {Princeton University Press},
  series = {Princeton Mathematical Series},
  title = {Convex analysis},
  year = 1970
}

@article{saut-1987-ucp,
title = {Unique continuation for some evolution equations},
journal = {Journal of Differential Equations},
volume = {66},
number = {1},
pages = {118-139},
year = {1987},
issn = {0022-0396},
doi = {https://doi.org/10.1016/0022-0396(87)90043-X},
url = {https://www.sciencedirect.com/science/article/pii/002203968790043X},
author = {Jean-Claude Saut and Bruno Scheurer}
}

@inproceedings {stroock-varadhan-support,
    AUTHOR = {Stroock, Daniel W. and Varadhan, S. R. S.},
     TITLE = {On the support of diffusion processes with applications to the
              strong maximum principle},
 BOOKTITLE = {Proceedings of the {S}ixth {B}erkeley {S}ymposium on
              {M}athematical {S}tatistics and {P}robability ({U}niv.
              {C}alifornia, {B}erkeley, {C}alif., 1970/1971), {V}ol. {III}:
              {P}robability theory},
     PAGES = {333--359},
 PUBLISHER = {Univ. California Press, Berkeley, CA},
      YEAR = {1972},
   MRCLASS = {60J60},
  MRNUMBER = {400425},
}

@misc{SurveyOT,
      title={Optimal transport for model calibration}, 
      author={Ivan Guo and Gregoire Loeper and Jan Obloj and Shiyi Wang},
      year={2021},
      eprint={2107.01978},
      archivePrefix={arXiv},
      primaryClass={q-fin.MF},
      url={https://arxiv.org/abs/2107.01978}, 
}

@article{guo2022lsv-calibration,
  title={Calibration of local-stochastic volatility models by optimal transport},
  author={Guo, Ivan and Loeper, Gr{\'e}goire and Wang, Shiyi},
  journal={Mathematical Finance},
  volume={32},
  doi={10.1111/mafi.12335},
  number={1},
  pages={46--77},
  year={2022},
  publisher={Wiley Online Library}
}

@book {brezis,
    AUTHOR = {Brezis, Ha\"im},
     TITLE = {Analyse fonctionnelle},
    SERIES = {Collection Math\'ematiques Appliqu\'ees pour la Ma\^itrise.
              [Collection of Applied Mathematics for the Master's Degree]},
      NOTE = {Th\'eorie et applications. [Theory and applications]},
 PUBLISHER = {Masson, Paris},
      YEAR = {1983},
     PAGES = {xiv+234},
      ISBN = {2-225-77198-7},
   MRCLASS = {46-01 (47-01)},
  MRNUMBER = {697382},
}

@book {stroock-varadhan-book,
    AUTHOR = {Stroock, Daniel W. and Varadhan, S. R. Srinivasa},
     TITLE = {Multidimensional diffusion processes},
    SERIES = {Classics in Mathematics},
      NOTE = {Reprint of the 1997 edition},
 PUBLISHER = {Springer-Verlag, Berlin},
      YEAR = {2006},
     PAGES = {xii+338},
      ISBN = {978-3-540-28998-2; 3-540-28998-4},
   MRCLASS = {60-02 (60H05 60J25 60J60 60J65)},
  MRNUMBER = {2190038},
}

\end{document}